\documentclass[10pt,reqno]{amsart}

\usepackage[utf8]{inputenc}

\usepackage[T1]{fontenc}

\usepackage[a4paper, margin=2.7cm]{geometry}

\usepackage{url}

\usepackage{enumitem}

\usepackage{amsmath, amsthm, amsfonts, amssymb}
\usepackage{mathrsfs}           % \mathscr font.

\usepackage{hyperref}
\hypersetup{colorlinks=true,linkcolor=blue,citecolor=blue,urlcolor=blue}

\usepackage{bbm}

\usepackage{url}
\numberwithin{equation}{section}

\def\N{\mathbb{N}}
\def\R{\mathbb{R}}

\def\bbP{\mathbb{P}}
\def\bbE{\mathbb{E}}
\def\ee{\mathrm{e}}

\DeclareMathOperator{\sgn}{sgn}

\newtheorem{theorem}{Theorem}[section]
\newtheorem{cor}[theorem]{Corollary}
\newtheorem{lemm}[theorem]{Lemma}
\newtheorem{prop}[theorem]{Proposition}

\theoremstyle{definition}
\newtheorem{definition}[theorem]{Definition}
\theoremstyle{remark}
\newtheorem{rema}[theorem]{Remark}
\title{Strong Existence and Uniqueness for Singular SDEs Driven by Stable Processes}

\author{Leonid Mytnik}
\address{The Faculty of Data and Decision Sciences, 
	Technion – Israel Institute
	of Technology, Haifa 3200003, Israel
}
\email{leonid@ie.technion.ac.il}

\author{Johanna Weinberger}
\address{The Faculty of Data and Decision Sciences, 
	Technion – Israel Institute
	of Technology, Haifa 3200003, Israel}
\email{wjohanna@campus.technion.ac.il}

\date{\today}

\begin{document}
\begin{abstract}
	We consider the one-dimensional stochastic differential equation
	\begin{equation*}
		X_t = x_0 + L_t + \int_0^t \mu(X_s)ds, \quad t \geq 0,
	\end{equation*} 
	where $\mu$ is a finite measure of Kato class $K_{\eta}$ with $\eta \in (0,\alpha-1]$ and $(L_t)_{t \geq 0}$ is a symmetric $\alpha$-stable
	process with $\alpha \in (1,2)$.  We derive weak and strong well posedness for this equation when $\eta \leq\alpha-1$ and $\eta < \alpha-1$, respectively and show that the condition $\eta \leq \alpha-1$ is sharp for weak existence. We furthermore reformulate the equation in terms of 
	the local time of the solution $(X_{t})_{t \geq 0}$ and prove its well posedness. To this end, we also derive a Tanaka-type formula for a  symmetric, $\alpha$-stable processes with $\alpha \in (1,2)$ that is perturbed by an adapted, right-continuous process of finite variation.  
\end{abstract}
\maketitle

\setcounter{tocdepth}{1}
\tableofcontents

\section{Introduction}
In this paper, we study solutions of the one-dimensional stochastic differential equation (SDE)
\begin{equation}\label{eq:singularSDE}
	X_t = x_0 + L_t + \int_0^t \mu(X_s)ds, \quad t \geq 0.
\end{equation} 
Here, $(L_t)_{t \geq 0}$ is a one-dimensional symmetric $\alpha$-stable L\'evy process, $\alpha \in (1,2)$ and $\mu$
is a measure of Kato class (see Definition~\ref{def:Kato}). 
Since $\mu$ is not necessarily absolutely continuous with respect to the Lebesgue measure, one needs to give a
rigorous definition of the above equation and in particular of the term $\int_0^t \mu(X_s)ds$. 
Here, we shall use the definition from \cite{Kim14}, which relies on an approximation scheme. 
Due to the non-regular nature of $\mu$  the equation \eqref{eq:singularSDE} is often called a
 singular SDE.  In the last decade, singular SDEs and SDEs driven by
 stable processes generated interest among many researchers, because they arise in many
 models in applications. For instance, there are many models in physical sciences that can not be captured by Brownian noise but by $\alpha$-stable noise (see \cite{Woyczynski01}).  
 The setting of equation \eqref{eq:singularSDE} captures, for instance, the movement of symmetric $\alpha$-stable processes with $\alpha \in (1,2)$ that experience an upward drift in certain fractal like sets. 

A natural first question regarding \eqref{eq:singularSDE} is the existence and uniqueness of solutions.  Weak existence and uniqueness for \eqref{eq:singularSDE} when $\mu \in K_{\alpha-1}$ was already established 
in \cite{Kim14} in dimension $d \geq 2$. We closely follow their arguments to extend the result to $d = 1$ when $\mu\in K_{\alpha-1}$ is a finite measure. Under slightly stricter conditions on $\mu$ we then prove pathwise uniqueness by combining the techniques from \cite{Athreya20} and \cite{Xiong19}. By a Yamada-Watanabe argument this is sufficient to prove strong existence and uniqueness. \par 
Besides well posedness, we are also interested in reformulating  the drift term $\int_0^t\mu(X_s)ds$ in terms of the 
local time of the solution $X$ and prove that this reformulation is equivalent to the original equation~\eqref{eq:singularSDE}
under certain conditions on $\mu$.
In the case where $(L_t)_{t \geq 0}$ is replaced by a standard Brownian motion $(B_t)_{t \geq 0}$ this was 
already done in \cite{LeGall84}, where the author, among other things, studied equations of the form
\begin{equation}\label{eq:brown_sde}
	X_t = x_0 +B_t + \int_\R \ell_t^x\mu(dx), \quad t\geq 0.
\end{equation}
Hereby,  the local time of the unknown process $X$ is given by Tanaka's formula
\begin{equation}
	|X_t-x| = |x_0-x| + \int_0^t \sgn(X_s-x)dX_s +  \ell_t^x, \quad t \geq 0, x\in \R.
\end{equation}
In this case we denote the local time by $\ell$, since it not only satisfies Tanaka's formula but is also the occupation density of $X$. 
In this paper, we define a local time $\gamma$ for solutions to \eqref{eq:singularSDE} in a similar spirit via a Tanaka-type formula
 and reformulate \eqref{eq:singularSDE} 
in terms of $\gamma$. That is, we derive a Tanaka-type formula
 for semimartingales of the form 
\begin{equation}\label{eq:semiintro}
	X_t =x_0+L_t+A_t, \quad t \geq 0,
\end{equation}
define a local time $\gamma$ from this Tanaka formula
and show that \eqref{eq:singularSDE} is equivalent to  
\begin{equation}\label{eq:SDE}
	X_t = x_0 +L_t + \int_\R \gamma_t^x \mu(dx), \quad t \geq 0.
\end{equation}
  As it is not clear a-priori that  $(\gamma_t^x)_{t \geq 0}^{x\in \R}$, defined via Tanaka's formula, coincides with the occupation density $(\ell_t^x)_{t \geq 0}^{x\in \R}$, we use the different notation $(\gamma_t^x)_{t \geq 0}^{x\in \R}$.
However, we show that the local time $(\gamma_t^x)_{t\geq 0}^{x\in \R}$
and the occupation density $(\ell_t^x)_{t \geq 0}^{x\in \R}$ coincide outside of sets of measure zero with respect to the measure $\mu$ and the Lebesgue measure for solutions of \eqref{eq:SDE} under
 some additional conditions (see Theorem~\ref{th:sdeilt}). To this end, we also show the joint continuity of $(\ell_t^x)_ {t \geq 0}^{x \in \R}$ of $X$ in this setting. 
 Finally, we can apply the reformulation in terms of $\gamma$ to prove that the weak existence result
 from Theorem~\ref{th:sde_exist} is sharp. \par 
 As mentioned above, proving existence and uniqueness of solutions to \eqref{eq:singularSDE} is one of the main goals of this paper. For this reason we want to give a short review of existing literature on this topic. The
 case of Brownian forcing was studied extensively. Hereby, proving strong existence and uniqueness
 for solutions to stochastic differential equations (SDEs) often involves the application of the Zvonkin
 transformation. The transformed equation admits a drift of higher regularity, enabling the utilization
 of classical tools. Initially, it was devised in \cite{Zvonkin75} for one-dimensional SDEs driven by Brownian
 motion with a bounded and measurable drift. This method later saw extension to the multidimensional
 setting \cite{Veretennikov81} and adaptation for locally unbounded drifts in \cite{Gyongy01,Krylov05}. Furthermore, weak existence and
 uniqueness was proven for drifts in a certain Kato class in \cite{Bass03} for $d \geq 2$ and for $d = 1$ strong existence
 and uniqueness was established for drifts which are distributional derivatives of of Hölder continuous
 functions with exponent $\gamma > 1/2$ was proven in \cite{Bass01}.\par
 In the setting of $\alpha$-stable forcing, we first want to mention
 \cite{Tanaka74, Priola12, Chen12}, where the authors considered the case where the drift is a function in a Hölder space or in a Kato class depending on the stability parameter $\alpha$. 
 In \cite{Priola12} strong existence and pathwise uniqueness were proven when the drift is of Hölder regularity $\beta > 1-\alpha/2$ when $d\geq 1$ and $\alpha \in [1,2)]$.
 In the setting of drifts of Kato class the authors of \cite{Chen12}
 showed weak existence and uniqueness of a solution to \eqref{eq:singularSDE} if the drift is a function in the Kato class $K_{\alpha-1}$ for $d \geq 1$ and $\alpha \in (1,2)$. \par
 While the above results already demonstrate that the class of possible drift terms is much larger
 compared to the deterministic setting, it turns out that it is even possible to consider generalized 
 functions and measures. 
As already mentioned above, the authors of \cite{Kim14} were able to prove weak existence and uniqueness in the case where $\mu$ is a measure in $K_{\alpha-1}$ for $d \geq 2$ and $\alpha \in (1,2)$.
 The case of generalized functions as drift was for instance investigated in \cite{Athreya20} , where strong existence
 and uniqueness were shown for $\mu \in B^\beta_{\infty,\infty}$ with $\beta > \frac{1-\alpha}{2}$ in dimension 
 $d = 1$ and for $\alpha \in (1,2)$. Hereby $B^{p,q}_s$, $p,q\in [1,\infty]$, $s\in \R$ denote the non-homogeneous Besov 
 spaces (see Definition~\ref{def:Besov}). We furthermore want to mention \cite{DeRaynal22}, 
 where time-dependent drifts in $B^\beta_{p,q}$ for general $p,q$ were considered.
 In \cite{Kremp22, Kremp22PhD} the authors considered drifts in $B^\beta_{\infty,\infty}$
 for $\beta>\frac{2-2\alpha}{3}$ and used the theory of paracontrolled distributions to
 show existence and uniqueness of solutions to the associated martingale problem and 
 weak rough path type existence and uniqueness. In order to use the theory of paracontrolled distributions 
 it is necessary to assume the drifts are enhanced drifts, meaning that they satisfy some additional assumptions  \par  
As outlined above, our aim is to derive a Tanaka-type formula for semimartingales $X$ of the from \eqref{eq:semiintro} in order to define a local time $\gamma$. In the case, where $X$ 
is a symmetic L\'evy process this was achieved in \cite{Salminen07}. In 
\cite{Tsukada18}  and \cite{Tsukada19} the author extended the result from \cite{Salminen07} to the nonsymmetric case and in \cite{Engelbert19} the authors considered all parameters $\alpha \in (0,2)$.\\
We also want to mention, that SDEs involving the local time of the unknown process with Brownian forcing were already extensively studied in the literature. 
This kind of equation was first considered in \cite{Strook81} in their discussion about pure martingales. 
Later, in \cite{LeGall84}, \cite{Engelbert85}, \cite{Blei13} and \cite{Bass05} the authors considered the case where $\mu$ is a finite measure 
such that $\mu(\{x\}) \leq 1$, for all $x$. In \cite{Groh86} the author considered the case where 
$\mu(dx)$ is the Lebesgue Stieltjes measure of a function $f$ of locally bounded variation satisfying some additional assumptions.  
We also want to mention \cite{blei2012note}, where the authors considered the case where
$|\mu({x})|>1$ for some $x \in \R$.
Furthermore, in \cite{Etore18} the time inhomogeneous case was discussed. \par 

%For the interval $\alpha \in (0,1)$ the authors of \cite{Tanaka74}
%provided examples of drift terms with $\alpha + \beta < 1$ such that uniqueness does not hold for  \eqref{eq:SDE}.  We call the condition $\alpha + \beta > 1$, where $\beta$ denotes the modulus of continuity the \emph{balance condition}. We conjecture that the 
%balance condition also constitutes the threshold  for well-posedness for $\alpha\in (1,2)$
%if we interpret the Kato class as our modulus of continuity. 
\subsection{Outline}
In Section~\ref{sec:Levy}, we first review some basic facts about $\alpha$-stable processes. 
Next, we discuss different notions of local times that are used throughout this paper in Section~\ref{sec:loc_time_intro}  and introduce the notion
  of capacities in and their relation to Kato classes in Section~\ref{sec:capacities}. Then, in Section~\ref{sec:defKato}, we review the definition of Besov spaces and their relationship with Kato class. 
  In Section~\ref{sec:mainresults}, we state the main results of the paper. This in particular includes strong existence and uniqueness of solutions to \eqref{eq:singularSDE}. Furthermore, we discuss the reformulation of 
  the singular SDE \eqref{eq:singularSDE} so that the drift is defined in terms of the local time $\gamma$, which in turn is defined via a Tanaka type formula. 
  This leads to \eqref{eq:SDE} and we show that solutions to \eqref{eq:singularSDE} and \eqref{eq:SDE} coincide under certain conditions on $\mu$. 
  We also show that   $\gamma$ and the occupation density $\ell$ of solutions to \eqref{eq:SDE} coincide
  under certain assumptions on the drift $\mu$. 
  The subsequent Sections ~\ref{sec:sde_exist}-\ref{sec:sdeilt} are devoted to the proofs of the main results. 
  Finally, in Section~\ref{sec:sharp} we use our definition of the local time $\gamma$ to demonstrate that our existence result for equation \eqref{eq:singularSDE} is sharp.
  \subsection*{Convention for Constants}
In this article, constants are denoted by $C$ and $c$, with their precise values considered irrelevant. We reserve the freedom to adjust these values without altering the notation for readability. 
Other constants are denoted by $c_1, c_2, \ldots$ and $C_1, C_2, \ldots$. When necessary, dependence on additional parameters is expressed as $C = C(\lambda)$.
  \subsection*{Acknowledgments}
 We want to thank Lukas Anzeletti,  Oleg Butkovsky, Panki Kim, and Renming Song for very helpful conversations.
 The work of the authors was supported in part by the ISF grant
 No. 1985/22. 
\section{Preliminaries}\label{sec:Preiminaries}
\subsection{Notation}
We use the following notation for common function spaces. 
\begin{itemize}
	\item $\mathcal{C}^k$ is the space of $k$ times continuously differentiable, real-valued functions on $\R$. When $k = 0$, we write $\mathcal{C}^k= \mathcal{C}$.
	\item $\mathcal{C}_c^k$ is the space of $k$ times continuously differentiable, compactly supported, real-valued, functions on $\R$.  When $k = 0$ we write $\mathcal{C}^k_c = \mathcal{C}_c$.
	\item $\mathcal{S}$ denotes the Schwarz space. That is, $\mathcal{S}$ is the space of all infinitely differentiable  functions $f$  on $\R$ such that 
	$\sup_{0 <\alpha<k, x\in\R }(1+|x|)^k|\partial^\alpha f| < \infty$, for all $k \in \mathbb{N}$.
	\item $\mathcal{S}'$ is the space of tempered distributions and is defined as the topological dual of $\mathcal{S}$.
	\item $\mathcal{B}$ denotes the Borel sigma algebra over $\R$.
	\item $L^p(A,m)$ for $p \in [1,\infty]$ denotes the usual $L^p$-space on the measure space $(A, \mathcal{A},m)$.
%	\item Let $(\Omega, \mathcal{F}, \bbP)$ be a complete probability space and $p \in [1,\infty]$.  We denote by 
%	$L^p(\bbP)$ for the space of all measurable maps $X \colon \Omega \rightarrow \R$
%	with $p$-th moments, endowed with the norm $\|X\|_{L^p(\bbP)} = \big(\mathbb{E}[|X|^p]\big)^{1/p}$, with the usual adaptations for $p = \infty$.
%	\item $\mathcal{D}$ is the space of all c\`adl\`ag paths from $[0,\infty)$ to $\R$, equipped with
%	the Skorohod topology. The space $\mathcal{D}_x$ is the subspace of paths starting at $x$.
\end{itemize}
In this article, we frequently consider real-valued right continuous functions of finite variation. If $f: I \rightarrow \R$ with $I \subseteq \R$ measurable, is such a function we shall
denote the Lebesgue-Stieltjes measure associated with $f$ by $\nu_f$. By the Hahn-Jordan decomposition theorem, 
every signed Radon measure $\nu$ can be split into $\nu = \nu^+-\nu^-$, where both $\nu^+$ and $\nu^-$ are non-negative measures. 
The total variation measure $|\nu|$ of $\nu$ is then defined as $|\nu|:= \nu^++\nu^-$.
\subsection{Symmetric $\alpha$-stable L\'evy Processes}\label{sec:Levy}
In this section, we recall several well-known properties of L\'evy processes and in particular of symmetric $\alpha$-stable processes. \par 
By the L\'evy decomposition theorem \cite[Chapter 4, Theorem 42 ]{Protter04}, every L\'evy process $(L_t)_{t \geq 0}$ admits the following decomposition
\begin{equation}
	L_t = \sigma B_t + bt + \int_{|r|>1}r\Pi_t(dr) + \int_{|r| \leq 1}r\big( \Pi_t(dr)-t\tilde{\pi}(dr)\big), \quad t \geq 0.
\end{equation}
Hereby, $b , \sigma\in \R $ and $B$ is a standard Brownian motion, independent of the Poission random measure $N$, which is given by
\begin{equation*}
	\Pi([0,T]\times G) = \sum_{0 \leq s  \leq t }\mathbbm{1}_{G}(\Delta L_s), \quad t \geq 0, G \in \mathcal{B}(G) ,
\end{equation*}
with intensity 
\begin{equation}\label{eq:intensity}
\pi(ds,dr) = dt\tilde{\pi}(dr),
\end{equation}
 where $\tilde{\pi}$ is  a $\sigma$-finite measure such that 
\begin{equation*}
	\int_\R (1 \wedge |r|^2) \upsilon(dr) < \infty.
\end{equation*}
In the case of the symmetric $\alpha$-stable process, 
\begin{equation*}
	b = \sigma = 0
\end{equation*}
and 
\begin{equation}\label{eq:poissoon_intense}
	\tilde{\pi}(dx) = c_1(\alpha) |x|^{-1-\alpha}dx,
\end{equation}
where 
\begin{equation}\label{eq:c_1}
	c_1(\alpha)= \alpha/(2\Gamma(1-\alpha)\cos(\alpha \pi/2)),
\end{equation}
and $\Gamma$ denotes the Gamma function.
This means, that the symmetric $\alpha$-stable process is a pure jump process and since for $\alpha \in (1,2)$
\begin{equation*}
	\int_\R (1\wedge|r|) \tilde{\pi}(dr) = \infty,
\end{equation*}
it is of unbounded variation by \cite[Lemma 2.12]{Kyprianou14}. In this article, we only consider the case, where the symmetric $\alpha$-stable process $(L_t)_{t \geq 0}$ has 
unbounded variation, that is we consider the case $\alpha \in (1,2)$.\par 
The generator of a symmetric $\alpha$-stable process is the fractional Laplacian, which is defined as 
\begin{equation}
	\Delta_{\alpha}f(x) = -\int_\R \hat{f}(\xi) \Psi(\xi) \ee^{ix\xi} d\xi = \int_\R \big(f(x+y)-f(x)- \frac{\partial}{\partial x}f(x)y\mathbbm{1}_{(0,1)} \big)\tilde{\pi}(dy),
\end{equation}
for all $f \in \mathcal{C}_c^\infty$ and $x \in \R$. 
Hereby, $\Psi(\xi) = |\xi|^\alpha$ denotes the L\'evy exponent of $L$, which satisfies
\begin{equation}\label{eq:characteristic}
	\bbE[\ee^{i\xi L_t}] =\ee^{-t \Psi(\xi)}= \ee^{-t|\xi|^\alpha}.
\end{equation}

\subsection{Notions of Local Times}\label{sec:loc_time_intro}
In the literature appear several notions of local times which do not always coincide. Below we give a short overview of the notions we use in this article. 
\subsubsection{Occupation Density}\label{sec:occupation}
First, we define the occupation measure of $(X_t)_{t \geq 0}$ with respect to the Lebesgue measure $m$. That is, for $B \in \mathcal{B}(\R)$
we define
\begin{equation}
	\eta_t^m(B,\omega) = m(\{s  \colon X_s(\omega) \in B, 0 \leq s \leq t\}).
\end{equation} 
The occupation density $(\ell_t^x)_{t \geq 0}^{x\in \R}$ is then defined as the density of the occupation measure with respect to the Lebesgue measure.  We sometimes write $\ell_t^x(X)$, for $t \geq 0$, $x\in \R$ to indicate that $\ell$ is the occupation density of the process $X$. 
In particular, the occupation density satisfies  
\begin{equation}
	\int_0^t f(X_s) ds = \int_\R f(x) \ell_t^x dx, \quad t \geq 0, \text{ a.s.},
\end{equation}
for all bounded  Borel measurable functions $f$.
Classically, one considers the occupation density with respect to the Lebesgue measure.  However, it is possible to  generalize this 
notion to densities with respect to different measures.

\subsubsection{Tanaka Formula}
Another way to defining local times is based on stochastic calculus.  Here we focus on Tanaka's formula. In the case of a continuous semimartingale $(X_t)_{t \geq 0}$ of the form
\begin{equation*}
	X_t = B_t+A_t, \quad t \geq 0, 
\end{equation*}
where $B$ is a standard Brownian motion and $A$ is an adapted right continuous process of finite variation  it reads as 
\begin{equation}\label{eq:tanaka_bm}
	|X_t-x| = |X_0-x| + \int_0^t \sgn(X_s-x)dX_s + \ell_t^x, \quad x \in \R, \quad t \geq 0.
\end{equation}
It turns out that $(\ell_t^x)_{t \geq 0}^{x\in\R}$ is also the occupation density of $X$. 
In the case when the martingale part of $X$ is an $\alpha$-stable process a new Tanaka formel needs to be derived. 
When $X_t = L_t$, $t  \geq  0$ is an symmetric $\alpha$-stable process with $\alpha \in (1,2)$ this was achieved by  \cite{Salminen07}. They derived the formula
\begin{equation*}
	v(L_t-x) = v(x_0-x) + N_t^x + \ell_t^x, \quad \forall x\in \R, \quad t \geq 0
	\end{equation*}
	where $N_t^a$ is a martingale, $v(x) = c(\alpha)|x|^{\alpha-1}$ for all $x \in \R$ and some constant $c(\alpha)$. Again,  $(\ell_t^x)_{t \geq 0}^{x \in \R}$ is the occupation density of $L$. 
	One of our main tasks  is to generalize this formula to the case $X  = L+A$, where $A$ is a right continuous, adapted process of finite variation.

\subsection{Capacities}\label{sec:capacities}
In this section, we review some results about capacities. In the context of Markov processes, capacities arise naturally through their associated Dirichlet forms.
Symmetric Dirichlet forms on $L^2(\R,m)$ also induce a norm on their domain
$D \subseteq L^2(\R,m)$. 
In some cases, these normed spaces coincide with function spaces like the Sobolev spaces. Although we expect 
some of
the solutions of \eqref{eq:singularSDE} to be Markov processes, they will certainly not be symmetric. However, it turns out that it is enough for our analysis to consider the symmetric part of the
corresponding Dirichlet form. For Markov solutions of \eqref{eq:singularSDE} the symmetric part
coincides with the fractional Laplacian and the induced norm is equivalent to the norm of
$B_{2,2}^{\alpha}$. For this reason, we define capacities in terms of this norm. To this end, we orient ourselves on \cite{Adams95}\\
\begin{definition}[Bessel Potential Space \cite{Adams95}]
	Let  $s >0$ and $G_s$ be the Bessel potential so that
	\begin{equation*}
		(1-\Delta)^{-s/2}u =: G_s \ast u \quad \forall u \in L^2(\R).
	\end{equation*}
	Then we define the Bessel potential space as
	\begin{equation*}
		\mathcal{L}_{s,2} = (1-\Delta)^{-s/2}(L^2(\R)),
	\end{equation*}
	endowed with the norm 
	\begin{equation*}
		\|f\|_{\mathcal{L}_{s,2}}  := \|g\|_{L^2}, \text{ where } f = G_s \ast g.
	\end{equation*}
\end{definition}
\begin{prop}
	Let $s >0$. Then $\mathcal{L}_{s,2} = B_{2,2}^s$ and the norm 
	\begin{equation}\label{eq:Lip_norm}
		\|f\|_{\Lambda_{2,2}^s} ^2:= \|f\|_{L^2} ^2+ \int_\R \int_\R\frac{|f(x)-f(y)|^2}{|x-y|^{1+2s}}dxds,
	\end{equation}
	is equivalent to the norms of $\mathcal{L}_{s,2}$ and $B_{2,2}^s$.
\end{prop}
\begin{proof}
	The equivalence of the norms of $\mathcal{L}_{s,2}$ and $B_{2,2}^s$ can be seen by using that 
	the partition of unity from \cite[Proposition 2.10]{Bahouri11}
	can be chosen so that 
	\begin{equation}
		1/2 \leq \Xi^2(x) + \sum_{j = \geq 0} \varphi(2^j x)^2 \leq 1, \quad \forall x \in \R,
	\end{equation}
	 and the fact that 
	\begin{equation*}
		\mathfrak{F}((1-\Delta))^{-s/2}u(\xi) \propto (1+|\xi|^2)^{-s/2}\mathfrak{F}u(\xi).
	\end{equation*}
	The equivalence of the norms $\| \cdot\|_{\mathcal{L}_{s,2}}$ and $\| \cdot\|_{\Lambda_{2,2}^s}$ follows from \cite[Chapter 5, Theorem 5]{Stein70}.
\end{proof}
The generator of  the symmetric, $\alpha$-stable process $(L_t)_{t \geq 0}$ is the fractional Laplacian $\Delta_\alpha$. The variational definition has the following form (see \cite[p.4-5]{Ainsworth18})
\begin{equation*}
	(\Delta_\alpha f, g) = c(\alpha) \int_\R \int_\R \frac{\big(f(x)-f(y)\big)\big(g(x)-g(y)\big)}{|x-y|^{1+\alpha}}dxdy \quad \forall g,f \in \mathcal{C}_c^\infty.
\end{equation*}
Hereby, $(\cdot,\cdot)$ denotes the inner product of $L^2(\R,m)$.
For $f =g$ this coincides (up to constants) with the second term of \eqref{eq:Lip_norm} for $s = \alpha/2$ and for this reason, we consider the capacity associated with 
the $B^s_{2,2}$. The following definition is taken from \cite{Adams95}.
\begin{definition}(Capacity)
	Let $K \subset \R$ be compact and set 
	\begin{equation*}
		\mathfrak{L}_K = \{ g \in \mathcal{S} \colon g \geq 1 \text{ on } K \}. 
	\end{equation*}
	Let $s \geq 0$ and set 
	\begin{equation*}
		\mathrm{Cap}_{s} (K) = \inf\{\|g\|^2_{B^s_{2,2}} \colon g \in \mathfrak{L}_K\}.
	\end{equation*}
	Furthermore, for open $A\subseteq \R$ define
	\begin{equation*}
		\mathrm{Cap}_s(A) = \sup\{\mathrm{Cap}_s(K) \colon K \subseteq A, K \text{ compact  }\}
	\end{equation*}
	and finally for arbitrary $E \subseteq \R$ 
	\begin{equation*}
		\mathrm{Cap}_s(E) = \inf\{\mathrm{Cap}_s(A) \colon E \subseteq A, A \text{ open  }\}.
	\end{equation*}
\end{definition}
For compact sets $D\subset \R$ it is possible to derive the following dual definition, which already suggests
that the set $A \in \mathcal{B}(A)$ such that $\mathrm{Cap} _s(A) = 0$ can be characterized by a certain family of measures. This is \cite[Theorem 2.2.7]{Adams95}.
\begin{theorem}\label{th:dual_def}
	Let $s >0$ and $D \subseteq \R$ compact. Then there exists a constant $c> 0$ such that 
	\begin{equation*}
		c^{-1}\mathrm{Cap}_{s}(D) \leq \sup_{\eta \in \mathcal{M}^+(K)}\Big(\frac{\eta(D)}{\|G_s\ast\eta\|_{L_2(\R)}}\Big)^2 \leq c \mathrm{Cap}_s(D),
	\end{equation*}
	where $G_s$ is the Bessel potential 
	and $\mathcal{M}^+(K)$ is the set of non-negative,  finite measures on $K$.
\end{theorem}
It turns out, that sets of $\mathrm{Cap}_s$ for $s\in (0,1/2)$ can be characterized by measures of a certain Kato class. 
These classes of  measures are defined in the following way and play a central role in our subsequent analysis.
\begin{definition}[{Kato class, }{\cite[Definition 1.2]{Kim14}}]\label{def:Kato}
	A signed measure $\nu$ on $\R^d$ is said to be of Kato class $\eta \in (0,d)$ if
	\begin{equation*}
		\lim_{r \downarrow 0}M_\nu^\eta(r)  = \lim_{r \rightarrow 0} \sup_{x\in \R^d} \int_{B(x,r)}
		\frac{|\nu|(dy)}{|x-y|^{d-\eta}} =0.
	\end{equation*} 
	We denote by $K_{\eta}$ the set of all measures of  Kato class $\eta$ and by $K_{\eta}^f$ the set of
	all finite measures of  Kato class $\eta$.
\end{definition}
\begin{rema}
	Kato class quantifies the regularity or smoothness of measures. Generally, the regularity of a measure decreases while $\eta$ increases. 
	To illustrate this point we can consider $d = 1$ and the measure $\nu_s(dx) = |x|^{-s}$. Then $\nu_s \in K_{\eta}$ for all $\eta > s$.
\end{rema}
The following theorem from \cite[Corollary 2.5]{Albeverio92} demonstrates that the measures of Kato class $K_{2s}$
define the sets of capacity zero with respect to $\mathrm{Cap}_{s}$ for $s \in (0,1/2)$.
\begin{theorem}\label{th:capacity_condition}
	Let $A \in \mathcal{B}(\R)$ and $s \in (0,1/2)$. Then $\mathrm{Cap}_s(A) = 0$ if and only if $\mu(A) = 0$ for
	all  measures $\mu \in K_{2s}^f$.
	%	such that
	%	\begin{equation*}
		%		\int_\R |v(x)|\mu(dx) \lesssim \|v\|_{B_{2,2}^s} \quad \forall v \in B_{2,2}^s \cap C_c(\R).
		%	\end{equation*}
\end{theorem}

\subsection{Besov Spaces}\label{sec:defKato}
In this section, we define the notion Besov spaces and investigate their relation to measures of Kato class.
Besov spaces are  another way to quantify the regularity of measures. These spaces are also suitable for the more general Schwarz distributions and were for instance used in \cite{Athreya20,Kremp22} to quantify the regularity of the drift. We will rely on Besov spaces and their properties for some technical details.

\begin{definition}[{Besov Spaces,}{ \cite[Chapter 2]{Bahouri11}}]\label{def:Besov}
	Let $\mathscr{C} = [-8/3, -3/4]\cup[3/4, 8/3]$ and let $\Xi$ and $\varphi$ be radial functions in $\mathcal{C}_c^\infty(\R)$
	supported on $B(0,3/4)$ and  $\mathscr{C}$, respectively. Furthermore, let $\Xi, \varphi$ be such that 
	\begin{align*}
		&\Xi(x) + \sum_{j \geq 0} \varphi(2^{-j}x) = 1, \quad \forall x \in \R, \\
		&\sum_{j \in \mathbb{Z}}\varphi(2^{-j}x) = 1, \quad \forall x \in \R\backslash\{0\}, 
	\end{align*}
	and such that 
	\begin{align*}
		&|j-j'| \geq 2 \Rightarrow \mathrm{supp}(\varphi(2^{-j} \cdot)) \cap \mathrm{supp}(\varphi(2^{-j'} \cdot)) = \emptyset,\\
		& j \geq 1 \Rightarrow \mathrm{supp}(\Xi) \cap \mathrm{supp}(\varphi(2^{-j'} \cdot)) = \emptyset.
	\end{align*}
	Set
	\begin{align*}
		&\Delta_{-1}f= \mathfrak{F}^{-1} \Xi \mathfrak{F} f , \quad \Delta_jf =\mathfrak{F}^{-1} \varphi(2^{-j}\cdot)\mathfrak{F} f , \quad j \geq 0, \forall f \in \mathcal{S}'.
	\end{align*}
	Hereby, $\mathfrak{F}$ denotes the Fourier transform, which is defined in the distributional sense.
	Let $\beta \in \R$ and $(p,r) \in [1,\infty]^2$. Define
	\begin{align}
		&B^\beta_{p,r} = \Big\{f \in \mathcal{S}' \colon \|f\|_{B_{p,r}^\beta}  = \big(\sum_{j \geq -1} 2^{rj\beta}\|\Delta_j f\|_{L_p}^r\big) ^{1/r}< \infty\Big\}, 
	\end{align}
	with the usual adaptations when $p = \infty$ or $r = \infty$.
	For $(p,r) \in [1,\infty]^2$ we call the spaces $B^\eta_{p,r}$ (nonhomogeneous) Besov spaces.
\end{definition}
In this paper we will only use the spaces $B^\eta_{p,q}$ with $p = q= \infty$. 
If $\beta > 0$ it holds  that $B^\beta_{\infty,\infty} $ is the space of H\"older continuous functions with exponent $\beta$. Recall, that we denoted $B^\beta_{\infty,\infty}$ by $\mathcal{C}^\beta$ for al $\beta \in \R$.
By \cite[Theorem 9, Remark 26]{Triebel88} it turns out that we have the following characterization of  $\mathcal{C}^\beta$ for $\beta <0$
\begin{equation}\label{eq:besov_char}
	c^{-1} \| f\|_{\mathcal{C}^\beta} \leq \sup_{t \in (0,1]} t^{-\beta/2}\| \ee^{t \Delta}f\|_\infty \leq c \| f\|_{\mathcal{C}^\beta} , \quad \forall f \in \mathcal{C}^\beta,
\end{equation}
where $c>0$ is a constant and $\ee^{t\Delta}$ is the heat semigroup defined as
\begin{equation*}
	\ee^{t\Delta}f (x)= \frac{1}{\sqrt{2\pi}} t^{-1/2}\int_\R \ee^{-\frac{|x-y|^2}{4t}}f(y)dy, \quad t>0, x\in \R,  f \in \mathcal{S}'.
\end{equation*}
In the next lemma, we investigate the relationship of Besov spaces and measures of Kato class $K_{\alpha-1}$ for $\alpha \in (1,2)$. Hereby, we restrict ourselves to $d = 1$. \\
\begin{lemm}\label{lem:KatoBesov}
	In dimension $d = 1$ we have the following inclusions:
	\begin{enumerate}[label = (\roman*)]
		\item \label{itm:fin_m_bes}
		\begin{equation}
			\mathcal{M}_f^+\cap \mathcal{C}^{-s}\subseteq K_{\eta}^f, \quad 0<s<\eta<1;
		\end{equation}
		
		\item \label{itm:kato_bes_non}
		\begin{equation}
			K_\eta^f \subseteq \mathcal{C}^{-s} \cap \mathcal{M}, \quad 0<\eta <s<1.
		\end{equation}
	\end{enumerate}
	Hereby,  $\mathcal{M}$ denotes the set of all signed Radon measures on $\R$, $\mathcal{M}^+$ denotes the set of all 
	non-negative Radon measures on $\R$ and $	\mathcal{M}_f^+$ denotes the set of all nonnegative, finite Radon measures on $\R$.
\end{lemm}
The proof of the above lemma can be found in Appendix~\ref{sec:appendic_bk}.

\section{Main Results}\label{sec:mainresults}
In this section, we state our main results. We restrict ourselves to the case $d =1$
in all our results.
\subsection{Existence and Uniqueness Results for \eqref{eq:singularSDE}}
We first turn our attention to solutions to the formal  singular SDE of the form \eqref{eq:singularSDE}.
Hereby, we make sense of the term $\int_0^t \mu(X_s)ds$ by employing an approximation scheme that uses mollification. 
 \begin{definition}[Mollifier]\label{def:mollifier}
 	Consider a non-negative, radial function $\varphi \in \mathcal{C}^\infty$ and assume that $\mathrm{supp}(\varphi) = [-1,1]$, $\int_\R\varphi(x)dx = 1 $. We denote
 	\begin{equation*}
 		\varphi^{(\epsilon)} (x) = \epsilon^{-1}\varphi(\epsilon^{-1} x),
 	\end{equation*}
 	and 
 	\begin{equation*}
 		\mu^{(\epsilon)}(x) = \int_\R \varphi_\epsilon(x-y)\mu(dy),
 	\end{equation*}
 	for all Radon measures $\mu$ and
 	\begin{equation*}
 		u^{(\epsilon)}(x) =\int_\R \varphi_\epsilon(x-y)u(y)dy, 
 	\end{equation*}
 	for all measurable functions $u$.
 	When $\epsilon = 2^{-n}$ we write $\varphi^{(2^{-n}) }= \varphi_n$, 
 	$\mu^{(2^{-n})} = \mu_n$ and $u^{(2^{-n})} = u_n$ instead.
 \end{definition}
 With this definition at hand we can state the notion of solution we are working with.
\begin{definition}\label{def:sing_sol}
	Let $\mu$ be a signed measure in $K_{\alpha-1}$, $\alpha \in (1,2) $ and let $(L_t)_{t \geq 0}$ be a symmetric $\alpha$-stable process in one dimension. 
	A c\`adl\`ag process $(X_t)_{t \geq 0}$ is a solution to \eqref{eq:singularSDE} with initial condition $x_0$, if there exists a continuous process $(A_t)_{t \geq 0}$ such that almost surely
	\begin{equation}
		X_t = x_0 + L_t +A_t, \quad \forall t \geq 0, 
	\end{equation}
	where
	\begin{enumerate}
		\item $\int_0^t \mu_n(X_s)ds$ converges to $A_t$ in probability,  uniformly in  $t$ over finite intervals;
		\item there exists a subsequence $\{n_k\}$ such that $\sup_k\int_0^t|\mu_{n_k}(X_s)|ds<\infty$ almost surely for each $t>0$;
		\item under $\bbP$, $L_0 = 0$ and $(L_t)_{t \geq 0}$ is a symmetric $\alpha$-stable process on $(\Omega, (\mathcal{F}_t)_{t \geq 0})$.
	\end{enumerate}
\end{definition}
A \emph{weak} solution to \eqref{eq:singularSDE} is a pair $(X,L)$ such that there exists a stochastic basis $(\Omega, (\mathcal{F}_t)_{t \geq 0}, \bbP)$ 
such that $L$ is a symmetric, $\alpha$-stable  process and $X$ is adapted to $(\mathcal{F}_t)_{t \geq 0}$ and is a solution to \eqref{eq:singularSDE}. Sometimes, with some abuse of notation, we will call a corresponding probability law $\bbP_x$ (the law of $X$ with $X_0 =x$) a weak solution to \eqref{eq:singularSDE}.
A \emph{strong} solution with respect to a given symmetric, $\alpha$-stable is a process $X$ which satisfies Definition~\ref{def:sing_sol} and is adapted to the filtration generated by $L$. 
\emph{Weak uniqueness} is said to hold for solutions to \eqref{eq:singularSDE} if, given any two weak solutions $(X,L)$ and $(\tilde{X},\tilde{L})$
 sharing the same initial distribution $X_0$, the processes $(X_t)_{t \geq 0}$ and $(\tilde{X}_t)_{t\geq 0}$ have the same law. \emph{Pathwise uniqueness}, 
 on the other hand, is established when any two weak solutions $(X,L)$ and $(\tilde{X}, L)$ on the same probability space and with a common $L$ and initial condition $X_0$ are indistinguishable.
  Finally, \emph{strong uniqueness} is achieved if any two strong solutions $X$ and $\tilde{X}$ of \eqref{eq:singularSDE}, relative to a common $L$ and initial condition $X_0$ are indistinguishable.
   Notably, pathwise uniqueness implies strong uniqueness.

In \cite{Kim14} the authors proved weak existence and uniqueness of solutions to \eqref{eq:singularSDE} when $\mu \in K_{\alpha-1}$ and in dimension $d \geq 2$. 
Following their approach, it is possible to extend this result to $d = 1$. In fact, in the one-dimensional setting it is possible to strengthen this result and obtain \emph{strong} existence and uniqueness at the expense of assuming that $\mu \in K^f_{\eta}$ with $\eta \in (0,\alpha-1)$. This can be achieved by proving pathwise uniqueness of solutions and then using a generalization of the Yamada- Watanabe theorem from \cite{Kurtz14}.
Here, we would like to mention the argument used in \cite{Tanaka74} for proving the pathwise uniqueness 
for a stochastic differential equation driven by additive noise and with H\"older drift. The proof in \cite{Tanaka74} heavily uses the assumption that the drift is function-valued. 
The generalization of this proof to measure-valued drift would be interesting but does not look trivial. Thus, we chose to use the method of proof that uses the 
Zvonkin transformation, which, in our opinion, could be further generalized and also applied to equations driven by a multiplicative noise. 
This means, we apply a Zvonkin transformation $\phi$, akin to the transformation used in \cite{Athreya20}, to \eqref{eq:singularSDE}. 
The resulting transformed equation is satisfied by $\phi(X)$, where $X$ is a solution to \eqref{eq:singularSDE}.
It possesses coefficients of higher regularity than the original equation, which enables the utilization of results from \cite{Knopova17} 
to establish weak uniqueness of solutions to the transformed equation.
With weak uniqueness in hand, we then employ similar techniques to \cite{Xiong19} to prove pathwise uniqueness for solutions to the transformed equation.
Given that every solution to \eqref{eq:SDE} is also a solution to the transformed equation, pathwise uniqueness for the latter implies pathwise uniqueness for the former.

\begin{theorem}\label{th:sde_exist}
	Let $\mu \in K_{\eta}^f$ with $\eta \in (0, \alpha-1]$.
	\begin{enumerate}[label = (\roman*)]
	\item \label{itm:weakex} Then  there exists a unique weak solution to \eqref{eq:singularSDE} and the solution is a strong Markov process and $(A_t)_{t\geq 0}$ is a continuous, adapted process of finite variation.
	\item \label{itm:stronex}If $\eta \in (0,\alpha-1)$, then the weak solution to \eqref{eq:singularSDE} is strong. Moreover, pathwise uniqueness holds for solutions to \eqref{eq:singularSDE}.
	\end{enumerate}
	 
\end{theorem}
An immediate consequence of this theorem is the following result.

\begin{theorem}
	Let $\mu \in K_{\eta}^f$ with $\eta \in (0, \alpha-1)$. Then there exists a unique strong solution to \eqref{eq:singularSDE}.
\end{theorem}

In Definition~\ref{def:sing_sol} we fix a particular sequence of smooth functions approximating the drift $\mu$. Another 
common approach to define a solution to \eqref{eq:singularSDE}, that was for instance used in \cite{Athreya20}, is to require the convergence 
\begin{equation}\label{eq:conv_A_PUc}
	\lim_{n \rightarrow \infty} \bbP (\sup_{t \in [0,T]}|\int_0^t\tilde{\mu}_n(X_s)ds-A_t|>\epsilon) = 0, \quad \forall \epsilon >0, \forall T > 0,
\end{equation}
to hold for every sequence of smooth functions $(\tilde{\mu}_n)_{n \in \N}$ that converges to $\mu$ in an appropriate way. 
In the next lemma we provide sufficient conditions on a sequence $(\tilde{\mu}_n)_{n \in \N}$ that guarantee that \eqref{eq:conv_A_PUc} holds. 

\begin{lemm}\label{cor:PUC}
	Let $\mu \in K_{\alpha-1}^f$ and let $X$ be a solution to \eqref{eq:singularSDE} in the sense of Definition~\ref{def:sing_sol}. Let $(\tilde{\mu}_n)_{n \in \N}$ be a sequence 
	of bounded smooth functions, converging weakly to $\mu$ such that 
	\begin{equation}\label{eq:limsup_M}
		\lim_{r \rightarrow 0} \sup_{n \in \N} M^{\alpha-1}_{\tilde{\mu}_n}(r) = 0.
	\end{equation}
	Then  $\Big(\int_0^t \tilde{\mu}_n(X_s)ds\Big)_{t \geq 0}$ converges to $A$ in probability,  uniformly in  $t$ over finite intervals.
\end{lemm}
The proof of Lemma~\ref{cor:PUC}  is given in Section~\ref{sec:aprox_seq}. 
\begin{rema}\label{rema:molli}
	The assumptions of Lemma~\ref{cor:PUC} are in particular satisfied for the sequence $(\mu_n)_{n\in \N}$ of mollifications of $\mu\in K_{\alpha-1}^f$ defined by 
	Definition~\ref{def:mollifier}. That is, it is well-known that $(\mu_n)_{n \in \N}$ converges weakly to $\mu$ and it is easily verified that 
	\begin{equation*}
		M^{\alpha-1}_{\mu_n}(r) \leq C M^{\alpha-1}_{\mu}(r) , \quad r \geq 0, \forall n \in \N,
	\end{equation*}
	which yields \eqref{eq:limsup_M}.
\end{rema}
\begin{rema}
	Lemma~\ref{cor:PUC} demonstrates that we can give an alternative formulation of  Definition~\ref{def:sing_sol}  by replacing (1)  and (2) with the following. 
	\begin{itemize}
		\item[(1)']  For any sequence $\tilde{\mu}_n \in \mathcal{C}_b^\infty$, converging weakly to $\mu$ and satisfying \eqref{eq:limsup_M}, it follows that $\Big(\int_0^t \mu_n(X_s)ds\Big)_{t \geq 0}$ converges in probability,
		  uniformly in  $t$ over finite intervals to $A$.
		\item[(2)'] For any $(\tilde{\mu}_n)_{n \geq 0}$ from (1)' there exists a subsequence $\{n_k\}$ such that $\sup_k\int_0^t|\tilde{\mu}_{n_k}(X_s)|ds<\infty$ almost surely for each $t>0$.
	\end{itemize}
	Then, by Theorem~\ref{th:sde_exist}, Lemma~\ref{cor:PUC} and Remark~\ref{rema:molli}, $X$ is a solution to \eqref{eq:singularSDE}
		 in the sense of Definition~\ref{def:sing_sol} if and only if it is a solution to \eqref{eq:singularSDE} in the sense of Definition~\ref{def:sing_sol} with (1), (2) replaced by (1)', (2)'.
\end{rema}
\subsection{SDEs involving the Local Time of the Unknown Process}
Now, we turn our attention to equations of the form \eqref{eq:SDE}. This includes rigorously defining the local time $\gamma$ and 
a notion of solution for \eqref{eq:SDE}, as well as investigating properties of $\gamma$ and proving the equivalence of \eqref{eq:SDE} and \eqref{eq:singularSDE} when $\mu \in K_{\alpha-1}^f$. \par
First, we define the local time of one-dimensional semi-martingales of the form
\begin{equation}\label{eq:semimartingale}
	X_t = x_0+ L_t +A_t, \quad t \geq 0,
\end{equation}
where $L$ is a symmetric $\alpha$-stable process with $\alpha \in (1,2)$ and $A$
is an adapted, right continuous process of finite variation and $x _0\in \R$. The main idea is to extend the result from \cite{Salminen07} to the semimartingale case. 
Before we state the definition, we motivate it with the following formal calculations.
Note that the function $h(x) = |x|$, used in Tanaka's formula for Brownian motion, is the renormalized  zero potential density of the Brownian motion. This means, we seek to replace $h$ with 
 a formal solution to the equation 
 \begin{equation*}
 	\Delta_\alpha v = \delta_0,
 \end{equation*}
 for $\alpha \in (1,2)$.
In the one-dimensional case, the zero potential of $L$ does not exist but it is possible to 
derive a renormalized soluton in the following way.
Let $u_\lambda$ be the $\lambda$-potential density of $L$ for $\lambda >0$. 
It is well known from \cite[p.67]{Bertoin96} that 
\begin{equation}\label{eq:lam_pot_dens_closed_form}
	u_\lambda(x) = \frac{1}{\pi} \int_0^\infty\frac{\cos(\xi x)}{\lambda + \Psi(\xi)}d\xi = \frac{1}{\pi} \int_0^\infty\frac{\cos(\xi x)}{\lambda + |\xi|^\alpha}d\xi,
\end{equation}
where $\Psi$ is defined in \eqref{eq:characteristic}. 
Then the following convergence holds due to \cite[Lemma 3.2]{Salminen07}.
For every $x \in \R$, as $\lambda \rightarrow 0$
\begin{equation}\label{eq:renorm_pot_dens}
	u_\lambda(0) -u_\lambda(x) \rightarrow v(x) =\frac{1}{\pi}\int_0^\infty\frac{1-\cos(\xi x)}{|\xi|^\alpha}d\xi= c_2(\alpha)|x|^{\alpha-1} .
\end{equation} 
Hereby, 
\begin{equation*}
	c_2(\alpha) = (2\pi c_1(\alpha-1))^{-1},
\end{equation*}
with $c_1(\alpha-1)$ given by  \eqref{eq:c_1}.
Consequently, by an application of the Leibniz rule
\begin{equation}\label{eq:deriv_u_lam}
	\frac{\partial}{\partial x}u_\lambda(x) = \frac{1}{\pi}\int_0^\infty \frac{-\xi \sin(\xi x)}{\lambda + |\xi|^\alpha}d\xi.
\end{equation}
It turns out that this yields (see \cite[Propoosition 4.1]{Salminen07})
\begin{equation}\label{eq:tanaka_alpha_stable}
	v(L_t-a) = v(x_0-a) + N_t^a + \gamma_t^a, \quad \forall a\in \R,
\end{equation}
where $N_t^a$ is a martingale and $(\gamma_t^x)_{t \geq 0}^{x \in \R}$ is the occupation density of $L$.  
Note that \eqref{eq:tanaka_alpha_stable} is an analogue of \eqref{eq:tanaka_bm}, the Tanaka formula for Brownian motion, with $v(\cdot)$ replacing the function $h(\cdot) = |\cdot|$.
We now seek to extend \eqref{eq:tanaka_alpha_stable} to the case where we replace $L$ by a semimartingale $X = x_0+L+A$. A formal application of It\^o's formula yields the formal equality
\begin{equation}\label{eq:formal_discussion}
	\int_0^t \delta_0(X_s-x)ds = v(X_t-x)- v(x_0-x) -\int_0^t v'(X_s-x)dA_s-N^x_t,
\end{equation}
where $N^x_t$ is a martingale. This formal discussion suggests that the right hand side of \eqref{eq:formal_discussion} is a good candidate to define a local time that can be identified with the occupation density of $X$.
Now, we give a rigorous definition of the local time that we will use in the remainder of this article. To this end, recall that the Poisson random measure of $L$ is given by $\Pi$ with compensator $\pi$ given by \eqref{eq:intensity}
and L\'evy measure $\tilde{\pi}$ given by \eqref{eq:poissoon_intense}.
\begin{definition}\label{def:loc_time}
	Let $X$ satisfy \eqref{eq:semimartingale} and fix arbitrary $t \geq 0$. Define the set
	\begin{equation}\label{eq:set_loc_time}
		\begin{split}
			D_t:= &\Big\{x \in \R \colon \sup_{0 \leq s \leq t}\Big(v(X_s-x) +\Big|\int_0^sv'(X_\tau-x)dA_\tau\Big|\Big) < \infty , \bbP_{x_0}\text{-a.s.} ,\\
			&\quad \bbE_{x_0}\Big[\sup_{0 \leq s \leq t} \int_0^s\int_\R (v(X_{\tau-}-x+r) -v(X_{\tau-}-x))^2 \tilde{\pi}(dr)d\tau\Big] < \infty\Big\}
		\end{split}
	\end{equation}
	and let
	\begin{equation}\label{eq:local_time_gamma_def}
		\gamma_s^x := \begin{cases}
			v(X_s-x)- v(x_0-x) -\int_0^s v'(X_\tau-x)dA_\tau -N^x_s, \quad  x \in D_t\\
			0, \quad \text{ else },
		\end{cases}
	\end{equation}
	where 
	\begin{equation}\label{eq:martingale}
		N^x_s = \int_0^s \int_\R (v(X_{\tau-}-x+r)-v(X_{\tau-}-x))(\Pi-\pi)(dr,d\tau),
	\end{equation}
	$v$ is defined as in \eqref{eq:renorm_pot_dens} and $v'(x) = c_2(\alpha)(\alpha-1)\sgn(x)|x|^{\alpha-2}$ for
	$x \neq 0$ and $v'(0) = 0$.
	We call $(\gamma^x_s)_{s\in [0,t]}$ the local time of $X$ at $x$ on $[0,t]$ and we sometimes use the notation $\gamma_t^x(X)$ to indicate that $\gamma$ is the local time of the process $X$.
\end{definition}
Throughout this paper the local time of a process $X$ satisfying \eqref{eq:semimartingale} is always given as in Definition~\ref{def:loc_time}. 
\begin{rema}$~$\\
	\begin{itemize}
		\item  The condition
		\begin{equation*}
			\bbE_{x_0}\Big[\sup_{0 \leq s \leq t} \int_0^t\int_\R (v(X_{s-}-x+r) -v(X_{s-}-x))^2 \tilde{\pi}(dr)ds \Big] < \infty
		\end{equation*}
		in \eqref{eq:set_loc_time} ensures that the stochastic integral in \eqref{eq:martingale} is well defined and that $(N^x_s)_{0\leq s \leq t}$
		is a martingale (\cite[Chapter 4]{Applebaum09}).
		\item It is easily seen that the map $(t,x) \mapsto \gamma_t^x$ is measurable $\bbP$-a.s.. First note, that $D_t$ is a measurable set, since all the the terms appearing its definition can easily be approximated continuous functions by replacing $v'$ by $-M\vee(v'(x)\wedge M)$ for all $x \in \R$ and by letting $M >0$ go to infinity. Furthermore, 
		the map $(t,x) \mapsto \mathbbm{1}_{D_T}(x)\gamma_t^x$  for $t \in [0,T]$, $x \in \R$ is $\bbP$-a.s. measurable for the same reason. 
	\end{itemize}
\end{rema}
Our first result is that $(\gamma_t^x)_{t\geq0}^{x\in \R}$ is non-trivial on any compact time interval and for any $x$ outside a set of zero capacity of parameter
 $(\alpha-1)/2$.  Hereby, we rely on the notion of capacity, because set of zero
 capacity of parameter $(\alpha-1)/2$ are of measure zero with respect to the Lebesgue measure and any measure $\mu$ of Kato class $K_{\alpha-1}$. 
 This is of significance, since we integrate $\gamma$ with respect to $\mu$ in \eqref{eq:SDE} and because we want to show that $\gamma$ is an occupation density.
\begin{theorem}\label{th:local_time}
	Let $X$ satisfy \eqref{eq:semimartingale}.
	Then, it holds for every $T > 0$ that 
	\begin{equation}
		\mathrm{Cap}_{(\alpha-1)/2}(D_T^c) = 0.
	\end{equation}
\end{theorem}
An important question is, whether the local time $\gamma$ according to Definition~\ref{def:loc_time} is 
an occupation density. The measure $\mu$ may be singular with respect to the Lebesgue measure, which means that we possibly need to identify $\gamma$ with the occupation density on sets of Lebesgue measure zero.
  For this reason it is of interest to investigate when 
solutions to \eqref{eq:singularSDE} admit a continuous occupation density.
\begin{prop}\label{lem:density}
	Let $X$ be a solution of \eqref{eq:singularSDE}, where $\mu\in K_{\eta}^f$  for $\eta \in (0, \alpha-1)$. Then there exists a jointly continuous occupation density $(\ell_t^x)_{t \geq 0}^{x\in \R}$
	of $X$.
\end{prop}
In the following $(\ell_t^x)_{t \geq 0}^{x\in\R}$ always denotes the continuous modification of the occupation density, if it exists.
\begin{prop}\label{lem: convergence}
	Let $X_t = x_0 + L_t + A_t$, $t \geq 0$ where $A$ is an adapted, right-continuous process of finite variation and let $\nu \in K_{\alpha-1}^f$. Then it follows   for any $\epsilon > 0$ that
	\begin{equation}
		\lim_{n \rightarrow \infty} \bbP_{x_0}\Big(\sup_{0 \leq s \leq t}\Big| \int_\R\int_0^s \varphi_n(X_\tau-x)d\tau\mu(dx)- \int_\R  \gamma_s^{x} \nu(dx)\Big|> \epsilon\Big)  = 0,
	\end{equation} 
	where $\phi_n$ is defined as in Definition~\ref{def:mollifier} and $\gamma$ is the local time as defined in Definition~\ref{def:loc_time}.
\end{prop}
This proposition immediately yields the following corollary, by setting $\mu(dx) = f(x)dx$ for $f \in \mathcal{C}_c$ and noting that $\int_\R\int_0^t\phi_n(X_s-x)dsf(x)dx = \int_0^t f_n(X_s)ds$ converges almost surely to $\int_0^tf(X_s)ds$. 
\begin{cor}\label{cor:occ_dens_lev}
	Let $X_t = x_0 +L_t+A_t$, $t \geq 0$ where $A$ is an adapted, right-continuous process of finite variation. Let $\gamma = \gamma(X)$ be the local time of $X$.  Then  
	\begin{equation*}
		\int_\R\gamma_t^xf(x)dx = \int_0^tf(X_s)ds, \quad t   \geq 0, \text{ a.s.},
	\end{equation*}
	for all $f \in \mathcal{C}_c$. That is, $(\gamma_t^x)_{t \geq 0}^{x\in \R}$ is an occupation density.
\end{cor}
Now, we are ready to rigorously define what we understand under a solution of \eqref{eq:SDE}.
%\begin{equation}\label{eq:SDEILT_mu}
%	X_t = x+ \int_\R \gamma_t^x\mu(dx) +L_t,\quad \forall t \geq 0,
%\end{equation}
%where $\mu \in K_{\alpha-1}$, $x_0\in \R$ and $\gamma$ is the local time of $X$, and prove the equivalence to \eqref{eq:Kato_SDE}.

\begin{definition}\label{def:SDEILT}
	Let $\mu$ be a signed measure in $K_{\alpha-1}$, $\alpha \in (1,2) $ and let $(L_t)_{t \geq 0}$ be a symmetric $\alpha$-stable process in one dimension. 
	A c\`adl\`ag process $(X_t)_{t \geq 0}$ is a solution to \eqref{eq:SDE} with initial condition $x_0$, if there exists a continuous process $(A_t)_{t \geq 0}$ such that almost surely
	\begin{equation*}
		X_t = x_0+L_t+A_t, \quad \forall t \geq 0,
	\end{equation*}
	where
	\begin{enumerate}
		\item $(A_t)_{t \geq 0}$ is adapted, right continuous, and of finite variation such that 
		\begin{equation*}
			A_t =  \int_\R \gamma_t^x\mu(dx), \quad  \bbP\text{-a.s.},
		\end{equation*}
		for all $t \geq 0$, where $\gamma$ is the local time of $X$;
		\item under $\bbP$, $L_0 = 0$ and $(L_t)_{t \geq 0}$ is a symmetric $\alpha$-stable process on $(\Omega, (\mathcal{F}_t)_{t \geq 0})$.
	\end{enumerate}
\end{definition}
The notions of weak and strong existence as well as weak uniqueness, pathwise uniqueness and arong uniqueness apply to Definition~\ref{def:SDEILT} in the same way as to Definition~\ref{def:sing_sol}.
\begin{theorem}\label{th:sdeilt}
	Let $L$ be a one-dimensional, symmetric $\alpha$-stable process with $\alpha \in (1,2)$
	and $\mu \in K_{\alpha-1}^f$. Then the following holds.
	\begin{enumerate}[label = (\roman*)]
		\item\label{itm:equiv} 	Every  weak solution to \eqref{eq:SDE} is also a weak solution to \eqref{eq:singularSDE} and vice versa.
		\item \label{itm:density} If $\mu \in K_{\eta}$ with $\eta \in (0,\alpha-1)$, then
		the solution of \eqref{eq:SDE} has a jointly continuous occupation density $(\ell_t^x)_{t \geq 0}^{x\in \R}$
		such that $\gamma_t^x = \ell_t^x$ $\bbP$-a.s. for all $t \geq 0$  and all  $x \in \R \backslash N$  for some set 
		$N \subset \R$  with $\mathrm{Cap}_{(\alpha-1)/2}(N) = 0$.
	\end{enumerate}
\end{theorem}
The above theorem together with Theorem~\ref{th:sde_exist} immediately yields the following theorem.
\begin{theorem}
	Let $\mu \in K^f_{\eta}$, $\eta \in (0, \alpha-1]$.
	\begin{enumerate}[label = (\roman*)]
	 \item Then there exists a unique weak solution to \eqref{eq:SDE}.
	 \item If  $\eta \in (0,\alpha-1)$ the weak solution to \eqref{eq:SDE} is strong. Moreover, pathwise uniqueness holds for solutions to \eqref{eq:SDE}.
	\end{enumerate}
\end{theorem}
Again, as an immediate consequence of the above theorems we obtain the following result.
\begin{theorem}
	Let $\mu \in K_{\eta}^f$ with $\eta \in (0, \alpha-1)$. Then there exists a unique strong solution to \eqref{eq:SDE}.
\end{theorem}

\begin{rema}
	In Theorem~\ref{th:sdeilt} we prove that $\gamma_t^x = \ell_t^x$ almost surely for all $t$ and all $x$ outside of a set of zero $(\alpha-1)/2$ -capacity 
	when $\mu \in K^f{\eta}$ for $\eta \in (0,\alpha-1)$. However, we conjecture that, in this case, there is in fact a continuous modification of $\gamma$ with spatial modulus of continuity of  $\frac{\alpha-1-\eta}{2}$ up to some logarithmic terms.
\end{rema}
The subsequent sections are dedicated to the proofs of the above results.
\section{Proof of Theorem~\ref{th:sde_exist}}\label{sec:sde_exist}
In this section, we discuss the  the proof of Theorem~\ref{th:sde_exist}.  For this purpose we fix
$\mu \in K_{\eta}^f$ with $\eta \in (0,\alpha-1]$ throughout this section.
First, we will discuss weak existence of solutions by constructing the transition densities in Section~\ref{sec:weakexist}.
Next, we analyze the associated potentials. These will play an important role in the proof
of weak uniqueness, pathwise uniqueness via Zvonkins transform and also in the proof of Proposition~\ref{lem:density}. 
By using the fact that the potentials uniquely determine the law, we prove weak uniqueness of solutions to \eqref{eq:singularSDE} in Section~\ref{sec:weak_unique}.  Thus, by combining the results in Sections~\ref{sec:weakexist} and  ~\ref{sec:weak_unique} we conclude the proof of Theorem~\ref{th:sde_exist}~\ref{itm:weakex}.\par
Finally, in Section~\ref{sec:pathwise}, we use a Zvonkin transformation to establish pathwise uniqueness of solutions when $\mu \in K^f_{\eta}$ for $\eta \in (0,\alpha-1)$. In \cite{Athreya20} the authors demonstrated that  singular SDEs driven by $\alpha$-stable L\'evy processes can be reformulated to fit in the framework of \cite{Kurtz14}, which allows to infer strong existence and uniqueness from weak existence and pathwise uniqueness. Since this is fairly standard we omit it here.  Thus, in Section~\ref{sec:pathwise} we conclude the proof of Theorem~\ref{th:sde_exist}~\ref{itm:stronex}.

\subsection{Weak Existence}\label{sec:weakexist}
To prove weak existence,  we first construct the transition densities associated with \eqref{eq:generator} via a perturbation approach. 
The generator of the solution
of \eqref{eq:singularSDE}
is expected to be of the form
\begin{equation}\label{eq:generator}
	\Delta_\alpha^\mu f  = \Delta_\alpha f  + \mu \frac{\partial }{\partial x}f
\end{equation} 
for $f \in \mathcal{C}_c^\infty$. The operator can be extended to a larger class of functions but this is not needed here.
The transition densities $p^\mu$  are the  fundamental solutions of the Cauchy problem associated with $(\frac{\partial}{\partial t}- \Delta_\alpha^{\mu})$, that is they satisfy
\begin{align*}
	&(\frac{\partial}{\partial t}- \Delta_\alpha^{\mu})p^\mu(t,x,y) =0, \quad x,y \in \R , t >0, \\
	&p^\mu(t,x,\cdot) \rightarrow \delta_x, \quad t \rightarrow 0^+, \quad x\in \R.
\end{align*}
In order to rigorously define $p^\mu$, we we will use Theorem~\ref{th:transition_kernel} below, which demonstrates that $p^\mu$ can be constructed as the following limit:
\begin{equation}\label{eq:dens_formal}
	p^\mu(t,x,y) = p(t,x,y) + \sum_{k = 1}^\infty( p\circledast (\frac{\partial}{\partial x} p\mu)^{\circledast k})(t,x,y).
\end{equation}
Here $p(t,x,y)$ for $t>0,x,y \in \R$ denote the transition densities of the fractional Laplacian $\Delta_\alpha$ and we use the notation
\begin{equation*}
	(f\circledast g)(t,x,y) = \int_0^t \int_\R f(t-s,x,z)g(s,z,y)dzds, \quad x,y \in \R , t >0, 
\end{equation*}
and
\begin{equation*}
	(f)^{\circledast k}(t,x,y) = \int_0^t \int_\R f(t-s,x,z)f^{\circledast (k-1)}(s,z,y)dzds,  \quad x,y \in \R , t >0,  \quad k \in \mathbb{N}, k \geq 2.
\end{equation*}
Theorem~\ref{th:transition_kernel} (i), (ii)  is a slight modification of \cite[Theorem 3.10]{Kim14} and can be proven along the same lines. 
Part (iii) of Theorem~\ref{th:transition_kernel} can be proven following the lines of \cite[Theorem 1]{Bogdan07} or \cite[Proposition 2.3]{Chen12}. Thus, the proof of Theorem~\ref{th:transition_kernel} is omitted.
\begin{theorem}\label{th:transition_kernel}
	Let $\mu \in K_{\alpha-1}^f$. Then the following holds.
	\begin{itemize}
		\item [(i)] There exist $T_0(\mu)>0$, such that  $p(t,x,y) + \sum_{k = 1}^\infty( p\circledast (\frac{\partial}{\partial x} p\mu)^{\circledast k})$
		converges locally uniformly on $(0, T_0] \times \R \times \R$ to a positive continuous function $p^{\mu}(t,x,y)$ and on 
		$(0, T_0] \times \R \times \R$ that satisfies 
		\begin{equation}\label{eq:trans_density_est}
			c_1^{-1}\Big(t^{-1/\alpha} \wedge \frac{t}{|x-y|^{1+\alpha}}\Big) \leq p^{\mu}(t,x,y) \leq c_1 \Big(t^{-1/\alpha} \wedge \frac{t}{|x-y|^{1+\alpha}}\Big), 
		\end{equation} 
		where $c_1(\mu)>0$ is a constant.
		Moreover, $\int_\R p^{\mu}(t,x,y)dy = 1$ for every $n \geq 1$, $t \in (0,T_0]$ and $x \in \R$.
		\item[(ii)] The function $p^{\mu}(t,x,y)$ defined in (i) can be extended uniquely to a positive
		jointly continuous function on $(0, \infty)\times \R \times \R$ so that for all $s,t \in (0,\infty)$ and $(x,y) \in \R \times \R$, $\int_\R p^{\mu}(t,x,y)dy = 1$ 
		and 
		\begin{equation}
			p^{\mu}(t+s,x,y) = \int_\R p^{\mu}(s,x,z)p^{\mu}(t,z,y)dz.
		\end{equation}
		\item[(iii)] Define 
		\begin{equation}
			P_t^{\mu}f(x) := \int_\R p^{\mu}(t,x,y)f(y)dy, \quad \forall f \in \mathcal{C}_c^\infty.
		\end{equation}
		Then for any $f \in \mathcal{C}_c^\infty(\R)$ and $g \in C_\infty(\R)$
		\begin{equation}
			\lim_{t \downarrow 0} \int_\R t^{-1} (P_t^{\mu}f(x)-f(x))g(x) dx= \int_\R \Delta_\alpha^{\mu}f(x)g(x)ds.
		\end{equation}
		Thus, $p^{\mu}(t,x,y)$ is the fundamental solution of $(\frac{\partial}{\partial t}- \Delta_\alpha^{\mu})$ in the distributional sense.
	\end{itemize}
\end{theorem}
Following the arguments of \cite[Proposition 4.2]{Kim14} and \cite[Proposition 2.3]{Chen12} again, the above theorem in particular implies the existence of a conservative Feller process $(X_t, \bbP_x)$ with transition semigroup $(P^\mu_t)_{t \geq 0}$.
Now, we need to show that the process $(X_t, \bbP_x)$ solves \eqref{eq:singularSDE} according
to Definition~\ref{def:sing_sol}. In fact, the argument to show existence is exactly the same as in \cite{Kim14}. 
For this reason, we only sketch the main ideas. \par
To this end, we need to consider the $\lambda$-potentials of the Markov process associated with the transition semigroup $(P_t^\mu)_{t \geq 0}$. They are denoted by
\begin{equation}
	W_\lambda^\mu f(x) = \int_0^\infty \ee^{-\lambda t}P_t^\mu f(x)dt.
\end{equation} 
Note that the action of both $(P_t^\mu)_{t \geq 0}$ and $W^\mu_\lambda$, $\lambda >0$ can also be extended to measures via 
\begin{equation*}
	P_t^\mu \nu(x) = \int_\R p^\mu(t,x,y) \nu(dy), 
\end{equation*}
whenever the integral wit respect to the measure $\nu$ is well-defined.
Definition~\ref{def:sing_sol} relies on an approximation scheme. For this reason we will need the following convergence result, which is a slight modification of \cite[Lemma 4.7]{Kim14}.
\begin{lemm}\label{lem:pot_conv}
	Let $\lambda >0$ and $\nu \in K_{\alpha-1}$. Then 
	$W^{\mu}_\lambda\nu_n$,  $W^{\mu_n}_\lambda\nu_n$
	and $W^{\mu_n}_\lambda\nu$ all converge to $W^{\mu}_\lambda\nu$, uniformly on compact subsets of $\R$ as $n \rightarrow \infty$.  Hereby, $\nu_n$ and $\mu_n$ denote the mollifications of $\nu$ and $\mu$, respectively, according to Definition~\ref{def:mollifier}.
	If $\nu$ is a finite measure, the potentials are bounded functions and the convergence holds uniformly on $\R$.
\end{lemm}
\begin{proof}
	The proof follows similar steps as the proofs of \cite[Lemma 4.7, 4.8]{Kim14} and, thus, is omitted.
\end{proof}
For any $\nu\in K_{\alpha-1}$ we can
split $\nu = \nu^+-\nu^-$ in the negative and positive part and consider 
$W^\mu_\lambda\nu^+$ and $W^\mu_\lambda\nu^-$, which are both bounded, continuous potentials of $X$, since they are, by Lemma~\ref{lem:pot_conv}, the limits of $W^\mu_\lambda\nu^+_n$ and $W^\mu_\lambda\nu^-_n$, respectively. 
Thus, by \cite[Corollary IV 3.14]{Blumenthal68} there exist  positive continuous additive functionals $A_t^{\nu^+}$  and $A_t^{\nu^-}$ on the same probability space as $X$, such that 
\begin{equation}
	W^\mu_\lambda\nu^+(x) =\bbE_x\Big[ \int_0^\infty \ee^{-\lambda t} dA_t^{\nu^+}\Big], \text{ and } W^\mu_\lambda\nu^-(x) = \bbE_x\Big[\int_0^t \ee^{-\lambda t} dA_t^{\nu^-}\Big].
\end{equation}
The process $A_t^\nu = A_t^{\nu^+}-A_t^{\nu^-}$ is continuous, adapted and of bounded variation. 
The following lemma demonstrates that $A_t^\nu$ is, in fact, the limit of $\int_0^t \nu_n(X_s)ds$ in the appropriate sense.
\begin{lemm}\label{lem:conv_A}
	Let $\nu \in K_{\alpha-1}^f$ and $(A^\nu_t)_{t \geq 0}$ be a continuous adapted process of finite variation such that 
	\begin{equation*}
		W_\lambda^\mu \nu (x)= \bbE_x\Big[\int_0^\infty\ee^{-\lambda t}dA^\nu_t \Big].
	\end{equation*} 
	Furthermore, let $(\nu_n)_{n \in \N}$ be a sequence of bounded $\mathcal{C}^\infty$ functions such that $W_\lambda^\mu\nu_n$ converges  to $W_\lambda^\mu\nu$ uniformly on $\R$. Then it follows that 
	\begin{equation}
		\lim_{n \rightarrow \infty}\sup_x\bbP_x\Big(\sup_{0 \leq t \leq T}|\int_0^t \nu_n(X_s)ds-A_t^\nu|> \epsilon\Big) = 0
	\end{equation}
	for all $\epsilon >0$ and all $T >0$.
\end{lemm}
The proof of the above lemma uses the same ideas as the proof of \cite[Proposition 5.6]{Kim14}. For the sake of completeness we provide it in Appendix~\ref{appendix}.\par
Now we are ready to show existence. 
\begin{lemm}\label{lem:weak_exist}
	Let $\mu \in K^f_{\eta}$ with $\eta \in (0,\alpha-1]$. Then there exists a weak solution to \eqref{eq:singularSDE}, which is a strong Markov process and $(A_t)_ {t\geq 0}$ is a continuous, adapted process of  finite variation. 
\end{lemm}
\begin{proof}
Because the transition densities from Theorem~\ref{th:transition_kernel} imply the existence of a conservative Feller process $(X_t, \bbP_x)$ and due to Lemma~\ref{lem:conv_A}, it only remains to verify that $(X_t-A_t-x_0)_{t \geq 0}$ is a symmetric, $\alpha$-stable process. This can be done by using
the convergence of the transition densities $p^{\mu_n}$ to $p^{\mu}$, which can be shown with the same arguments as in the proof of \cite[Theorem 3.9]{Kim14}  and the convergence of the $\lambda$-potentials  and identifying the distribution via the characteristic function. For the details we point to the proof of \cite[Proposition 5.6]{Kim14}. 
This proves the weak existence of a solution to \eqref{eq:singularSDE}, which is a strong Markov process and is such that $A$ is a continuous, adapted process  of finite variation. 
\end{proof}
\subsection{Analysis of $\lambda$-Potentials}
As mentioned above the $\lambda$-potentials $W_\lambda^\mu$ play a central role 
in the definition of Zvonkins transformation and in the proof of Proposition~\ref{lem:density}.
To this end we will derive the following formal equality
\begin{equation}\label{eq:pot_as_sum}
	W^{\mu}_\lambda\nu(x) = \sum_{k = 0}^\infty U_\lambda (N_{\mu }U_\lambda)^k\nu(x), \quad x \in \R, \nu \in K_{\alpha-1},
\end{equation} 
for $\lambda>0$ large enough, where $U_\lambda$ denotes the $\lambda$-potential 
\begin{equation*}
	U_\lambda f(x) = \int_0^\infty \ee^{-\lambda t} P_tf(x)dt, \quad t \geq 0, x \in \R, \quad f \in \mathcal{C}_b,
\end{equation*}
associated with the fractional Laplacian $\Delta_\alpha$ with transition semigroup $(P_t)_{t\geq 0}$ and where
\begin{equation*}
	N_{\mu}U_\lambda\nu(x) = \frac{\partial}{\partial x}U_\lambda\nu(x)\mu(dx).
\end{equation*}
Again, the action of both $(P_t)_{t \geq 0}$ and $U_\lambda$, $\lambda >0$ can also be extended to measures via 
\begin{equation*}
	P_t \nu(x) = \int_\R p(t,x,y) \nu(dy), 
\end{equation*}
whenever the integral with respect to the measure $\nu$ is well-defined. The density $u_\lambda$ of $U_\lambda$, 
is given by 
\begin{equation*}
	u_\lambda(x,y) = \int_0^\infty \ee^{-\lambda t } p(t,x,y)dt.
\end{equation*}
and satisfies 
\begin{equation*}
	U_\lambda f(x) = \int_\R u_\lambda(x,y)f(y) dy, \quad x\in \R , f \in \mathcal{C}_b.
\end{equation*}
We furthermore  use the notation
\begin{equation*}
	\Big|\frac{\partial}{\partial x} U_\lambda\Big||\nu|(x)= \int_\R \big|\frac{\partial}{\partial x}u_\lambda(x,y)\big||\nu|(dy),
\end{equation*}
for $\nu\in K_{\alpha-1}$, whenever the integral is well defined.
The identity \eqref{eq:pot_as_sum} can once again be derived with similar arguments as in \cite{Kim14} and heavily relies on the following estimates for the density $u_\lambda$ of $U_\lambda$.  
It was shown in \cite[Lemma 7, Lemma 9]{Bogdan07} that there exist constants $C_1,c_1 > 0$ such that for all $x, y \in \R$ 
\begin{equation}\label{eq:lamPot_est_bod}
	\begin{split}
		&c_1^{-1}\big((\lambda^{1/\alpha-1}\vee |x-y|^{\alpha-1})\wedge(\lambda^{-2}|x-y|^{-1-\alpha})\big)
		\leq u_\lambda(x,y)\\
		& \qquad \leq c_1\big((\lambda^{1/\alpha-1}\vee |x-y|^{\alpha-1})\wedge(\lambda^{-2}|x-y|^{-1-\alpha}) \big)
	\end{split}
\end{equation}
and 
\begin{equation}\label{eq:lamPot_grad}
	\begin{split}
		C_1^{-1}\big(|x-y|^{\alpha-2}\wedge \lambda^{-2}|x-y|^{-2-\alpha}\big) &\leq \big|\frac{\partial}{\partial x}  u_\lambda(x,y)\big| \\
		& \qquad \leq 
		C_1\big(|x-y|^{\alpha-2}\wedge \lambda^{-2}|x-y|^{-2-\alpha}\big).
	\end{split}
\end{equation}
Note that \eqref{eq:lamPot_est_bod} specifically implies
\begin{equation}\label{eq:lamPot_bound}
	|u_\lambda(x,y)| \leq c(\lambda),  \quad \forall x,y \in \R,
\end{equation}
for some constant $c(\lambda) >0$.
These estimates can be used to derive the following preliminary results, that are needed to derive \eqref{eq:pot_as_sum}. 

\begin{lemm}\label{lem:deriv_u_lam_conv_append}
	Let $\nu \in K_{\alpha-1}^f$ and let $\nu_m$, $m \in \mathbb{N}$ be its mollification. Then  for all $\lambda >0$ it holds that
	\begin{equation*}
		\frac{\partial}{\partial x}U_\lambda \nu_m(x) \rightarrow \frac{\partial}{\partial x}U_\lambda \nu(x),
	\end{equation*}
	uniformly on $\R$ as $m \rightarrow \infty$.
\end{lemm}
The proof of the above lemma is similar to the proof of  \cite[Lemma 3.5]{Kim14}. For the sake of completeness, we provide the proof in Appendix~\ref{appendix}.

The next to lemmas provide important estimates that will be used in several instances throughout this article.  They are slight modifications of \cite[Proposition 5.1, Proposition 5.2]{Kim14}.
Hereby the following constant, that depends on $\lambda$ plays a key role. 
Let $\delta \in (0,2/(2+\alpha))$ be fixed and define
\begin{equation}\label{eq:c_mu_lam}
	\mathcal{R}(\lambda, \nu) := C_1 \big(M_\nu^\alpha(\lambda^{-\delta})+|\nu|(\R)\lambda^{-2+\delta(2+\alpha)}\big),
\end{equation}
for $\nu \in K_{\alpha-1}^f$, where $C_1$ is the constant from \eqref{eq:lamPot_grad}. It is easy to see that 
\begin{equation*}
	\mathcal{R}(\lambda, \nu) \rightarrow 0,
\end{equation*}
as $ \lambda \rightarrow \infty$.
\begin{prop}\label{prop:grad_meas}$~$\par
	\begin{enumerate}
		\item Let $\nu \in K_{\alpha-1}^f$, then for all $\lambda > 0$, $U_\lambda\nu$ is a $\mathcal{C}^1$ function with 
		\begin{equation}
			\Big|\frac{\partial}{\partial x} U_\lambda \nu(x)\Big| \leq \Big|\frac{\partial}{\partial x} U_\lambda\Big||\nu| (x) 
			\leq \mathcal{R}(\lambda, \nu).
		\end{equation} 
		\item For all $\lambda >0$, there exists $C = C(\lambda)$ such that for every bounded measurable function $g$, 
		\begin{equation}
			\Big|\frac{\partial}{\partial x} U_\lambda g\Big| \leq \Big|\frac{\partial}{\partial x} U_\lambda \Big||g|(x) \leq C(\lambda) \|g\|_\infty, \quad \forall x \in \R
		\end{equation}
		where $C(\lambda) \rightarrow 0$ as $\lambda \rightarrow \infty$.
	\end{enumerate}
\end{prop}
\begin{proof}
	The proof is analogous to the proof of \cite[Proposition 5.1]{Kim14}, with only slight modifications to accommodate for the fact that $\mu$ is merely a finite measure and not necessarily a measure with  compact support. 
	Since $\frac{\partial}{\partial x}u_\lambda$ is integrable with respect to $\nu$ due to
	\eqref{eq:lamPot_grad} it follows by the Leibniz integral rule that
	\begin{equation*}
		\frac{\partial}{\partial x} U_\lambda \nu(x) = \int_{\R}\frac{\partial}{\partial x}u_\lambda(x,y) \nu(dy), \quad 	\frac{\partial}{\partial x} U_\lambda g(x) = \int_{\R}\frac{\partial}{\partial x}u_\lambda(x,y) g(y)dy,
	\end{equation*}
	and thus as a consequence of \eqref{eq:lamPot_grad} both $\frac{\partial}{\partial x} U_\lambda \nu(x)$ 
	and $\frac{\partial}{\partial x} U_\lambda g(x)$ are continuous.
	Using \eqref{eq:lamPot_grad} again we obtain
	\begin{align*}
		&\Big|\int_{\R}\frac{\partial}{\partial x}u_\lambda(x,y) \nu(dy) \Big|\leq C_1 \Big(\int_\R| x-y|^{\alpha-2}\wedge \lambda^{-2}|x-y|^{-2-\alpha}|\nu|(dy)\Big)\\
		&\leq   C_1 \Big(\int_{B(x,\lambda^{-\delta})}| x-y|^{\alpha-2}|\nu|(dy)+ \int_{B(x,\lambda^{-\delta})^c}\lambda^{-2}|x-y|^{-2-\alpha}|\nu|(dy)\Big)\\
		& \leq C_1\Big(M_{\nu}^\alpha(\lambda^{-\delta} ) + c({\nu})\lambda^{-2+\delta(2+\alpha)}\Big),
	\end{align*}
	and 
	\begin{align*}
		\Big|\frac{\partial}{\partial x} U_\lambda g(x)\Big| \leq \|g\|_{\infty} \int_\R \big|\frac{\partial}{\partial x} u_\lambda(x,y)\big|dy \leq C(\lambda) \|g\|_{\infty}
	\end{align*}
	This finishes the proof.
\end{proof}

\begin{prop}\label{prop:total_var_Nmu}
	Suppose that $\mu, \nu \in K_{\alpha-1}^f$.Then
	$\tilde{\nu} := N_{\mu}U_\lambda \nu$
	and 
	$\tilde{\nu}_n := N_{\mu_n} U_\lambda\nu $ are in $K_{\alpha -1}$. Moreover, for all $r>0$, $\lambda >0$,
	\begin{equation}
		M^\alpha_{\tilde{\nu}_n}(r) \leq \mathcal{R}(\lambda,\nu) M_\mu^\alpha(r), \quad M^\alpha_{\tilde{\nu}}(r) \leq \mathcal{R}(\lambda,\nu) M_\mu^\alpha(r)
	\end{equation}
	and 
	\begin{align*}
		|\tilde{\nu}|(dx) \leq \mathcal{R}(\lambda,\nu) |\mu|(dx), \quad 	|\tilde{\nu}_n|(dx) \leq \mathcal{R}(\lambda,\nu) |\mu_n|(dx).
	\end{align*}
\end{prop}
\begin{proof}
	This is an immediate consequence of Proposition~\ref{prop:grad_meas}.
\end{proof}

Finally, we can derive \eqref{eq:pot_as_sum}.
\begin{prop}\label{prop:pot_sum}
	For all measures $\nu \in K_{\alpha-1}^f$ and all $g \in \mathcal{C}_b$ there exists $\lambda_0 > 0$ such that
	for all $\lambda \geq \lambda_0$ it holds that
	\begin{equation}\label{eq:potential_sum_nu}
		W^{\mu}_\lambda\nu(x) = \sum_{k = 0}^\infty U_\lambda (N_{\mu }U_\lambda)^k\nu(x), \quad \forall x \in \R,
	\end{equation}
	and
	\begin{equation}\label{eq:potential_sum_g}
		W^{\mu}_\lambda g(x) = \sum_{k = 0}^\infty U_\lambda (N_{\mu }U_\lambda)^kg(x), \quad \forall x \in \R.
	\end{equation}
\end{prop}
\begin{proof}
	The proof essentially follows along the same lines as the proof of \cite[Proposition 5.5]{Kim14}. 
	It is only necessary to make some adaptations to account for the fact that we allow for more general measures $\mu$ by choosing $\lambda_0$ 
	large enough and replacing Lemma 3.5, Lemma 4.7,  Proposition 5.1, Proposition 5.2 from \cite{Kim14} by 
	Lemma~\ref{lem:deriv_u_lam_conv_append}, Lemma~\ref{lem:pot_conv},  Proposition~\ref{prop:grad_meas} and Proposition~\ref{prop:total_var_Nmu}.
\end{proof}
The proposition above also suggests that $W_\lambda^\mu$ has a density $w_\lambda$. In fact,  following the arguments from \cite{Bogdan07} yields the following lemma.
\begin{lemm}\label{lem:pot_mu_dens}
	There exists $\lambda_0>0$ such that for all $\lambda \geq \lambda_0$
	the operator $W_\lambda$ has the kernel 
	\begin{equation*}
		w_\lambda^\mu(x,y) =\sum_{k = 0}^\infty (u_\lambda \ast (\frac{\partial}{\partial x}u_\lambda \mu)^{\ast k})(x,y) , \quad x\neq y, x,y \in \R.
	\end{equation*}
\end{lemm}
\begin{proof}
	Follow the proof in \cite[p.191-192]{Bogdan07} line by line, while only making trivial adaptations where needed.
\end{proof}

In the next lemma we derive estimates for $W^\mu_\lambda\mu$. This map will be used 
in the Zvonkin transformation in Section~\ref{sec:pathwise}.
\begin{lemm}\label{lemm:wlmm_bound_deriv}
	There exists $\lambda_0>0$ such that for all $\lambda \geq  \lambda _0$ the following holds. There exist $C(\lambda)>0$ such that
	\begin{equation*}
		\|(W_{\lambda}^{\mu}\mu) '\|_\infty + \|W_{\lambda}^{\mu}\mu \| _\infty \leq C(\lambda),
	\end{equation*}
	where $C(\lambda) \rightarrow 0$ as $\lambda \rightarrow \infty$.
\end{lemm}
\begin{proof}
	  We will only prove, that there exists $\lambda_0 > 0$ such that 
	\begin{equation*}
		\|\frac{\partial}{\partial x}W_\lambda^\mu\mu(x)\|_\infty\leq  C(\lambda), \quad \forall \lambda \geq \lambda_0,
	\end{equation*}
	because the assertion for $\|w_\lambda^{\mu,\mu}\|_\infty$ follows along the same lines.
	Note that for $\lambda>0$ sufficiently large
	\begin{equation}\label{eq:deriv_pot_sum}
		\frac{\partial}{\partial x}W_\lambda^\mu\mu(x) = \frac{\partial }{\partial x} \sum_{k = 0}^\infty U_\lambda (N_\mu U_\lambda)^k\mu(x),
	\end{equation}
	by Proposition~\ref{prop:pot_sum}.
	It can be easily seen, with the same arguments as in the proof of \cite[Proposition 5.5 ]{Kim14}, that $(N_\mu U_\lambda)^k\mu$, $k \geq 1$ 
	is a measure of Kato class $K_{\alpha-1}$ with total variation measure bounded by $C(\lambda,\mu)^k|\mu|(dx)$, 
	where $C(\lambda, \mu) = C \mathcal{R}(\lambda, \mu)$ and $\mathcal{R}(\lambda,\mu)$ is given by \eqref{eq:c_mu_lam}. Thus, for $k \geq 1$ we have
	\begin{equation}\label{eq:interderiv_est}
		\begin{split}
			&\Big|\frac{\partial}{\partial x}U_\lambda (N_\mu U_\lambda)^k\mu(x)\Big|\\
			& \leq \int_\R \big|\frac{\partial}{\partial x} u_\lambda(x,y)\big||(N_\mu U_\lambda)^k\mu|(dy)\\
			& \leq C(\lambda, \mu)^k \int_\R \big|\frac{\partial}{\partial x} u_\lambda(x,y)\big||\mu|(dy).
		\end{split}
	\end{equation}
	Now, fix $\lambda_0 > 0$ such that \eqref{eq:deriv_pot_sum} holds and $C(\lambda,\mu)<1$, for any $\lambda \geq \lambda_0$. Thus, \eqref{eq:interderiv_est} together with \eqref{eq:lamPot_grad} demonstrates that there exists $C(\lambda)>0$ such that
	\begin{align*}
		\sup_{x \in \R} \Big| W_\lambda^\mu\mu (x)\Big| \leq \sum_{k = 0}^\infty C(\lambda, \mu)^k \sup_{x \in \R}\int_\R \big|\frac{\partial}{\partial x} u_\lambda(x,y)\big||\mu|(dy) \leq C(\lambda), \lambda \geq \lambda_0,
	\end{align*}
	where $C(\lambda) \rightarrow 0$ as $\lambda \rightarrow \infty$.
\end{proof}
In the next lemma we prove the uniform convergence of the derivative of $W_\lambda^{\mu_n}\mu_n$ to the derivative of  $W_\lambda^{\mu}\mu$.
\begin{lemm}\label{lem:deriv_pot_conv}
	There exists $\lambda_0 > 0$ such that for all $\lambda \geq \lambda_0$
	\begin{align*}
		\|(w_{\lambda}^{\mu,\mu} )'-(w_{\lambda}^{\mu_n,\mu_n})'\|_\infty \rightarrow 0, \quad \text{ as } n \rightarrow \infty.
	\end{align*}
\end{lemm}
\begin{proof}
	First note that for all $\lambda >0$ large enough and all $x \in \R$
	\begin{equation}\label{eq:grad_conv_mumu}
		\begin{split}
			&\Big|\frac{\partial}{\partial x}\Big(W_\lambda^{\mu_n}\big(\mu_n-\mu\big)(x)\Big)\Big|\\
			& = \Big|\frac{\partial}{\partial x}\sum_{k = 0}^\infty U_\lambda(N_{\mu_n}U_\lambda)^k(\mu_n-\mu)(x)\Big|,
		\end{split}
	\end{equation}
	by Proposition~\ref{prop:pot_sum}.
	Note that by Proposition~\ref{prop:grad_meas}
	\begin{align*}
		&\Big|\frac{\partial}{\partial x}U_\lambda(N_{\mu_n}U_\lambda)(\mu_n-\mu)(x)\Big|\\
		&\leq \int_\R \big|\frac{\partial}{\partial x}u_\lambda(x,y)\big|\big|(\frac{\partial}{\partial x}U_\lambda)(\mu_n-\mu)(y)\big||\mu_n|(dy)\\
		& \leq \sup_{y\in \R}\big|(\frac{\partial}{\partial x}U_\lambda)(\mu_n-\mu)(y)\big|\int_\R\big|\frac{\partial}{\partial x}u_\lambda(x,y)\big||\mu_n|(dy)\\
		&\leq C(\lambda,\mu) \sup_{y\in \R}\big|(\frac{\partial}{\partial x}U_\lambda)(\mu_n-\mu)(y)\big|,
	\end{align*}
	where $C(\lambda,\mu)$ can be chosen independently of $n$ and such that $C(\lambda,\mu) \rightarrow 0$ as $\lambda \rightarrow \infty$,
	since it is not hard to verify that $\sup_n\mathcal{R}(\lambda, \mu_n) \rightarrow 0$ as $\lambda \rightarrow \infty$, where $\mathcal{R}(\lambda, \mu_n) $ is given by \eqref{eq:c_mu_lam}. We choose 
	$\lambda_0$ such that $C(\lambda, \mu)<1$ for all $\lambda \geq \lambda_0$. Repeating the above, using that 
	\begin{equation*}
		\sup_{y\in \R}\big|(\frac{\partial}{\partial x}U_\lambda)(\mu_n-\mu)(y)\big| \rightarrow 0,
	\end{equation*}
	due to Lemma~\ref{lem:deriv_u_lam_conv_append} and noting that $\sum_{k = 0}^\infty C(\lambda,\mu)^{k-1}< \infty$ demonstrates the convergence of \eqref{eq:grad_conv_mumu} to zero.
	Since both sums 
	\begin{align*}
		\sum_{k = 0}^\infty \frac{\partial}{\partial x}U_\lambda (N_{\mu_n }U_\lambda)^k\mu(x), \\
		\sum_{k = 0}^\infty \frac{\partial}{\partial x}U_\lambda (N_{\mu }U_\lambda)^k\mu(x),
	\end{align*}
	converge absolutely, 
	it remains to prove
	\begin{equation}\label{eq:conv_of_pot_deriv}
		\begin{split}
			&\sum_{k = 0}^\infty \frac{\partial}{\partial x}U_\lambda (N_{\mu_n }U_\lambda)^k\mu(x) - \sum_{k = 0}^\infty \frac{\partial}{\partial x}U_\lambda (N_{\mu }U_\lambda)^k\mu(x)\\
			& = \sum_{k = 1}^\infty \sum_{ \ell = 0}^{k-1}\frac{\partial}{\partial x}U_\lambda (N_{\mu_n}U_\lambda)^{\ell} (N_{\mu_n}U_\lambda-N_{\mu}U_\lambda) (N_{\mu}U_\lambda)^{k-l-1}\mu \rightarrow 0, \quad \text{ as } n\rightarrow \infty .
			\end {split}
		\end{equation}
		This can be demonstrated with the same arguments as in the proof of \cite[Proposition 5.5]{Kim14}. 
	\end{proof}
\subsection{Weak Uniqueness and Proof of Theorem~\ref{th:sde_exist}\ref{itm:weakex}}\label{sec:weak_unique}
\begin{lemm}\label{lem:weak_unique}
	Let $\mu \in K^f_{\eta}$ with $\eta \in (0, \alpha-1]$. Then there exists a unique weak solution to \eqref{eq:singularSDE}.
\end{lemm}
\begin{proof}
	Existence is proven in Lemma~\ref{lem:weak_exist}. 
To prove weak uniqueness, we can  follow the same arguments as in \cite[Proposition 5.8]{Kim14} and we thus omit some details.
	Let $\mathbb{P}_x$ be the solution constructed in Lemma~\ref{lem:weak_exist} and
	 assume that $\tilde{\bbP}_x$ is another solution to \eqref{eq:singularSDE} and that $g \in \mathcal{C}^2$. 
	With the same steps as in \cite[Proposition 5.8]{Kim14} we arrive at
	\begin{equation}
			\bbE_{\tilde{P}_x}\Big[\int_0^\infty \ee^{-\lambda t} g(X_t) dt\Big] = U_\lambda g(x) + \bbE_{\tilde{\bbP}_x}\Big[ \int_0^\infty \ee^{-\lambda t} \frac{\partial }{\partial x}U_\lambda g(X_t)dA_t\Big].
		\end{equation}
	Note that $V_\lambda g (x): = \bbE_{\tilde{P}_x}\int_0^\infty \ee^{-\lambda t} g(X_t) dt $ is the $\lambda$-potential with respect to $\tilde{\bbP}$.
	Thus,
	\begin{equation}\label{eq:pot_unique}
			V_\lambda g (x)= U_\lambda g(x) + \bbE_{\tilde{\bbP}_x}\Big[ \int_0^\infty \ee^{-\lambda t} \frac{\partial }{\partial x}U_\lambda g(X_t)dA_t\Big].
		\end{equation}
	By taking limits, it turns out that \eqref{eq:pot_unique} holds for all continuous, bounded functions $g$. 
	Using the convergence of $\int_0^t \mu_n(X_s)ds$ towards $A_t$ allows us to identify
	\begin{align*}
			&\bbE_{\tilde{\bbP}_x}\Big[ \int_0^\infty \ee^{-\lambda t} \frac{\partial }{\partial x}U_\lambda g(X_t)dA_t \Big]\\
			&= U_\lambda N_\mu U_\lambda g(x) + \bbE_{\tilde{\bbP}_x} \Big[\int_0^\infty \ee^{-\lambda t} \frac{\partial }{\partial x}U_\lambda( N_\mu U_\lambda)g(X_t)dA_t\Big].
		\end{align*}
	Iterating this process eventually yields 
	\begin{equation*}
			V_\lambda g(x) = \sum_{k = 0}^\infty U_\lambda (N_\mu U_\lambda)^k g(x) = W^\mu_\lambda g(x) \quad \forall g \in \mathcal{C}_b(\R), 
		\end{equation*}
	where the last equality is due to Proposition~\ref{prop:pot_sum}. Thus the $\lambda$-potentials for both solutions are equal for all $\lambda \geq \lambda_0$ an this implies the weak uniqueness.
\end{proof}
\begin{proof}[Proof of Theorem~\ref{th:sde_exist}\ref{itm:weakex}]
	Immediate consequence of Lemma~\ref{lem:weak_exist} and Lemma~\ref{lem:weak_unique}.
\end{proof}
\subsection{Pathwise Uniqueness and Proof of Theorem~\ref{th:sde_exist}\ref{itm:stronex}}\label{sec:pathwise}
In this section, we will prove Theorem~\ref{th:sde_exist}~\ref{itm:stronex}. First, we use techniques from \cite{Athreya20}, \cite{Xiong19} and \cite{Knopova17} to prove pathwise uniqueness of \eqref{eq:singularSDE} 
when $\mu\in K_{\eta}^f$ for $\eta < \alpha-1$.  For the remainder of this section we fix $\mu \in  K_{\eta}^f$ for some fixed $\eta < \alpha-1$.\\
By Lemma~\ref{lem:KatoBesov}~\ref{itm:kato_bes_non} the conditions on $\mu$ in particular imply
$|\mu| \in 
\mathcal{C}^{1-\alpha+\epsilon}$  for some $\epsilon >0$ with $1-\alpha+\epsilon<0$. We fix this $\epsilon >0$ for the remainder of this section. 
The observation above allows us 
to employ the techniques from \cite{Athreya20} more easily. 
This includes analyzing  the 
resolvent equation
\begin{equation}\label{eq:resolvent}
	\lambda w- \Delta_{\alpha} w- \mu \frac{\partial}{\partial x}w = \mu.
\end{equation}
Let $w_\lambda^{\mu,\mu}$ denote a solution to \eqref{eq:resolvent}. Later we will show that such a solution exists and is unique when $\lambda$ is sufficiently large.
Then we will use the map 
\begin{equation}\label{eq:zvonkin_transform}
	z \mapsto \phi(z) = w_\lambda^{\mu,\mu}(z) +z,
\end{equation}
where $w_\lambda^{\mu,\mu}$ is a solution to \eqref{eq:resolvent}, to transform the equation  \eqref{eq:singularSDE}. 
The idea here is that $\phi(X)$, where $X$ is a solution to \eqref{eq:singularSDE}, satisfies a different SDE which has more regular coefficients.\par 
First, we need to prove that \eqref{eq:resolvent} has a unique solution in the distributional sense. 
Since $(P^\mu_t)_{t \geq 0}$ is a Feller semigroup with generator $(\Delta_\alpha + \mu \frac{\partial}{\partial x})$,
a natural candidate is given by $w_\lambda^{\mu,\mu} :=W_\lambda^\mu \mu$. We will show that $w_\lambda^{\mu,\mu}$ 
is in fact a solution in the next lemma.
To this end, we use similar arguments as in \cite{Athreya20}  and thus show existence and uniqueness of solutions in $\mathcal{C}^{1+\epsilon/2}$, with $\epsilon >0$. Recall that $\epsilon >0$ was fixed at the beginning of this section.
	\begin{lemm}\label{lem:boundw_lmm}
		There exists $\lambda_0 > 0$ such that there exists a unique solution of \eqref{eq:resolvent} in $\mathcal{C}^{1+\epsilon/2}$ for all $\lambda \geq\lambda_0$.
	\end{lemm}
	\begin{proof}
		The proof of uniqueness essentially follows the same steps as the proof of \cite[Proposition 2.7]{Athreya20}. Choose arbitrary $\lambda >0$.
		Let $w_1,w_2$ be two solutions of \eqref{eq:resolvent} and set $v = w_1-w_2$. Denote $g = v'\mu$. By setting $\gamma = 1+ \epsilon/2$ and $\beta = 1-\alpha-\epsilon$
		%	$\beta = 1-\alpha +\epsilon$ and choosing $\lambda_0 >1$ we get
		%	\begin{align*}
			%			&\|v\|_{\dot{B}_{\infty,\infty}^{\beta}} \leq \|U_\lambda g\|_{\dot{B}_{\infty,\infty}^{\beta}} \\
			%			& \leq C\sup_{t \in 0,\infty)}\|t^{-\beta/2}\ee^{t \Delta}U_\lambda g\|_{\infty}\\
			%			& \leq C\|g\|_{\dot{B}_{\infty,\infty}^{\beta}} \int_\R |u_\lambda(|z|)|dz\\
			%			& \leq C\|g\|_{\dot{B}_{\infty,\infty}^{\beta}}  ( \lambda^{-2}  + \int_0^1 (\lambda^{1/\alpha-1}\vee |z|^{\alpha-1})\wedge\lambda^{-2}|z|^{-1-\alpha}dz)\\
			%			& \leq C\|g\|_{\dot{B}_{\infty,\infty}^{\beta}} (\lambda^{-2}+ \lambda^{1/\alpha-1}+\int_{\lambda^{\frac{1/\alpha-1}{\alpha-1}}}^1   |z|^{\alpha-1}\wedge\lambda^{-2}|z|^{-1-\alpha}dz)\\
			%			& \leq  C\|g\|_{\dot{B}_{\infty,\infty}^{\beta}} (\lambda^{-2}+ \lambda^{1/\alpha-1}+\int_{\lambda^{\frac{1/\alpha-1}{\alpha-1}}}^1   |z|^{\alpha-1}\wedge\lambda^{1/\alpha-1}dz)\\
			%			& \leq C\|g\|_{\dot{B}_{\infty,\infty}^{\beta}} \lambda^{1/\alpha-1},
			%	\end{align*}
		%	\begin{align*}
			%		&\|v\|_{\dot{B}_{\infty,\infty}^{\eta}} \leq \int_0^\infty \ee^{-\lambda t}\|P_t g\|_{\dot{B}_{\infty,\infty}^\eta}dt\\
			%		&= \int_0^1 \ee^{-\lambda t}\|P_t g\|_{\dot{B}_{\infty,\infty}^{\eta}}dt + \int_1^\infty  \ee^{-\lambda t}\|P_t g\|_{\dot{B}_{\infty,\infty}^{\eta}}dt \\
			%		& \leq C\|g\|_{\dot{B}_{\infty,\infty}^{\eta}} \int_0^\infty \ee^{-\lambda t}dt \\
			%		& \leq C \lambda^{-1} \|g\|_{\dot{B}_{\infty,\infty}^{\eta}},
			%	\end{align*}
		%	where $g = v'\mu$ and we used the notation $u_\lambda(x,y) = u_\lambda(|x-y|)$ and the estimate \eqref{eq:lamPot_est_bod}.
		and using \cite[Lemma 4.1]{Athreya20} we arrive at 
		\begin{equation}\label{eq:ineq_sol_res}
			\begin{split}
				\|v\|_{\mathcal{C}^{\gamma}}
				\leq C \lambda^{-1-\beta/\alpha + \gamma/\alpha} \|g\|_{\mathcal{C}^{\beta}}.
			\end{split}
		\end{equation}
		
		Now, note that by the characterization \eqref{eq:besov_char} of the Besov norm, we get for a measure $\nu$, such that $|\nu| \in \mathcal{M}\cap \mathcal{C}^{-s}$, $s >0$, that 
		\begin{equation*}
			\|f\nu\|_{\mathcal{C}^{-s}} \leq C\|f\|_{L^\infty} \big\||\nu|\big\|_{\mathcal{C}^{-s}}.
		\end{equation*}
		Thus, by also using \cite[Lemma 3.1 (iii)]{Athreya20} we arrive at
		\begin{align*}
			&\| v' \mu\|_{\mathcal{C}^{\beta}}\\
			& \leq C\|v\|_{\mathcal{C}^\gamma}   \big\|  |\mu| \big\|_{\mathcal{C}^{\beta}}.
		\end{align*}
		By plugging in $g = v'\mu$ we arrive at 
		\begin{equation}\label{eq:estimate_besnorm}
			\|v\|_{\mathcal{C}^{\gamma}}\\
			\leq C(\mu)\big( \lambda^{-1-\beta/\alpha + \gamma/\alpha}\vee \lambda^{\frac{1-\alpha}{\alpha}}\big) \|v\|_{\mathcal{C}^{\gamma}}.
		\end{equation}
		For all $\lambda >0$ large enough, it follows that $ C(\mu)\big( \lambda^{-1-\eta/\alpha + \gamma/\alpha}\vee \lambda^{\frac{1-\alpha}{\alpha}}\big) < 1$ and thus $v = 0$. \par
		Now, we prove existence. Since $(P_t^\mu)_{t \geq 0}$ is a conservative Feller semigroup, it follows that $w_n :=W_\lambda^\mu\mu_n$ solves 
		\begin{equation}
			\lambda w_n-\Delta_\alpha w_n - \mu \frac{\partial}{\partial x}w_n = \mu_n,
		\end{equation}
		in the distributional sense. Due to Lemma~\ref{lem:pot_conv} and by the proof of Lemma~\ref{lem:deriv_pot_conv} it follows that $w_n$ converges to $w_\lambda^{\mu,\mu}$ uniformly on $\R$ and for $\lambda \geq \lambda_0$ it follows that 
		$(w_n)'$ converges to $(w_\lambda^{\mu,\mu})$ uniformly on $\R$. Furthermore, since $\mu \in \mathcal{C}^{1-\alpha + \epsilon}$ it is not hard to verify that $\lim_{n \rightarrow \infty}\| \mu- \mu_n\|_{\mathcal{C}^{1-\alpha}}$.
		Thus, it holds 
		\begin{equation}\label{}
			\begin{split}
				&\|\lambda w_\lambda^{\mu,\mu}-\Delta_\alpha w_\lambda^{\mu,\mu} - \mu (w_\lambda^{\mu,\mu})' - \mu \|_{\mathcal{C}^{1-\alpha}}\\
				&\leq C\Big(\lambda \|w_\lambda^{\mu,\mu} - w_n\|_\infty +\|\Delta_\alpha (w_\lambda^{\mu,\mu}-w_n) \|_{\mathcal{C}^{1-\alpha}}+ \|w_n'-(w_\lambda^{\mu,\mu})'\|_\infty\| |\mu|\|_{\mathcal{C}^{1-\alpha}}+ \|\mu_n-\mu\|_{\mathcal{C}^{1-\alpha}}\Big)\\
				& \leq C\Big(\|w_\lambda^{\mu,\mu} - w_n\|_\infty + \|w_n'-(w_\lambda^{\mu,\mu})'\|_\infty\|  |\mu|\|_{\mathcal{C}^{1-\alpha}} + \|\mu_n-\mu\|_{\mathcal{C}^{1-\alpha}}\Big) \rightarrow 0, \quad n \rightarrow \infty,
			\end{split}
		\end{equation}
		by using \cite[Lemma 3.1, Lemma 3.2]{Athreya20} and \cite[Proposition 1, Section 2.5.7]{Triebel10}.
		This demonstrates, that $w_\lambda^{\mu,\mu}$ is indeed a solution of \eqref{eq:resolvent}. By employing \eqref{eq:ineq_sol_res} it follows with $\beta = 1-\alpha+\epsilon$, $\gamma = 1+ \epsilon/2$, $\rho = \epsilon/2$ and $g = \mu + \mu (w_\lambda^{\mu,\mu})'$ that 
		\begin{align*}
			&\|w_\lambda^{\mu,\mu}\|_{\mathcal{C}^{1+\epsilon/2}} \\
			& \leq C\big(\lambda^{-1-\frac{\beta}{\alpha}+ \frac{\gamma}{\alpha}}\vee \lambda^{\frac{1-\alpha}{\alpha}}\big)\||\mu|\|_{\mathcal{C}^{\beta}}\Big(1+\|w_\lambda^{\mu,\mu}\|_{\mathcal{C}^{1+\epsilon/2}} \Big).
		\end{align*} 
		Thus $w_\lambda^{\mu,\mu} \in C^{1+\epsilon/2}$ for $\lambda \geq \lambda_0$.
	\end{proof}
	From now on we always assume that $\lambda_0 >0$ is large enough such that \eqref{eq:resolvent} has a unique solution in $\mathcal{C}^{1+\epsilon/2}$ for all $\lambda \geq \lambda_0$ and we denote this unique solution by $w_\lambda^{\mu,\mu}$.
	In order to derive the transformed equation, we have to approximate $\mu$ by $\mu_n$
	so that we can apply It\^o's formula to $w_\lambda^{\mu_n,\mu_n}$. In a second step we then show 
	that all terms converge in an appropriate sense. To this end we derive the following lemma. 
	\begin{lemm}\label{lem:conv_riem}
		Let $(X_t)_{t \geq 0}$ be a solution to \eqref{eq:singularSDE} with $\mu \in K_{\eta}^f$ and $\eta < \alpha-1$. 
		Furthermore, let $(f_n)_{n \geq 1}  \subseteq \mathcal{C}_b$ converge uniformly to $f \in  \mathcal{C}_b$. Then  there exists a subsequence $(n_k)_{k \in \N}$ such that
		\begin{equation*}
			\int_0^t f_{n_k}(X_s)dA_s - \int_0^t f_{n_k}(X_s)\mu_{n_k}(X_s)ds \rightarrow 0,
		\end{equation*}
		almost surely as $k \rightarrow \infty$, where $A_t = X_t-x_0-L_t$, for $t \geq 0$.
	\end{lemm}
	\begin{proof}
		First, we pass to a sub-sequence $(n_k)_{k \geq 1}$ such that 
		\begin{equation}\label{eq:bounded_var}
			\sup_{n_k}\int_0^t|\mu_{n_k}(X_s)|ds < \infty, \quad \bbP\text{-a.s.}
		\end{equation}
		and such that 
		\begin{equation}\label{eq:uniform_conv}
			\lim_{k \rightarrow 0}\sup_{0 \leq s \leq t}|\int_0^s\mu_{n_k}(X_\tau)d\tau -A_s| = 0,  \quad \bbP\text{-a.s.}
		\end{equation}
		Such a sub-sequence exists by Definition~\ref{def:sing_sol} of a solution. 
		For the sake of a more concise notation, we suppose that we already passed to this sub-sequence and denote it by $(n)_{n \geq 1}$ from here on. 
		Now, consider the following 
		\begin{equation*}
			\begin{split}
				&	\Big| \int_0^t f_n(X_s)dA_s- \int_0^t f_n(X_s)\mu_n(X_s)ds\Big| \\
				&	\leq 	\Big| \int_0^t f_n(X_s)dA_s- \int_0^t f(X_s)dA_s\Big| + 	\Big| \int_0^t f(X_s)dA_s- \int_0^t f(X_s)\mu_n(X_s)ds\Big| \\
				&\quad + 	\Big| \int_0^t f(X_s)\mu_n(X_s)ds- \int_0^t f_n(X_s)\mu_n(X_s)ds\Big| \\
				& =: E_1^n+E_2^n+E_3^n.
			\end{split}
		\end{equation*}
		The term $E_1^n$ converges almost surely to zero as $n \rightarrow \infty$ by the dominated convergence theorem. Furthermore, it holds almost surely that
		\begin{equation*}
			E_3^n \leq \| f -f_n\|_\infty \sup_{n\geq 1} \int_0^t |\mu_n|(X_s)ds \rightarrow 0, \quad n \rightarrow \infty, 
		\end{equation*}
		by \eqref{eq:bounded_var}. Thus it remains to prove the almost sure convergence of $E_2^n$ to zero as 
		$n \rightarrow \infty$.  From \eqref{eq:bounded_var} and \eqref{eq:uniform_conv} it follws that there exists $\tilde{\Omega}\subseteq \Omega$ such that $\bbP(\tilde{\Omega}) = 1$,
		and $\mu_n(X_s(\omega))ds$ converges weakly to $dA_s(\omega)$ on $[0,t]$ for all $\omega \in \tilde{\Omega}$. Since $A$ is almost surely continuous and $X$ has c\`adl\`ag paths it 
		follows that the set of discontinuities of the bounded function $f(X_s)$ is almost surely of measure zero with respect to $dA$. Thus it follows that $E_2^n \rightarrow 0$ almost surely, as $n \rightarrow \infty$. 
		This concludes the proof. 
		
	\end{proof}

	Now, we are ready to derive the transformed equation.
	\begin{lemm}\label{lemm:trans_eq}
		Let $(X_t)_{t \geq 0}$ be a solution to \eqref{eq:singularSDE} with $\mu \in K_{\eta}^f$ and $\eta < \alpha-1$.
		Then there exists $\lambda_0 > 0$ such that for all $\lambda \geq \lambda_0$
		\begin{equation}\label{eq:trams_eq_X}
			\begin{split}
				&w_{\lambda}^{\mu,\mu} (X_t) + X_t = w_{\lambda}^{\mu,\mu} (x) + x +\lambda \int_0^t w_{\lambda}^{\mu,\mu} (X_s)ds\\
				& \qquad  +\int_0^t\int_\R \big(w_{\lambda}^{\mu,\mu} (X_{s-}+r)-w_{\lambda}^{\mu,\mu} (X_{s-}) \big)\big(\Pi-\pi\big)(ds,dr) +L_t ,
			\end{split}
		\end{equation}
		for all $t \geq 0$, $\bbP$-almost surely.
	\end{lemm}
	\begin{proof}
		First note that $\mu\in K_{\eta}^f$ for $\eta < \alpha-1$ implies that $|\mu| \in \mathcal{C}^{1-\alpha+\epsilon}$
		for some fixed $\epsilon >0$, by Lemma~\ref{lem:KatoBesov}. Furthermore, 
		for the sake of a more concise notation, we denote $w_n$ to be the solution to 
		\begin{equation}\label{eq:smoothed_res}
			\lambda w_n -\Delta_{\alpha}w_n - \mu_n (w_n)' = \mu_n,
		\end{equation}
		where, as usual, $\mu_n = \varphi_n \ast \mu$. An application of It\^o's Lemma and \eqref{eq:smoothed_res} yield
		\begin{equation}\label{eq:Ito_wn}
			\begin{split}
				w_n(X_t)& = w_n(x) + \lambda \int_0^t w_n(X_s)ds - \int_0^t \mu_n(X_s)ds \\
				& + \int_0^t (w_n)'(X_s)dA_s - \int_0^t (w_n)'(X_s)\mu_n(X_s)ds \\
				& + \int_0^t \int_\R \big[w_n(X_{s-}+r)-w_n(X_{s-})\big]\big(\Pi-\pi\big)(ds,dr), \quad t \geq 0, \quad \bbP\text{-a.s.}
			\end{split}
		\end{equation}
		Now we fix arbitrary $t \geq 0$.
		We first show that there exists a subsequence such that   $\int_0^t (w_n)'(X_s)dA_s - \int_0^t (w_n)'(X_s)\mu_n(X_s)ds$ converges to zero in probability. 
		By Lemma~\ref{lem:conv_riem} it is enough to show that 
		$w_n \rightarrow w^{\mu.\mu}$ uniformly, which follows from Lemma~\ref{lem:deriv_pot_conv}. \par
		%Thus it remains to prove the boundedness of the sequence in $\mathcal{C}^{\epsilon/2}$.
		%		In fact, using \eqref{eq:ineq_sol_res} with $\beta = 1-\alpha+\epsilon$, $\gamma = 1+ \epsilon/2$, $\rho = \epsilon/2$ and $g = \mu_n + \mu_n w_n'$ we get 
		%		\begin{align*}
			%			&\|w_n\|_{\mathcal{C}^{1+\epsilon/2}} \\
			%			& \leq C\big(\lambda^{-1-\frac{\beta}{\alpha}+ \frac{\gamma}{\alpha}}\vee \lambda^{\frac{1-\alpha}{\alpha}}\big)\||\mu_n|\|_{\mathcal{C}^{\beta}}\Big(1+\|w_n\|_{\mathcal{C}^{1+\epsilon/2}} \Big).
			%		\end{align*}
		%		Now, we use the characterization \eqref{eq:besov_char} of the Besov norm to get
		%		\begin{equation*}
			%		\begin{split}
				%			&\||\mu_n|\|_{\mathcal{C}^{-s}} \leq C \sup_{t \in (0,1]}\|t^{s/2}\ee^{t \Delta}|\mu_n|\|_{L^\infty}\\
				%			& \leq C \sup_{t \in (0,1]}\sup_{x\in \R}t^{s/2}\Big|\int_\R h(t,x,y)|\mu_n|(dy)\Big|\\
				%			& \leq C \int_\R \phi_n(z) \sup_{t \in (0,1]} \sup_{x\in \R}\int_\R t^{s/2}h(t,x+z,y)|\mu|(dy) dz\\
				%			& \leq C \| |\mu|\|_{\mathcal{C}^{-s}},
				%		\end{split}
			%		\end{equation*}
		%		for $s >0$.
		%		Thus for $\lambda >0$ large enough
		%		\begin{equation*}
			%			\sup_n\|(w_n)'\|_{\mathcal{C}^{\epsilon/2}}\leq 	C\sup_n\|w_n\|_{\mathcal{C}^{1+\epsilon/2}} < C_1,
			%		\end{equation*}
		%		since $\||\mu|\|_{\mathcal{C}^{\beta}}< \infty$. 
		Now, we set 
		\begin{equation*}
			f_n(s,r,\omega) = w^{\mu,\mu}_\lambda(X_{s-}+r)-w^{\mu,\mu}_\lambda(X_{s-})-w_n(X_{s-}+r)+w_n(X_{s-}),\quad s \geq 0
		\end{equation*}
		and note that 
		\begin{equation*}
			|f_n(s,r,\omega)| \leq C\big(\|w_n'-(w^{\mu,\mu}_\lambda)'\|_\infty |r|\big)\wedge C\|w_n-w^{\mu,\mu}_\lambda\|_{\infty}, \quad \forall s\geq 0,\quad  r \in \R.
		\end{equation*}
		By applying \cite[Lemma 5.1]{Athreya20} to $f_n$ with $C_{f_n} = C \max(\|w_n'-(w^{\mu,\mu}_\lambda)'\|_\infty ,\|w_n-w^{\mu,\mu}_\lambda\|_{\infty} )$ and with Lemma~\ref{lem:pot_conv} and Lemma~\ref{lem:deriv_pot_conv}
		it follows that, as $n \rightarrow \infty$,
		\begin{align*}
			&\int_0^t \int_\R \big[w_n(X_{s-}+r)-w_n(X_{s-})\big] \big(\Pi-\pi\big)(ds,dr) \\
			& \qquad \rightarrow \int_0^t \int_\R \big[w_{\lambda}^{\mu,\mu} (X_{s-}+r)-w_{\lambda}^{\mu,\mu} (X_{s-})\big] \big(\Pi-\pi\big)(ds,dr)
		\end{align*}
		in probability.  The terms $w_n(X_t)$ and $\int_0^t w_n(X_s) ds$ also converge $\bbP$-almost surely to
		$w_{\lambda}^{\mu,\mu} (X_t)$ and $\int_0^t w_{\lambda}^{\mu,\mu} (X_s)ds$,  respectively.
		Now, we pass to a subsequence $(n_k)$ such that all the convergences discussed above hold and rewrite \eqref{eq:Ito_wn} as
		\begin{equation}\label{eq:It_w_n_sub}
			\begin{split}
				w_{n_k}(X_t) +X_t& = w_{n_k}(x) + \lambda \int_0^t w_{n_k}(X_s)ds +X_t- \int_0^t \mu_{n_k}(X_s)ds \\
				& + \int_0^t (w_{n_k})'(X_s)dA_s - \int_0^t (w_{n_k})'(X_s)\mu_{n_k}(X_s)ds \\
				& + \int_0^t \int_\R \big[w_{n_k}(X_{s-}+r)-w_{n_k}(X_{s-})\big]\big(\Pi-\pi\big)(ds,dr).
			\end{split}
		\end{equation}
		It is then evident, that the right hand side and the left hand side of \eqref{eq:It_w_n_sub} converge in probability to the right and left hand side of
		\eqref{eq:trams_eq_X}, respectively. Since all the processes in \eqref{eq:trams_eq_X} are c\`adl\`ag, we get that \eqref{eq:trams_eq_X} holds $\bbP$-almost surely for all $t \geq 0$. This concludes the proof.
	\end{proof}
	With the help of the above lemmas, we are able to derive transformed equation for which we can establish pathwise uniqueness with tools from \cite{Knopova17} and \cite{Xiong19}.
	To this end set 
	\begin{align*}
		&b(z)= w_{\lambda}^{\mu,\mu} (\phi^{-1}(z)), \\
		&\sigma(z,r) = w_{\lambda}^{\mu,\mu} (\phi^{-1}(z)+r)-w_{\lambda}^{\mu,\mu} (\phi^{-1}(z)),
	\end{align*}
	where $\phi$ is defined as in \eqref{eq:zvonkin_transform}.
	Setting $W_t = \phi(X_t)$ yields
	\begin{equation}\label{eq:trans_eq}
		W_t = \phi(x) + \lambda\int_0^t b(W_s)ds + \int_0^t\int_\R \sigma(W_{s-},r)\big(\Pi-\pi\big)(ds,dr) +L_t, \quad t \geq 0.
	\end{equation}
	It turns out that the transformation $\phi$ is invertible and its inverse is Lipshitz continuous.  
	\begin{lemm}
		There exists $\lambda_0 > 0$ such that for all $\lambda \geq \lambda_0$ the map $\phi$
		is strictly monotone increasing and satisfies
		\begin{equation}\label{eq:lip_phi}
			(1-\epsilon_0(\lambda)) |x-y| \leq  |\phi(x)-\phi(y)| \leq (1+\epsilon_0(\lambda))|x-y|, \quad x,y\in \R,
		\end{equation} 
		where $\epsilon_0(\lambda) \rightarrow 0$ as $\lambda \rightarrow 0$. 
		Furthermore, $\phi$  has a  strictly monotone increasing inverse $\phi^{-1}$, satisfying 
		\begin{equation}\label{eq:lip_phi_inv}
			(1+\epsilon_0(\lambda))^{-1} |x-y| \leq  |\phi^{-1}(x)-\phi^{-1}(y)| \leq (1-\epsilon_0(\lambda))^{-1}|x-y|,  \quad x,y\in \R.
		\end{equation}
	\end{lemm}
	\begin{proof}
		The strict monotonicity of $x \mapsto \phi(x) = w_{\lambda}^{\mu,\mu}(x) +x$ readily follows from the fact that 
		\begin{equation}\label{eq:bound_w_lam_prime}
			\|(w_{\lambda}^{\mu,\mu})'\|_\infty \leq \epsilon_0(\lambda) < 1,
		\end{equation}
		for $\lambda$ large enough by Lemma~\ref{lemm:wlmm_bound_deriv}. 
		Using the same estimate, one can easily derive \eqref{eq:lip_phi} and \eqref{eq:lip_phi_inv}.
	\end{proof}
	Before we verify the pathwise uniqueness, we  first show that solutions to \eqref{eq:trans_eq}
	are weakly unique.
	
	\begin{lemm}\label{lemm:weak_unique}
		There exists $\lambda_0>0$ such that for all $\lambda \geq \lambda_0$
		solutions to \eqref{eq:trans_eq} are weakly unique.
	\end{lemm}
	\begin{proof}
		In order to show weak uniqueness for \eqref{eq:trans_eq}, we prove that solutions to the associated martingale problem are unique.
		By \cite[Theorem 2.3]{Kurtz11} the martingale problem is equivalent to weak solutions of  \eqref{eq:trans_eq} 
		and thus uniqueness of solutions to the martingale problem implies weak uniqueness.
		This means, that the aim is to show that the martingale problem associated with the operator $A$ defined as 
		\begin{equation}
			\begin{split}
				&Af(x) = \lambda f'(x)b(x) + \int_\R \big(f(x+ \sigma(x,r)+r)-f(x) - f'(x)(\sigma(x,r)+\mathbbm{1}_{|r|<1}r) \big)\tilde{\pi}(dr)\\
				& = f'(x)\Big(\lambda b(x) - \int_{|r|\geq 1} \sigma(x,r) \tilde{\pi}(dr) \Big)\\
				& \qquad+ \int_\R \big(f(x+ \sigma(x,r)+r)-f(x) - f'(x)\mathbbm{1}_{|r|<1}(\sigma(x,r) +r)\big) \tilde{\pi}(dr).
			\end{split}
		\end{equation}
		is well posed.
		We will rewrite $A$ in such a way that we can show that it satisfies the conditions of \cite[Theorem 2.4]{Knopova17}.
		First note, that the map $r \mapsto \sigma(x,r)+r$ has an inverse for every $x \in \R$ and any $\lambda \geq \lambda_0$ if $\lambda_0$
		is chosen large enough so that \eqref{eq:bound_w_lam_prime} holds and such that \eqref{eq:resolvent} has a unique solution. We will denote this inverse by $g(x,r)$, for all $x \in \R , r \in \R$ and it holds that
		\begin{equation*}
			g(x,r) = \phi^{-1}(x+r)-\phi^{-1}(x), \quad x,r\in\R.
		\end{equation*}
		Thus the following holds
		\begin{align*}
			&\int_\R f(x+ \sigma(x,r)+r)-f(x) - f'(x)\mathbbm{1}_{|r|<1}(\sigma(x,r) +r) \tilde{\pi}(dr)\\
			&= \int_\R \big( f(x+r) -f(x) - f'(x)\mathbbm{1}_{|g(x,r)|<1}r\big) \big| \frac{r}{g(x,r)} \big|^{1+\alpha} \big((w_\lambda^{\mu,\mu})'(\phi^{-1}(x)+r) +1\big) \tilde{\pi}(dr)\\
			& =\int_\R \big( f(x+r) -f(x) - f'(x)\mathbbm{1}_{|r|<1}r\big) \big| \frac{r}{g(x,r)} \big|^{1+\alpha} \big((w_\lambda^{\mu,\mu})'(\phi^{-1}(x)+r) +1\big) \tilde{\pi}(dr)\\
			& + f'(x) \int_\R \big(\mathbbm{1}_{|g(x,r)|<1} - \mathbbm{1}_{|r|<1}\big) \big| \frac{r}{g(x,r)} \big|^{1+\alpha} \big((w_\lambda^{\mu,\mu})'(\phi^{-1}(x)+r) +1\big) \tilde{\pi}(dr),
		\end{align*}
		for all $x\in \R$. Now we set 
		\begin{equation}
			\begin{split}
				&a(x) = \lambda b(x) - \int_{|r|\geq 1} \sigma(x,r) \tilde{\pi}(dr) \\
				&\qquad +\int_\R \big(\mathbbm{1}_{|g(x,r)|<1} - \mathbbm{1}_{|r|<1}\big) \big| \frac{r}{g(x,r)} \big|^{1+\alpha} \big((w_\lambda^{\mu,\mu})'(\phi^{-1}(x)+r) +1\big) \tilde{\pi}(dr), \quad x\in \R,
			\end{split}
		\end{equation}
		and
		\begin{equation}
			m(x,r) = \big| \frac{r}{g(x,r)} \big|^{1+\alpha} \big((w_\lambda^{\mu,\mu})'(\phi^{-1}(x)+r) +1\big), \quad x,r \in \R,
		\end{equation}
		and rewrite
		\begin{equation*}
			Af(x) = a(x)f'(x) +\int_\R \big(f(x+r)-f(x)-\mathbbm{1}_{|r|<1}f'(x)r\big)m(x,r) \tilde{\pi}(dr), \quad x\in \R.
		\end{equation*}
		We need to check the conditions of \cite[Theorem 2.4]{Knopova17}. That is, we need to show
		\begin{align}
			\label{eq:a_bound}&|a(x)| \leq c_1, \quad \forall x \in \R\\
			\label{eq:m_bound}c_2 \leq & m(x,r) \leq c_3, \quad \forall x,y,r\in \R,
		\end{align}
		and 
		\begin{equation}\label{eq:m_a_holder}
			|m(x,r)-m(y,r)| + |a(x)-a(y)| \leq c_4(|x-y|^\delta\wedge 1), \quad \forall x,r \in \R,
		\end{equation}
		for some $\delta \in (0,1)$, $c_1,c_2,c_3>0$.\\
		Note that $\|(w_\lambda^{\mu,\mu})'\|_\infty \leq \epsilon_0(\lambda)< 1$, for $\lambda \geq \lambda_0$
		and thus for some $C>1$ we have
		\begin{equation}\label{eq:m_est1}
			C^{-1} \leq	\Big| \frac{r}{g(x,r)} \Big| \leq \frac{|r|}{|r| (1-\epsilon_0)} \leq C,
		\end{equation}
		by using \eqref{eq:lip_phi_inv}.
		Furthermore, again by \eqref{eq:bound_w_lam_prime} we get  that there exists $C>1$ such that 
		\begin{equation}\label{eq:m_est2}
			C^{-1} \leq \|(w_\lambda^{\mu,\mu})'(\phi^{-1}(\cdot)+ r)+1\|_\infty \leq C.
		\end{equation}
		Thus it is easily seen that \eqref{eq:a_bound} and \eqref{eq:m_bound} hold.
		Now we turn to \eqref{eq:m_a_holder}. To prove the inequality it is enough to show that 
		\begin{equation*}
			|m(x,r)-m(y,r)| + |a(x)-a(y)| \leq c_4|x-y|^\delta,
		\end{equation*}
		for all $x,y\in \R$ such that $|x-y| < \delta_0$ for some fixed $\delta_0 > 0$ that is chosen sufficiently small. First we find an estimate for the term involving $m$.
		So, let $x,y \in \R$ be such that $|x-y|\leq \delta_0$, then 
		\begin{align*}
			&|m(x,r)-m(y,r)|\\
			&  =\Big| \big| \frac{r}{g(x,r)} \big|^{1+\alpha} \big((w_\lambda^{\mu,\mu})'(\phi^{-1}(x)+r) +1\big) - \big| \frac{r}{g(y,r)} \big|^{1+\alpha} \big((w_\lambda^{\mu,\mu})'(\phi^{-1}(y)+r) +1\big) \Big|\\
			& \leq \Big| \big| \frac{r}{g(x,r)} \big|^{1+\alpha} \big((w_\lambda^{\mu,\mu})'(\phi^{-1}(x)+r) +1\big) - \big| \frac{r}{g(y,r)} \big|^{1+\alpha} \big((w_\lambda^{\mu,\mu})'(\phi^{-1}(x)+r) +1\big) \Big| \\
			&+ \Big| \big| \frac{r}{g(y,r)} \big|^{1+\alpha} \big((w_\lambda^{\mu,\mu})'(\phi^{-1}(x)+r) +1\big) - \big| \frac{r}{g(y,r)} \big|^{1+\alpha} \big((w_\lambda^{\mu,\mu})'(\phi^{-1}(y)+r) +1\big) \Big| \\
			&\leq c  \Big| \big| \frac{r}{g(x,r)} \big|^{1+\alpha} - \big| \frac{r}{g(y,r)} \big|^{1+\alpha} \Big| \\
			&+ c\Big|(w_\lambda^{\mu,\mu})'(\phi^{-1}(x)+r)  -  (w_\lambda^{\mu,\mu})'(\phi^{-1}(y)+r) \Big| \\
			& =: J_1(x,y) + J_2(x,y),
		\end{align*}
		where we used \eqref{eq:m_est1} and \eqref{eq:m_est2} in the second inequality.
		It is easy to see that 
		\begin{align*}
			J_2(x,y)\leq c \|(w_\lambda^{\mu,\mu})'\|_{\mathcal{C}^{\epsilon/2}} |x-y|^{\epsilon/2}, 
		\end{align*}
		since $\phi^{-1}$ is Lipschitz continuous and $(w_\lambda^{\mu,\mu})' \in \mathcal{C}^{\epsilon/2}$ by Lemma~\ref{lem:boundw_lmm} and \cite[Lemma 3.1]{Athreya20}.
		Now, we turn to $J_1$. Note that the sign of $g$ only depends on $r$. That is, if  $\sgn(g(x,r)) = \sgn(r)$ for all $x \in \R$, since $\phi^{-1}$ is monotone.
		Thus, without loss of generality we assume $r >0$ and get
		\begin{align*}
			&J_1 (x,y)\leq cr^{1+\alpha} \big| g(x,r)^{-1-\alpha} - g(y,r)^{-1-\alpha} \big| \\
			& \leq cr^{1+\alpha}\big| g(x,r)^{-1-\alpha} -g(x,r)^{-1}g(y,r)^{-\alpha} + g(x,r)^{-1}g(y,r)^{-\alpha}- g(y,r)^{-1-\alpha} \big| \\
			& \leq cr^{1+\alpha} \Big( r^{-1}\big| g(x,r)^{-1} -g(y,r)^{-1} \big|^{\alpha} + r^{-\alpha}\big| g(x,r)^{-1} -g(y,r)^{-1} \big|\Big), \quad \forall x,y\in \R, |x-y| \leq \delta_0.
		\end{align*}
		It is not to hard to check, that in fact $\phi^{-1} \in \mathcal{C}^{1+\epsilon/2}$.
		Together with \cite[Lemma 5.6]{Athreya20} this leads us to the following estimate
		\begin{align*}
			| g(x,r)^{-1} -g(y,r)^{-1} \big| \leq c r^{-2} |g(x,r)-g(y,r)| \leq cr^{-1} \| \phi^{-1}\|_{\mathcal{C}^{1+\epsilon/2}}|x-y|^{\epsilon/2}, \quad \forall x,y\in \R, |x-y| \leq \delta_0.
		\end{align*}
		Plugging this in yields 
		\begin{equation*}
			J_1(x,y) \leq c |x-y|^\delta,
		\end{equation*}
		for some $\delta > 0$ when $|x-y|\leq\delta_0$. This means it holds 
		\begin{equation}\label{eq:m_holder}
			|m(x,r)-m(y,r)| \leq c|x-y|^{\delta},
		\end{equation}
		for all $r\in \R$ and $x,y$ such that $|x-y|\leq\delta_0$, where $c$ is independent of $r$.
		Now, it remains to consider $a$.  It holds
		\begin{align*}
			&|a(x)-a(y)|  \leq \lambda|b(x)-b(y)|  + \sup_{|r|>1}|\sigma(x,r)-\sigma(y,r)|\int_{|r|\geq 1} \tilde{\pi}(dr) \\
			& + \Big|\int_\R \big(\mathbbm{1}_{|g(x,r)|<1} - \mathbbm{1}_{|r|<1}\big) \big| \frac{r}{g(x,r)} \big|^{1+\alpha} \big((w_\lambda^{\mu,\mu})'(\phi^{-1}(x)+r) +1\big) \tilde{\pi}(dr) \\
			& \qquad - \int_\R \big(\mathbbm{1}_{|g(y,r)|<1} - \mathbbm{1}_{|r|<1}\big) \big| \frac{r}{g(y,r)} \big|^{1+\alpha} \big((w_\lambda^{\mu,\mu})'(\phi^{-1}(y)+r) +1\big) \tilde{\pi}(dr)\Big|, \quad x,y\in \R.
		\end{align*}
		It is straightforward to see that  the  Hölder condition holds for some  $\delta > 0$ for the first two terms. Due to \eqref{eq:m_holder}
		it is enough to prove 
		\begin{equation}\label{eq:D_leb}
			\mathrm{Leb}\big(D(x,y)\big) \leq c  |x-y|^{\delta}, \quad  \forall x,y \in \R,|x-y|\leq \delta_0,
		\end{equation}
		and
		\begin{equation}\label{eq:D_lower_bound}
			\inf_{r \in D(x,y)} |r| \geq c >0, \quad  \forall x,y \in \R,|x-y|\leq \delta_0,
		\end{equation}
		where 
		\begin{equation}
			D(x,y) := \Big\{ r \in \R \colon \big|\mathbbm{1}_{|g(x,r)|<1} - \mathbbm{1}_{|g(y,r)|<1}\big| >0\Big\},  \quad  \forall x,y \in \R,|x-y|\leq \delta_0.
		\end{equation}
		First, we verify \eqref{eq:D_lower_bound}. To this end note that if $r \in D(x,y)$, then  either $1 \leq |g(x,r)|$ or
		$1 \leq |g(y,r)|$. Without loss of generality we assume  $1 \leq g(x,r) $. It thus holds 
		\begin{equation*}
			1 \leq g(x,r)  = \phi^{-1}(x+r) - \phi^{-1} (x) \Leftrightarrow \phi(1+\phi^{-1}(x)) - \phi(\phi^{-1}(x))\leq r,
		\end{equation*} 
		since $\phi$ is strictly monotone.
		Due to the fact that $\phi$ is monotone increasing and  \eqref{eq:lip_phi} it thus follows $0<c \leq r$ as desired. 
		Note that the sign of $g(\cdot,r)$ only depends on $r$. For this reason it is enough to verify 
		\begin{equation}\label{eq:Dtilde_leb}
			\mathrm{Leb}(\tilde{D}(x,y)) \leq c |x-y|^\delta,  \quad  \forall x,y \in \R,|x-y|\leq \delta_0
		\end{equation}
		where $\tilde{D}(x,y) = \{ r \in \R \colon 0 \leq g(x,r) <1 \leq g(y,r)  \}$, to verify \eqref{eq:D_leb}.
		Not that $r \in \tilde{D}(x,y)$,  only if $0 \leq r$ and 
		\begin{equation*}
			r < \phi(1+ \phi^{-1}(x)) -x , \quad r \geq \phi(1+ \phi^{-1}(y)) -y. 
		\end{equation*}
		By noting that $\phi$ and $\phi^{-1}$ are Lipshitz continuous \eqref{eq:Dtilde_leb} is easily verified. This finishes the proof.
	\end{proof}
	Now we show that \eqref{eq:trans_eq} has pathwise unique solutions. 
	\begin{lemm}\label{lem:path_unique_trans}
		There exists $\lambda_0 >0$ so that for all $\lambda \geq \lambda_0$ solutions to \eqref{eq:trans_eq} are pathwise unique.
	\end{lemm}
	\begin{proof}
		Assume that $(W^1_t)_{t \geq 0}$ and $(W^2_t)_{t \geq 0}$ 
		are two solutions driven by the same L\'evy process $(L_t)_{t \geq 0}$. Note that $\sigma(w,r)+ w$ is non-decreasing in $w$ since
		$\sigma$ is Lipschitz with Lipschitz constant smaller than one (use the estimate of $w_{\lambda}^{\mu,\mu}
		$ like in \cite[Lemma 5.6]{Athreya20}). An application of Tanaka's formula yields
		\begin{equation}\label{eq:max_Ito}
			\begin{split}
				&W_t^1\vee W_t^2 = W_t^1 + \big(W_t^2-W_t^1\big)^+ \\
				& = W_t^1 + \int_0^t \mathbbm{1}_{W_{s-}^2>W_{s-}^1}d\big(W_s^2-W_s^1\big) \\
				&\quad  + \sum_{0 \leq s\leq t} \mathbbm{1}_{W_{s-}^2>W_{s-}^1} \big(W_s^2-W_s^1\big)^- + \sum_{0 \leq s\leq t} \mathbbm{1}_{W_{s-}^2\leq W_{s-}^1} \big(W_s^2-W_s^1\big)^+,
			\end{split}
		\end{equation}
		for $t \geq 0$.
		The same argument as in the proof of \cite[Proposition 4.1]{Xiong19} demonstrates that the last two terms on the right-hand side of \eqref{eq:max_Ito} are equal to zero. 
		Thus, we have
		\begin{align*}
			&W_t^1\vee W_t^2  =  W_t^1 + \int_0^t \mathbbm{1}_{W_{s-}^2>W_{s-}^1}d\big(W_s^2-W_s^1\big) \\
			& \quad = \phi(x) + L_t + \int_0^t \int_\R \sigma(W_{s-}^1\vee W_{s-}^2,r) \big(\Pi-\pi\big)(ds,dr) + \lambda \int_0^t b\big(W_s^1\vee W_s^2\big)ds, \quad t \geq 0.
		\end{align*}
		This means, that $(W_t^1\vee W_t^2)_{t \geq 0}$ is also a solution. Note that we already established the weak uniqueness of solutions to \eqref{eq:trans_eq} in Lemma~\ref{lemm:weak_unique}. 
		This means, that for any $t \geq 0$, $W_t^1\vee W_t^2 $, $W^1_t$, and $W^2_t$ all have the same law. Since $W^1_t$ and $W^2_t$ are random variables with the same 
		law on the same probability space and their maximum also has the same law, it follows that they are equal almost surely. 
		Since this holds for all $t \geq 0$ and $W^1$ and $W^2$ are c\`adl\`ag processes we immediately get that $W^1_t = W^2_t$ for $t \geq 0$, $\bbP$-almost surely.
	\end{proof}
	\begin{proof}[Proof of Theorem~\ref{th:sde_exist}\ref{itm:stronex}]
		Note, that $\phi$, defined in \eqref{eq:zvonkin_transform}, is one to one.
		 For this reason, the pathwise uniqueness of solutions to \eqref{eq:trans_eq}, given by Lemma~\ref{lem:path_unique_trans}, implies that solutions to \eqref{eq:trams_eq_X} are pathwise unique as well. 
		By Lemma~\ref{lemm:trans_eq} every solution to \eqref{eq:singularSDE} is a solution to \eqref{eq:trams_eq_X}. This yields  pathwise uniqueness of solutions to \eqref{eq:singularSDE}.
		 Using similar arguments as in the proof of \cite[Theorem 2.3]{Athreya20}, we can use the generalized version of the classical Yamada-Watanabe theorem from \cite{Kurtz14} to infer that there exists a strong solution. Since this is a standard argument, we omit it here. 
		 This concludes the proof of Theorem~\ref{th:sde_exist}~\ref{itm:stronex}.
	\end{proof}
\subsection{Other Approximating Sequences - Proof of Lemma~\ref{cor:PUC}}\label{sec:aprox_seq}
In this section, we briefly want to discuss conditions on a sequence $(\tilde{\mu}_n)_{n\in N}$ approximating $\mu$, that ensure the convergence \eqref{eq:conv_A_PUc}. 
By Lemma~\ref{lem:conv_A}, it is enough to show that the $\lambda$-potentials $W_\lambda^{\tilde{\mu}_n}$ converge uniformly to $W_\lambda^\mu\mu$.  In the following lemma we give sufficient conditions on the sequence $(\tilde{\mu}_n)_{n\in N}$ to ensure this convergence.
\begin{lemm}\label{lem:conv_nu_seq_pot}
	Let $\lambda >0$ and $\nu\in K^f_{\alpha-1}$. Furthermore, let $(V_n)_{n \in \N} \subset \mathcal{C}_b^\infty$ and denote $v_n(dx) = V_n(x)dx$. Assume that $v_n$ converges weakly to $\nu$ and that 
	\begin{equation}\label{eq:limsup_Mnu}
		\lim_{r \rightarrow 0}\sup_{n \in \N}M^{\alpha-1}_{v_n}(r) =0.
	\end{equation}
	Then  $W_\lambda^\mu v_n$ converges to $W_\lambda^\mu \nu$ uniformly on $\R$.
\end{lemm}
\begin{proof}
	To prove the assertion we follow an approach similar to the one in the proof of \cite[Lemma 4.7]{Kim14}. 
	By using \eqref{eq:limsup_Mnu},  we can follow the arguments of \cite[Lemma 4.7]{Kim14}
	with only slight modifications to see that it is enough to prove 
	\begin{equation*}
		\lim_{n \rightarrow \infty } \sup_{(t,x) \in [\delta,T] \times B(0,R) }\Big|\int_{B(0,R)}\ee^{-\lambda t}p^\mu(t,x,y)\big(\nu(dy)-v_n(dy)\big)\Big| =0,
	\end{equation*}
	for all $0<\delta< T< \infty$, $0<R<\infty$. However, since $p^\mu$ is jointly continuous by Theorem~\ref{th:transition_kernel}, this follows by the weak convergence of $v_n$. This finishes the proof.
\end{proof}
By Lemma~\ref{lem:conv_A}, the above lemma yields Lemma~\ref{cor:PUC}.

\section{Proof of Theorem~\ref{th:local_time}}\label{sec:proof_loc_time}
\begin{proof}[Proof of Theorem~\ref{th:local_time}]
	%It can easily be verified that for all $\mu$ of Kato class $K_{\alpha-1}$ and of finite energy integral $\gamma_t^x \in L^1(\mu)$ for all $t$ and $\bbP$-a.s. 
	We denote the set of nonnegative finite measures of  Kato class $K_{\alpha-1}$ by $\mathcal{S}_{K_0}$.  Furthermore, denote by $\upsilon_A$ the Lebesgue-Stieltjes measure of $A$.
	We see that for all $\omega \in \Omega$ such that $\sup_{0 \leq t\leq T}|X_t|< \infty$ it follows
	\begin{align*}
		&\int_\R  \sup_{0 \leq t \leq T}\Big|\int_0^t v'(X_s-x)dA_s\Big|\mu(dx)\\
		& \leq c(\alpha) \int_0^T \int_{B(X_s,R)}|X_s-x|^{\alpha-2}\mu(dx) + \int_{B(X_s,R)^c}|X_s-x|^{\alpha-2}\mu(dx) |\upsilon_A|(ds)\\
		&\leq c(\alpha)\big(M_{\mu}^{\alpha-1}(R)+ R^{\alpha-2}\mu(\R)\big)|\upsilon_A|([0,T]) <\infty
	\end{align*}
	for $R > 0$ and all  $\mu \in S_{K_0}$. Hereby $M_{\mu}^{\alpha-1}(R)$ is defined as in 
	Definition~\ref{def:Kato}.
	Finally, we have by substituting $(X_{t-}-x)z = r$
	\begin{align*}
		&\int_\R\sup_{0 \leq t\leq T} \bbE_{x_0}\Big[\sup_{0 \leq t \leq T} \int_0^t\int_\R (v(X_{t-}-x+r) -v(X_{t-}-x))^2 \tilde{\pi}(dr)ds \Big] \mu(dx) \\
		&\leq c(\alpha) \int_0^T\bbE_{x_0}\Big[\int_\R |X_t-x|^{\alpha-2}\mu(dx) \Big] dt< \infty \quad \forall \mu \in S_{K_0}. 
	\end{align*}
	This implies $\mu(D_T^c) = 0$ for all $\mu \in S_{K_0}$. Showing that $\mathrm{Cap}_{\alpha-1}(D_T^c) = 0$ finishes the proof.
	However, this claim holds true due to Theorem~\ref{th:capacity_condition}.
\end{proof}

\section{Proof of Proposition~\ref{lem:density}}\label{sec:identification_of_local_time}
In \cite{Boylan64} the author derived criteria for the potential density that yield the existence of a jointly continuous occupation density of the corresponding process. More precisely, to prove the continuity of the occupation density of solutions to  \eqref{eq:singularSDE} we need to show that there
exist  $\lambda > 0$ and $\epsilon_0 > 0$ such that for all $x,y \in \R$ with $|x-y|<\epsilon_0$ it holds 
\begin{equation}\label{eq:continuity_dens}
	|w_\lambda^\mu (x,y)-w_\lambda^\mu(x,x)| < g(|x-y|) ,\quad |w_\lambda^\mu (y,x)-w_\lambda^\mu(x,x)| < g(|x-y|), 
\end{equation}
where $g$ is a nonnegative, increasing function such that 
\begin{equation}\label{eq:good_function}
	\sum_{n \geq 1} n g(2^{-n})^{1/2} < \infty.
\end{equation}
\begin{rema}
	Clearly, every increasing, Hölder continuous function $g$ with $g(0) = 0$ satisfies \eqref{eq:good_function}.
\end{rema}
The following theorem is a slight modification of \cite[Theorem 1]{Boylan64}.
\begin{theorem}\label{th:density_exist_boyloan}
	Let $X$ be a strong Markov process with almost surely right continuous sample paths
	and with potential density $w_\lambda^\mu$. If there exists $\lambda >0$
	such that \eqref{eq:continuity_dens} is satisfied, then there exists an occupation density $\ell_t^x$ that is 
	almost surely jointly continuous in $x$ and $t$.
\end{theorem}
\begin{proof}
	In \cite[Theorem 1]{Boylan64} the statement is proven for $\lambda = 1$. 
	The claim follows for $1 \neq \lambda >0$ by following the proof in \cite{Boylan64} line by line and making only trivial adaptations.
\end{proof}
In the following lemma, we prove that \eqref{eq:continuity_dens} is satisfied when $\mu \in K_{\alpha-1}$
satisfies some slightly stricter assumptions.  To prove the lemma we use the representation from Lemma~\ref{lem:pot_mu_dens} and employ the estimates of $u_\lambda$.

\begin{lemm}\label{lem:regularity_potential}
	Let $\mu\in K_{\alpha-1}$ be a finite measure such that 
	\begin{equation}\label{eq:strict_kato}
		\sup_{x \in \R}\int_{B(x,r)}|x-y|^{\alpha-2}\frac{1}{g(|x-y|)} |\mu|(dy) \rightarrow 0, \text{ as } r \rightarrow 0,
	\end{equation}
	for some nonnegative, increasing function $g$ such that $\bar{g}(x) = g(x^{1/4})$ satisfies \eqref{eq:good_function}.
	Then there exist $\lambda_0 >0$ and $\epsilon_0 > 0$ such that for all $\lambda > \lambda_0$ and $|x-y| \leq \epsilon_0$
	\begin{align}
		\label{eq: ineq1}&|w_\lambda^\mu(x,y)-w_\lambda^\mu(x,x)|  \leq C_2\tilde{g}(|x-y|),\\
		\label{eq:ineq2}&|w_\lambda^\mu(y,x)-w_\lambda^\mu(x,x)| \leq C_2\tilde{g}(|x-y|),
	\end{align}
	for some constant $C_2 = C_2(\lambda,\mu,\alpha, \epsilon_0)> 0$ and some nonnegative, increasing function $\tilde{g}$ satisfying \eqref{eq:good_function}.
\end{lemm}
\begin{proof}
	First, we prove \eqref{eq: ineq1}.
	Due to Lemma~\ref{lem:pot_mu_dens} we have an explicit representation of the potential densities at hand
	\begin{equation}\label{eq:density_w}
		\begin{split}
			&w_\lambda^\mu(x,y)-w_\lambda^\mu(x,x) \\
			& \qquad = \sum_{k = 0}^\infty \Big( u_\lambda \ast \big(\frac{\partial}{\partial x}u_\lambda \mu\big)^{\ast k} \Big)(x,y)
			-\sum_{k = 0}^\infty \Big( u_\lambda \ast \big(\frac{\partial}{\partial x}u_\lambda \mu\big)^{\ast k} \Big)(x,x ).
		\end{split}
	\end{equation}
	Rewrite \eqref{eq:density_w} as
	\begin{equation}\label{eq:difference1}
		\begin{split}
			&\sum_{k = 0}^\infty\Big( u_\lambda \ast \big(\frac{\partial}{\partial x}u_\lambda \mu\big)^{\ast k} \Big)(x,y)
			-\sum_{k = 0}^\infty \Big( u_\lambda \ast \big(\frac{\partial}{\partial x}u_\lambda \mu\big)^{\ast k} \Big)(x,x ) \\
			& = u_\lambda(x,y)-u_\lambda(x,x) \\
			& \quad + \sum_{k= 1}^\infty \int_\R \Big( u_\lambda \ast \big(\frac{\partial}{\partial x}
			u_\lambda \mu\big)^{\ast k-1} \Big)(x,z)\big(\frac{\partial}{\partial z}u_\lambda(z,y)-\frac{\partial}{\partial z}u_\lambda(z,x)\big) \mu(dz).
		\end{split}
	\end{equation}
	By the proof in \cite[Section 4]{Boylan64} 
	it follows that 
	\begin{equation}\label{eq:u_lam_holder}
		|u_\lambda(x,y)-u_\lambda(x,x)|\leq c(\lambda, \alpha)|x-y|^{\alpha-1},
	\end{equation}
	for $|x-y|$ small enough and some constant $c= c(\lambda,\alpha)>0$.
	For this reason it only remains to estimate the second term on the right hand side of \eqref{eq:difference1}.
	By Lemma~\ref{lem:grad_est_pot_dens} in Appendix ~\ref{append3} there exists $\lambda_0 >0$ such that for all $\lambda \geq \lambda_0$ it holds
	\begin{align*}
		&\int_\R \Big| u_\lambda \ast \big(\frac{\partial}{\partial x}u_\lambda \mu\big)^{\ast k-1} \Big|(x,z)
		\Big|\frac{\partial}{\partial z}u_\lambda(z,y)-\frac{\partial}{\partial z}u_\lambda(z,x)\Big| |\mu|(dz) \\
		& \leq c(\lambda,\mu)^{k-1}\int_\R |u_\lambda(x,z)|\Big|\frac{\partial}{\partial z}u_\lambda(z,y)-\frac{\partial}{\partial z}u_\lambda(z,x)\Big| |\mu|(dz),
	\end{align*}
	where $c(\lambda,\mu)<1$.
	For $x,y$ fixed we split $\R$ in two parts, $D := \{z \colon x-z, y-z \leq 0\} \cup \{z \colon x-z, y-z \geq 0\}$ and $D^c = \{z \colon x \leq z \leq y\} \cup \{z \colon y \leq z \leq x\}$.
	Consider the following decomposition
	\begin{align*}
		&\int_\R |u_\lambda(x,z)|\Big|\frac{\partial}{\partial z}u_\lambda(z,y)-\frac{\partial}{\partial z}u_\lambda(z,x)\Big| |\mu|(dz)\\
		&= \int_\R \mathbbm{1}_D|u_\lambda(x,z)|\Big|\frac{\partial}{\partial z}u_\lambda(z,y)-\frac{\partial}{\partial z}u_\lambda(z,x)\Big| |\mu|(dz)\\
		& \quad + \int_\R \mathbbm{1}_{D^c}|u_\lambda(x,z)|\Big|\frac{\partial}{\partial z}u_\lambda(z,y)-\frac{\partial}{\partial z}u_\lambda(z,x)\Big| |\mu|(dz)\\
		& =: J_1 +J_2.
	\end{align*}
	We first estimate $J_1$. To this end, note that 
	for all $z \in D$ 
	the difference of the derivatives of $u_\lambda$ given by \eqref{eq:deriv_u_lam} amounts to
	\begin{equation}\label{eq:est_deriv1}
		\begin{split}
			&\Big|\frac{\partial}{\partial z}(u_\lambda(z,y)-u_\lambda(z,x))\Big|\\
			&= c \Big||x-z|^{\alpha-2}\int_0^\infty\frac{\sin(\xi)\xi}{\lambda|x-z|^\alpha+|\xi|^\alpha}d\xi-|y-z|^{\alpha-2}\int_0^\infty\frac{\sin(\xi)\xi}{\lambda|y-z|^\alpha+|\xi|^\alpha}d\xi\Big|\\
			&\leq c \Big(\Big||x-z|^{\alpha-2}-|y-z|^{\alpha-2} \Big| \Big| \int_0^\infty\frac{\sin(\xi)\xi}{\lambda|y-z|^\alpha+|\xi|^\alpha} d\xi\Big|\\
			& \quad + |x-z|^{\alpha-2}  \Big| \int_0^\infty\frac{\sin(\xi)\xi}{\lambda|y-z|^\alpha+|\xi|^\alpha}- \frac{\sin(\xi)\xi}{\lambda|x-z|^\alpha+ |\xi|^\alpha} d\xi\Big| \Big)\\
			& \leq c\Big( |x-y|^{2-\alpha} \frac{1}{|x-z|^{2-\alpha}|y-z|^{2-\alpha}} \\
			& + \quad |x-z|^{\alpha-2} \Big| \int_0^\infty\sin(\xi)\xi\Big(\frac{\lambda|x-z|^\alpha-\lambda|x-y|^\alpha}{(\lambda|y-z|^\alpha + |\xi|^\alpha)(\lambda|x-z|^\alpha + |\xi|^\alpha)}\Big)  d\xi\Big|\Big).
		\end{split}
	\end{equation}
	Hereby, we used Lemma~\ref{lemm:help_bounded} from Appendix~\ref{append3} in the last inequality.
	We begin by estimating the second term on the right hand side of \eqref{eq:est_deriv1}.
	To this end consider
	\begin{align*}
		& \Big| \int_0^\infty\sin(\xi)\xi\Big(\frac{\lambda|x-z|^\alpha-\lambda|y-z|^\alpha}{(\lambda|y-z|^\alpha + |\xi|^\alpha)(\lambda|x-z|^\alpha + |\xi|^\alpha)}\Big)  d\xi\Big|  \\
		&\leq  \int_0^1\xi^\alpha\Big|\frac{\lambda|x-z|^\alpha-\lambda|y-z|^\alpha}{(\lambda|y-z|^\alpha + |\xi|^\alpha)(\lambda|x-z|^\alpha + |\xi|^\alpha)}\Big| d\xi \\
		&\quad+ \int_1^\infty\xi\Big|\frac{\lambda|x-z|^\alpha-\lambda|y-z|^\alpha}{(\lambda|y-z|^\alpha + |\xi|^\alpha)(\lambda|x-z|^\alpha + |\xi|^\alpha)}\Big| d\xi\\
		& \leq c(\alpha) \lambda |x-y| \max(|x-z|, |y-z|)^{\alpha-1}\\
		&\qquad \times\Big(\int_0^1\frac{\xi^\alpha}{(\lambda|y-z|^\alpha + |\xi|^\alpha)(\lambda|x-z|^\alpha + |\xi|^\alpha)} d\xi \\
		& \qquad \qquad + \int_1^\infty\xi\frac{1}{(\lambda|y-z|^\alpha + |\xi|^\alpha)(\lambda|x-z|^\alpha + |\xi|^\alpha)}d\xi\Big)
		=: G_1 +G_2.
	\end{align*} 
	Hereby, we used the mean value theorem in the second inequality. 
	In order to bound $G_1$ note that
	\begin{align*}
		&\int_0^1\frac{\xi^\alpha}{(\lambda|y-z|^\alpha + |\xi|^\alpha)(\lambda|x-z|^\alpha + |\xi|^\alpha)} d\xi \\
		& \leq \int_0^1\frac{\xi^\alpha}{\lambda\min(|x-z|,|y-z|)^\alpha + |\xi|^\alpha}\frac{1}{(\lambda\max(|x-z,|y-z|)^\alpha + |\xi|^\alpha)}d\xi \\
		& \leq \int_0^1\frac{1}{(\lambda\max(|x-z|,|y-z|)^\alpha + |\xi|^\alpha)}d\xi \\
		& \leq c(\alpha)\lambda^{1/\alpha-1}\max(|x-z,|y-z|)^{1-\alpha}
	\end{align*}
	and hence 
	\begin{equation}
		G_1 \leq c(\alpha)\lambda^{1/\alpha}|x-y|.
	\end{equation}
	Now we turn to $G_2$.
	It is evident that 
	\begin{equation*}
		\int_1^\infty\xi\frac{1}{(\lambda|y-z|^\alpha + |\xi|^\alpha)(\lambda|x-z|^\alpha + |\xi|^\alpha)}d\xi \leq \int_1^\infty \frac{\xi}{\xi^{2\alpha}}d\xi \leq c(\alpha),
	\end{equation*}
	for some constant $c(\alpha) >0$, since $2\alpha-1 > 1$. Thus
	\begin{equation}
		G_2 \leq c(\alpha) \lambda|x-y|\max(|x-z|, |y-z|)^{\alpha-1}.
	\end{equation}
	Using the estimates for $G_1$ and $G_2$ yields
	\begin{align*}
		&\int_\R u_\lambda(x,y)|x-z|^{\alpha-2} \Big| \int_0^\infty\sin(\xi)\xi\Big(\frac{\lambda|x-z|^\alpha-\lambda|x-y|^\alpha}{(\lambda|y-z|^\alpha + |\xi|^\alpha)(\lambda|x-z|^\alpha + |\xi|^\alpha)}\Big)  d\xi\Big|\\
		& \leq c(\alpha)\lambda^{1/\alpha}|x-y|\int_\R u_\lambda(x,z)|x-z|^{\alpha-2}|\mu|(dz)\\
		& + \quad c(\alpha)|x-y|\int_\R u_\lambda(x,z)|x-z|^{\alpha-2}\max(|x-z|,|y-z|)^{\alpha-1}|\mu|(dz)\\
		& \leq c(\alpha, \lambda, \epsilon_0) |x-y|,
	\end{align*}
	where we used \eqref{eq:lamPot_est_bod} in the last inequality.
	Now we turn to the first term on the right hand side of \eqref{eq:est_deriv1}. Assume that $|x-y|<\epsilon_0<1$ and denote $B^{x,y} = B(x,|x-y|/2) \cup B(y,|x-y|/2)$.
	By integrating the term $\frac{|x-y|^{2-\alpha}}{|x-z|^{2-\alpha}|y-z|^{2-\alpha}}$ we get 
	\begin{align*}
		&\int_\R| u_\lambda(x,z )| \frac{|x-y|^{2-\alpha}}{|x-z|^{2-\alpha}|y-z|^{2-\alpha}} |\mu|(dz) \\
		&\leq \int_{B(x,|x-y|/2)}| u_\lambda(x,z )| \frac{|x-y|^{2-\alpha}}{|x-z|^{2-\alpha}|y-z|^{2-\alpha}} |\mu|(dz) \\
		& \quad + \int_{B(y,|x-y|/2)}| u_\lambda(x,z )| \frac{|x-y|^{2-\alpha}}{|x-z|^{2-\alpha}|y-z|^{2-\alpha}}|\mu|(dz) \\
		& \quad + \int_{B(x,|x-y|^{1/4}/2)\backslash B^{x,y} }| u_\lambda(x,z )| \frac{|x-y|^{2-\alpha}}{|x-z|^{2-\alpha}|y-z|^{2-\alpha}}|\mu|(dz) \\ 
		& \quad + \int_{B(y,|x-y|^{1/4}/2)\backslash B^{x,y} }| u_\lambda(x,z )| \frac{|x-y|^{2-\alpha}}{|x-z|^{2-\alpha}|y-z|^{2-\alpha}}|\mu|(dz) \\
		& \quad +\int_{(B(x,|x-y|^{1/4}/2)\cup B(y,|x-y|^{1/4}/2)^c}| u_\lambda(x,z )| \frac{|x-y|^{2-\alpha}}{|x-z|^{2-\alpha}|y-z|^{2-\alpha}}|\mu|(dy) \\
		& = I_1 + I_2 +I_3 +I_4+ I_5.
	\end{align*}
	The first two terms can be estimated in the same way. We exemplary show how to estimate $I_1$. 
	Note that for $z \in B(y, |x-y|/2)^c$ it holds that 
	\begin{equation}\label{eq:y_z_est}
		|y-z|^{\alpha-2} \leq (|x-y|/2)^{\alpha-2}
	\end{equation}
	and thus 
	\begin{align*}
		I_1 \leq c(\lambda) g(|x-y|) \int_{B(x,|x-y|/2)} \frac{1}{|x-z|^{2-\alpha}g(|x-z|)} |\mu|(dz) \leq c(\lambda,\mu) g(|x-y|),
	\end{align*}
	where we used that $g$ is increasing and \eqref{eq:lamPot_bound}.
	Also, $I_3$ and $I_4$ can be estimated similarly and we will only demonstrate how to bound $I_3$.
	Note that \eqref{eq:y_z_est} also holds for $z \in (B^{x,y})^c$.
	Using \eqref{eq:y_z_est}, \eqref{eq:strict_kato} and that $g$ is increasing and \eqref{eq:lamPot_bound} we arrive at
	\begin{align*}
		&I_3 \leq k \int_{B(x, |x-y|^{1/4}/2)\backslash B^{x,y}} \frac{1}{|x-z|^{2-\alpha}}|\mu|(dz) \\
		&\leq  \int_{B(x, |x-y|^{1/4}/2)\backslash B^{x,y}} \frac{g(|x-z|)}{|x-z|^{2-\alpha}g(|x-z|)}|\mu|(dz) \leq c(\alpha,\mu)g(|x-y|^{1/4}).
	\end{align*}

	Finally, using \eqref{eq:lamPot_bound} it is easy to verify that
	\begin{align*}
		I_5 \leq c(\alpha, \mu)|x-y|^{(2-\alpha)/2}.
	\end{align*}
	Now we estimate $J_2$. To this end note that 
	$D^c \subseteq B(x, |x-y|) \cup B(y,|x-y|)$. It holds
	\begin{align*}
		&\int_{D^c}\Big| \int_0^\infty \frac{\xi \sin(\xi(x-z))}{\lambda + \xi^\alpha}d\xi \Big||\mu|(dz) \leq \int_{D^c} |x-z|^{\alpha-2} \Big|\int_0^\infty \frac{\xi \sin(\xi)}{\lambda|x-z|^\alpha+ |\xi|^\alpha} d\xi\Big| |\mu|(dz) \\
		&\leq c(\alpha) \int_{D^c} |x-z|^{\alpha-2}|\mu|(dz)\\
		& \leq c(\alpha)g(|x-y|) \int_{D^c} \max\big(|x-z|^{\alpha-2}\frac{1}{g(|x-z|)}, |y-z|^{\alpha-2}\frac{1}{g(|y-z|)}\big)|\mu|(dz)\\
		& \leq c(\alpha, \mu)g(|x-y|),
	\end{align*}
	where we applied Lemma~\ref{lemm:help_bounded} in the second inequality.
	Thus, 
	\begin{align*}
		&J_2 = \Big|\int_{D^c} u_\lambda(x,z)\big(\frac{\partial}{\partial z}u_\lambda(z,y)-\frac{\partial}{\partial z}u_\lambda(z,x)\big) \mu(dz)\Big| \\
		& \leq c(\alpha) \Big(\int_{D^c}\big|\frac{\partial}{\partial z}u_\lambda(z,y)\big||\mu|(dz) + \int_{D^c}\big|\frac{\partial}{\partial z}u_\lambda(z,x)\big||\mu|(dz) \Big)\\
		& \leq c(\alpha,\mu)g(|x-y|).
	\end{align*}
	This concludes the proof of \eqref{eq: ineq1} with $\tilde{g}_1(x) =  \max(g(x^{1/4}), |x|^\frac{2-\alpha}{2}))$ for $x \in [0,1)$ and some constant $C_2= C_2(\lambda,\mu,\alpha, \epsilon_0)$.\\
	Now we prove \eqref{eq:ineq2}. Again, we employ the explicit representation of the potential density to obtain
	\begin{equation}\label{eq:ineq_2_start}
		\begin{split}
			&\sum_{k = 0}^\infty\Big( u_\lambda \ast \big(\frac{\partial}{\partial x}u_\lambda \mu\big)^{\ast k} \Big)(y,x)
			-\sum_{k = 0}^\infty \Big( u_\lambda \ast \big(\frac{\partial}{\partial x}u_\lambda \mu\big)^{\ast k} \Big)(x,x ) \\
			& = u_\lambda(y,x)-u_\lambda(x,x) + \sum_{k= 1}^\infty \int_\R  (u_\lambda(y,z) -u_\lambda(x,z))
			\Big(\big(\frac{\partial}{\partial x}u_\lambda \mu\big)^{\ast k} \Big)(dz,x).
		\end{split}
	\end{equation}
	By \eqref{eq:u_lam_holder} it is enough to estimate the second term on the right-hand side of
	\eqref{eq:ineq_2_start}.
	It holds  for $|x-z| \neq |y-z|$
	\begin{equation}\label{eq:u_lam_holder2}
		\begin{split}
			&|u_\lambda(y,z)-u_\lambda(x,z)| = \Big|\int_0^\infty \frac{\cos(\xi|x-z|)-\cos(\xi|y-z|)}{\lambda + |\xi|^\alpha} d\xi\Big| \\
			& = 2\Big| \int_0^\infty \frac{\sin\big(\xi/2\big||x-z|-|y-z|\big|\big)\sin\big(\xi/2\big||x-z|+|y-z|\big|\big)}{\lambda + |\xi|^\alpha}d\xi\Big|\\
			& \leq |x-y|^{\alpha-1} \Big|\int_0^\infty \frac{\sin(\xi)\sin\big(\xi \big(\big||x-z|+|y-z|\big|/\big||x-z|-|y-z|\big|\big)\big)}{\lambda\big||x-z|-|y-z|\big|^\alpha + |\xi|^\alpha}d\xi\Big|\\
			& \leq |x-y|^{\alpha-1} \Big(\int_0^1 \frac{\xi}{\lambda\big||x-z|-|y-z|\big|^\alpha 
				+ |\xi|^\alpha}d\xi\\
			&\quad +\int_1^\infty \frac{1}{\lambda\big||x-z|-|y-z|\big|^\alpha + |\xi|^\alpha}d\xi \Big) \\
			& \leq |x-y|^{\alpha-1} \Big(\int_0^1 \frac{1}{ |\xi|^{\alpha-1}}d\xi
			+\int_1^\infty \frac{1}{ |\xi|^\alpha}d\xi \Big) \\
			& \leq c(\alpha)|x-y|^{\alpha-1}.
		\end{split}
	\end{equation}
	If $|x-z| = |y-z|$ then $|u_\lambda(y,z)-u_\lambda(x,z)| = 0$.
	Using \eqref{eq:u_lam_holder2} and Lemma~\ref{lem:pot_grad_upper_est} from Appendix~\ref{append3}
	immediately implies \eqref{eq:ineq2} with $\tilde{g}_2(x) = |x|^{\alpha-1}$ and some constant $C_2= C_2(\lambda,\mu,\alpha, \epsilon_0)> 0$.\\
	Setting $\tilde{g} = \max(\tilde{g}_1, \tilde{g}_2)$ finishes the proof.
\end{proof}
Due to Lemma~\ref{lem:regularity_potential} the conditions of Theorem~\ref{th:density_exist_boyloan} are satisfied for solutions of \eqref{eq:singularSDE} given that 
\eqref{eq:strict_kato}  holds. This implies Proposition~\ref{lem:density}, because the condition that $\mu \in K_{\alpha-1-\epsilon}$ is stronger than \eqref{eq:strict_kato}. In fact, we proved the existence of an continuous occupation density under the weaker assumption \eqref{eq:strict_kato}.

\section{Proof of Proposition~\ref{lem: convergence}}\label{sec:loc_time}
Throughout this section we fix a finite measure $\mu \in K_{\alpha-1}$ and a process $(X_t)_{t \geq 0}$
as defined in \eqref{eq:semimartingale}.
\subsection{It\^o's formula for mollified  renormalized $0$-potential}
In this section, we first apply It\^o's formula to $u_{\lambda,n}(X_t-x)$, $x\in \R$, where $u_{\lambda,n}$ is the mollification of the $\lambda$-potential density of the fractional Laplacian. Then,  by letting $\lambda$ go to zero, we 
derive an analogous formula for $v_n(X_t-x)$, where $v_n$ is the mollification of the renormalized $0$-potential $v$.\\
Let $(A_t)_{t \geq 0}$ be an adapted, right continuous process of finite variation as in \eqref{eq:semimartingale} and set
\begin{equation*}
	\bar{A}_t = |A|_t , \quad t \geq 0
\end{equation*}
to be the total variation process of  $A$.

Furthermore, recall that $u_\lambda$ denotes the $\lambda$-potential density given by \eqref{eq:lam_pot_dens_closed_form}
and $u_{\lambda,n}$ denotes its mollification according to Definition~\ref{def:mollifier}.
Finally, recall that 
\begin{equation}\label{eq:v_def}
	v(x):= \lim_{\lambda \downarrow 0}(u_\lambda(0)-u_\lambda(x))
\end{equation}
and 
denote by $v_n$ its mollification according to Definition~\ref{def:mollifier}.
Now, let us apply It\^o's formula to $u_{\lambda,n}(X_t-x)$, $x\in \R$. This is justified since $u_{\lambda,n}$ is smooth.
\begin{equation}\label{eq:itoulamn}
	\begin{split}
		&u_{\lambda,n}(X_t-x) \\
		&= u_{\lambda,n}(x_0-x) +N_\lambda^{n,x}(t) +\int_0^t (u_{\lambda,n})'(X_s-x)dA_s 
		+\int_0^t \Delta_\alpha u_{\lambda,n}(X_s)ds\\
		& = u_{\lambda,n}(x_0-x) +N_\lambda^{n,x}(t) +\int_0^t (u_{\lambda,n})'(X_s-x)dA_s \\
		& \quad - \int_0^t \phi_n(X_s-x) ds
		+\lambda \int_0^t u_{\lambda,n}(X_s-x)ds,
	\end{split}
\end{equation}
where we used that $u_{\lambda, n}$ solves
\begin{equation*}
	-\Delta_\alpha u_{\lambda,n} + \lambda u_{\lambda,n} = \varphi_n.
\end{equation*}
in the last equality, because $u_\lambda$ is the $\lambda$-potential density.
Furthermore,
\begin{equation}\label{eq:mart_lam}
	N_\lambda^{n,x}(t) = \int_0^t\int_\R \big(u_{\lambda,n}(X_{s-}-x-r)-u_{\lambda,n}(X_{s-}-x)\big)\big(\Pi-\pi\big)(dr,ds).
\end{equation}
Our aim, as said before, is to use \eqref{eq:itoulamn} to let $u_{\lambda}(0)-u_{\lambda,n}(X_t-x)$ go to the limit as $\lambda \rightarrow 0$ and by doing so and using \eqref{eq:v_def} 
derive It\^o's
formula for $v_n(X_t-x)$.
\begin{lemm}\label{lem:ulam_conv}
	Let $m \in \mathbb{N}$ and $(X_t)_{t \geq 0}$ be as defined in \eqref{eq:semimartingale}. Then
	\begin{equation}\label{eq:lem51_conv}
		u_{\lambda}(0)- u_{\lambda,n}(X_{t}-x)  \rightarrow v_n(X_{t}-x)  \quad \bbP_{x_0}\text{-a.s.},
	\end{equation}
	for all $x \in \R$ and $t \geq 0$, as $\lambda \rightarrow 0$.
\end{lemm}
\begin{proof}
	Fix arbitrary $x\in \R$. It holds by \eqref{eq:lam_pot_dens_closed_form} and substituting $x\zeta =\xi$ that
	\begin{equation}\label{eq:u_est_by_v}
		\begin{split}
			0 \leq u_\lambda(0)-u_{\lambda}(x) &=\int_0^\infty \frac{1-\cos(\zeta x)}{\lambda + |\zeta|^\alpha}d\zeta= |x|^{\alpha-1} \int_0^\infty \frac{1-\cos(\xi)}{\lambda|x|^\alpha+ |\xi|^\alpha} d\xi\\
			& \leq |x|^{\alpha-1} \int_0^\infty \frac{1-\cos(\xi)}{|\xi|^\alpha} d\xi = v(x).
		\end{split}
	\end{equation}
	Thus,
	\begin{align*}
		|u_{\lambda}(0)- u_{\lambda,n}(X_{t}-x) |  \leq \int_{\R} \varphi_n(z)v(X_{t}-x-z)dz.
	\end{align*}
	Since $v_n(X_{t}-x)$ is almost surely finite and $u_ {\lambda}(0)-u_{\lambda,n}(x) \rightarrow v_n(x)$ for all $x\in \R$ the convergence \eqref{eq:lem51_conv} follows. This concludes the proof.
\end{proof}
In the next two lemmas we show the convergence as $\lambda \rightarrow 0$ of the terms $\int_0^t(u_{\lambda,n})'(X_s-x)dA_s$ and $\lambda \int_0^t u_{\lambda,n}(X_s-x)ds$ from the right hand side of \eqref{eq:itoulamn}.
\begin{lemm}\label{lem:dA_conv}
	Let $(X_t)_{t \geq 0}$ be defined as in \eqref{eq:semimartingale}. Then 
	\begin{equation}\label{eq:lem52_conv}
		\int_0^{t}(u_{\lambda,n})'(X_s-x) dA_s \rightarrow  \int_0^{t}-(v_n)'(X_s-x) dA_s, \quad \bbP_{x_0}\text{-a.s.},
	\end{equation}
	for all $x \in \R$ and $t \geq 0$, as $\lambda \rightarrow 0$.
\end{lemm}
\begin{proof}
	First, we consider the pointwise difference of the derivatives. Note that by \eqref{eq:lam_pot_dens_closed_form} it holds  that 
	\begin{align*}
		&u_\lambda(x)-u_\lambda(0) +v(x) = \frac{1}{\pi}\int_0^\infty \frac{\cos(\xi x)-1}{\lambda + |\xi|^a} +\frac{1}{\pi}\int_0^\infty \frac{1-\cos(\xi x)}{|\xi|^\alpha}d\xi\\
		&= \frac{1}{\pi} \lambda \int_0^\infty \frac{\cos(\xi x)-1}{(\lambda+ |\xi|^\alpha)|\xi|^\alpha}d\xi, \quad \forall x\in \R.
	\end{align*}
	Thus,
	\begin{equation}\label{eq:diff_gen_pot}
		\begin{split}
			&|(u_\lambda)'(x)+v'(x)| = \frac{1}{\pi} \lambda \Big|\int_0^\infty\frac{\xi \sin(\xi x)}{|\xi|^\alpha(\lambda +|\xi|^\alpha)}d\xi\Big| \\
			& = |x|^{2\alpha-2}\frac{1}{\pi} \lambda \Big|\int_0^\infty \frac{\sin(\xi)\xi}{|\xi|^{\alpha}(\lambda|x|^\alpha +|\xi|^\alpha)}d\xi\Big| \\
			& \leq |x|^{2\alpha-2}\frac{1}{\pi} \lambda  \Big( \int_0^1 \frac{|\xi|^{2-\alpha}}{(\lambda|x|^\alpha +|\xi|^\alpha)}d\xi +\Big|\int_1^\infty \frac{\sin(\xi)\xi}{|\xi|^{\alpha}(\lambda|x|^\alpha +|\xi|^\alpha)}d\xi\Big|\Big), \quad \forall x \in \R.
		\end{split}
	\end{equation}
	The first integral on the right-hand side of \eqref{eq:diff_gen_pot} needs to be treated with delicacy
	\begin{align*}
		&\int_0^1 \frac{|\xi|^{2-\alpha}}{(\lambda|x|^\alpha +|\xi|^\alpha)}d\xi \leq \int_0^1 \frac{1}{(\lambda|x|^\alpha +|\xi|^\alpha)}d\xi\\
		& = |x|^{1-\alpha} \lambda^{1/\alpha -1}\int_0^{(\lambda^{1/\alpha}|x|)^{-1}} \frac{1}{1+|\zeta|^\alpha} d\zeta\\
		& \leq c(\alpha) |x|^{1-\alpha} \lambda^{1/\alpha -1}.
	\end{align*}
	The second integral on the right-hand side of \eqref{eq:diff_gen_pot} can easily be bounded by a constant as is evident from
	\begin{equation*}
		\Big|\int_1^\infty \frac{\sin(\xi)\xi}{|\xi|^{\alpha}(\lambda|x|^\alpha +|\xi|^\alpha)}d\xi\Big|
		\leq \int_1^\infty  \xi^{1-2\alpha}d\xi <\infty.
	\end{equation*}
	Together this yields
	\begin{align*}
		|(u_\lambda)'(x)+v'(x)| \leq c\big( |x|^{\alpha-1} \lambda^{1/\alpha}+ \lambda |x|^{2\alpha-2}\big), \quad \forall x \in \R.
	\end{align*}
	Thus, for almost all $\omega \in \Omega$,
	\begin{equation}\label{eq:conv_deriv_n}
		\begin{split}
			&\int_0^{t}\big|(u_{\lambda,n})'(X_s-x)+(v_n)'(X_s-x)\big|d \bar{A}_s\\
			&  \leq c(\alpha) \int_0^{t} \int_\R \phi_n(z) \big(\lambda^{1/\alpha} |X_s-x-z|^{\alpha-1} +\lambda |X_s-x-z|^{2\alpha-2}\big)dzd\bar{A}_s\\
			& \leq c(\alpha)\sup_{0 \leq s \leq t}\int_\R \phi_n(z) \big( \lambda^{1/\alpha} |X_s-x-z|^{\alpha-1} + \lambda|X_s-x-z|^{2\alpha-2}\big)dz \bar{A}_t, 
		\end{split}
	\end{equation}
	for all $x\in \R$.
	Since $\sup_{0\leq s \leq t}|X_s|$ and $\bar{A}_t$ are almost surely finite the convergence \eqref{eq:lem52_conv} follows. 
	This finishes the proof.
\end{proof}
\begin{lemm}\label{lem:lamulam_conv}
	Let $(X_t)_{t \geq 0}$ be defined as in \eqref{eq:semimartingale}.
	Then
	\begin{equation*}
		\lambda \int_0^{t}u_{\lambda,n}(X_s-x)ds \rightarrow 0 \quad \bbP_{x_0}-a.s.,
	\end{equation*}
	for all $x \in \R$ and $t \geq 0$, as $\lambda \rightarrow 0$.
\end{lemm}
\begin{proof}
	Plugging in the definition of $u_{\lambda,n}$ yields
	\begin{equation}\label{eq:lem5.3:1}
		\begin{split}
			&\int_0^{t} \big|u_{\lambda,n}(X_s-x)\big|ds\\
			&\leq\int_0^{t} |u_{\lambda,n}(X_s-x)-u_{\lambda}(0)|ds+ tu_{\lambda}(0)\\
			&\leq \int_\R \phi_n(z) \int_0^{t} |X_s-x-z|^{\alpha-1}dsdz \int_0^\infty \frac{1-\cos(\xi)}{|\xi|^\alpha}d\xi  + t u_{\lambda,n}(0)\\
			&\leq C(t,n,x,\alpha, \omega )+ t u_{\lambda,n}(0),
		\end{split}
	\end{equation}
	where $C(t,n,x,\alpha, \omega ) < \infty$ almost surely, since $\sup_{0 \leq s \leq t}|X_t|<\infty$ almost surely .
	Furthermore, using the substitution $\xi = \lambda^{1/\alpha}\zeta$ we arrive at  
	\begin{equation}
		\begin{split}\label{eq:lem5.3:2}
			u_{\lambda}(0) 
			\leq c \int_0^\infty \frac{1}{\lambda + |\xi|^\alpha}d\xi = c\lambda^{1/\alpha-1}\int_0^\infty\frac{1}{1+ |\zeta|^\alpha}d\zeta.
		\end{split}
	\end{equation}
	Using \eqref{eq:lem5.3:1} and \eqref{eq:lem5.3:2} leads to
	\begin{equation}\label{eq:exp_vn_lam}
		\lambda \int_0^{t} \big|u_{\lambda,n}(X_s-x)\big|ds \leq C(t,n,x,\alpha, \omega)\Big(\lambda + \lambda^{1/\alpha}\int_0^\infty\frac{1}{1+ |\zeta|^\alpha}d\zeta\Big).
	\end{equation}
	This implies the convergence of the left hand side of \eqref{eq:exp_vn_lam} to zero as $\lambda \rightarrow 0$, since $\int_0^\infty \frac{1}{1+ |\zeta|^\alpha}d\zeta <\infty$. This concludes the proof.
\end{proof}
In the next lemma, we finally derive the Ito formula for $v_n(X_t-x)$ by using Lemmas~\ref{lem:ulam_conv}, \ref{lem:dA_conv} and \ref{lem:lamulam_conv}  and taking the limit of the martingale term \eqref{eq:mart_lam} as $\lambda \rightarrow 0$. This gives us the following result. 
\begin{lemm}\label{lem:Ito_tan_func}
	Let $(X_t)_{t \geq 0}$ be defined as in \eqref{eq:semimartingale} and let $t \geq 0$. Then it holds $\bbP_{x_0}$-a.s. that
	\begin{equation} \label{eq:Itov_n}
		v_n(X_t-x) = v_n(x_0-x) + N^{n,x}_t + \int_0^t (v_n)'(X_s-x) dA_s + \int_0^t\phi_n(X_s-x)ds,
	\end{equation}
	where $(X_t)_{t \geq 0}$ and $ (A_t)_{t \geq 0}$ are defined as in \eqref{eq:semimartingale} and $(N^{n,x}_t)_{t \geq 0}$ is a local martingale given by
	\begin{equation}\label{eq:smooth_martingale}
		N^{n,x}_t = \int_0^t \int_\R  \big(v_n(X_{s-}-x+r)-v_n(X_{s-}-x)\big)\big(\Pi-\pi\big)(dr,ds)
	\end{equation}
	with $\Pi$ being the Poisson random measure of the symmetric $\alpha$-stable process $L$ with intensity $\pi$ as defined in \eqref{eq:intensity}.
\end{lemm}
\begin{proof}
	First set 
	\begin{equation*}
		\tau_m = \inf\{t \geq 0 : \bar{A}_t >m\}.
	\end{equation*}
	By plugging $t\wedge\tau_m$ into \eqref{eq:itoulamn} we arrive at 
	\begin{align*}
		&u_{\lambda,n}(X_{t\wedge\tau_m}-x)\\
		&= u_{\lambda,n}(x_0-x) +N_\lambda^{n,x}(t\wedge\tau_m) +\int_0^{t\wedge\tau_m} (u_{\lambda,n})'(X_s-x)dA_s\\
		&\quad  - \int_0^{t\wedge\tau_m}\phi_n(X_s-x) ds
		+ \lambda \int_0^{t\wedge\tau_m} u_{\lambda,n}(X_s-x)ds.
	\end{align*}
	Now, we consider the difference $u_{\lambda}(0)- u_{\lambda,n}(X_t\wedge\tau_m-x)$ in order to apply Lemma~\ref{lem:ulam_conv} and obtain 
	\begin{equation}\label{eq:Ito_molli}
		\begin{split}
			&u_{\lambda}(0)- u_{\lambda,n}(X_{t\wedge\tau_m}-x) = u_{\lambda}(0)- u_{\lambda,n}(x_0-x) -N_\lambda^{n,x}(t\wedge\tau_m) \\
			&-\int_0^{t\wedge\tau_m} (u_{\lambda,n})'(X_s-x)dA_s + \int_0^{t\wedge\tau_m}\phi_n(X_s-x) ds -  \lambda \int_0^{t\wedge\tau_m} u_{\lambda,n}(X_s-x)ds.
		\end{split}
	\end{equation}
	Now, Lemma~\ref{lem:ulam_conv}, Lemma~\ref{lem:lamulam_conv} and Lemma~\ref{lem:dA_conv} demonstrate that all
	terms in \eqref{eq:Ito_molli} besides the martingale term converge on the event $\{t \leq \tau_m\}$ $\bbP_{x_0}$-almost surely 
	to the corresponding term in \eqref{eq:Itov_n} stopped at $t \wedge\tau_m$, as $\lambda \rightarrow 0$. 
	Since the probability of the event $\{t \leq \tau_m\}$ converges to one, 
	it remains to identify the limit of the martingale term. To this end consider
	\begin{equation*}
		N^{n,x}_ {t\wedge \tau_m}= \int_0^t \int_\R  \mathbbm{1}_{s\leq \tau_m}\big(v_n(X_{s-}-x+r)-v_n(X_{s-}-x)\big)\big(\Pi-\pi\big)(ds,dr)
	\end{equation*}
	and note that by It\^o's-isometry (see \cite[Theorem 4.2.3]{Applebaum09})
	\begin{align*}
		&\mathbb{E}_{x_0}[(N_\lambda^{n,x}(t\wedge \tau_m))^2]\\
		&=  \mathbb{E}_{x_0}\Big[\int_0^t\int_\R \mathbbm{1}_{s\leq\tau_m}\big(u_{\lambda,n}(X_{s}-x+r)-u_{\lambda,n}(X_{s}-x)\big)^2\tilde{\pi}(dr)ds\Big].
	\end{align*}
	We now proceed to proof the $L^2$ convergence of $N_\lambda^{n,x}(t\wedge \tau_m)$ to 
	$N^{n,x}_{t\wedge\tau_m}$.
	To this end, we first derive bounds that are independent of $\lambda$, in order to apply dominated convergence.
	First assume $|r|<1$ and note that by \eqref{eq:deriv_u_lam} and the mean value theorem
	\begin{equation}\label{eq:mart_est_1}
		\begin{split}
			&|u_{\lambda,n}(X_{s}-x+r)-u_{\lambda,n}(X_{s}-x)|\\
			&\leq |r|\sup_{y\in I_{s,x,r}}|u_{\lambda,n}'(y)|\\
			&\leq c|r|\sup_{y\in I_{s,x,r}}\int_\R \varphi_n(z)|y-z|^{\alpha-2}\int_0^\infty \frac{\xi \sin(\xi)}{\lambda|y-z|^\alpha + |\xi|^\alpha}d\xi dz\\
			& \leq c(\alpha,n)|r|\sup_{y\in I_{s,x,r}}(|y-2^{-n}|^{\alpha-1}+|y+2^{-n}|^{\alpha-1})\\
			&\leq c(\alpha,n)|r|(|X_s-x|^{\alpha-1}+1),
		\end{split}
	\end{equation}
	where $I_{s,x,r} = [X_s-x,X_s-x+r]\cup[X_s-x+r,X_s-x]$ and we used Lemma~\ref{lemm:help_bounded} in the third inequality.
	Now assume that $|r|\geq 1$. Using \eqref{eq:u_est_by_v} leads to
	\begin{equation}\label{eq:mart_est_2}
		\begin{split}
			&|u_{\lambda,n}(X_{s}-x+r)-u_{\lambda,n}(X_{s}-x)| \\
			&\leq 	\Big|\int_\R \phi_n(z)(u_\lambda(X_s-x+r-z)-u_\lambda(0))dz\Big| \\
			& \qquad + 	\Big|\int_\R \phi_n(z)(u_\lambda(X_s-x-z)-u_\lambda(0))dz\Big|\\
			& \leq |v_n(X_s-x+r)| + |v_n(X_s-x)| \\
			&\leq c(n) (|r|^{\alpha-1}+|X_s-x|^{\alpha-1}+1),
		\end{split}
	\end{equation}
	where we used $\int_\R \phi_n(z)|x-z|^{\alpha-2}dz \leq c(n)(|x|^{\alpha-1}+1)$ in the last inequality.
	For the L\'evy measure $\tilde{\pi}$ it holds that 
	\begin{equation*}
		\int_\R |r|^2\wedge1 \tilde{\pi}(dr) \leq \int_\R |r|^2\wedge|r|^{2\alpha-2} \tilde{\pi}(dr)<\infty,
	\end{equation*}
	since $1+\alpha-2\alpha+2 = 3-\alpha>1$.
	Furthermore, note that 
	\begin{equation}\label{eq:expectation_est1}
		\begin{split}
			&\sup_{0 \leq s < t\wedge \tau_m}|X_s|^{\alpha-1} \leq \sup_{0 \leq s < t\wedge \tau_m}( 1\vee|X_s|^{2\alpha-2}) \leq \sup_{0 \leq s < t\wedge \tau_m}|X_s|^{2\alpha-2}+ 1\\
			& \leq 	\sup_{0 \leq s <t\wedge \tau_m}\bar{A}_s^{2\alpha-2} + \sup_{0 \leq s <t\wedge \tau_m}|L_s|^{2\alpha-2} +1.
		\end{split}
	\end{equation}
	By Doob's martingale inequality  and since $2\alpha-2 \in (0,\alpha)$ for $\alpha \in (1,2)$ we have
	\begin{equation}\label{eq:expectation_est2}
		\begin{split}
			&\mathbb{E}_{x_0}\Big[ \sup_{0 \leq s\leq t }|L_s|^{2\alpha-2}\Big]<\infty.
		\end{split}
	\end{equation}
	Together, \eqref{eq:expectation_est1} and \eqref{eq:expectation_est2} imply
	\begin{equation}\label{eq:expect_est3}
		\mathbb{E}_{x_0}\Big[ \sup_{0 \leq s < t\wedge \tau_m-}|X_s|^{\alpha-1} \Big] \leq \mathbb{E}_{x_0}\Big[ \sup_{0 \leq s < t\wedge \tau_m-}|X_s|^{2\alpha-2} \Big] +1 \leq C(m),
	\end{equation}
	since $\sup_{0 \leq s <t\wedge \tau_m-}\bar{A}_s^{2\alpha-2} = \bar{A}_{t_\wedge (\tau_m)_-}^{2\alpha-2} \leq m^{2\alpha-2}$.
	For this reason, \eqref{eq:mart_est_1} and \eqref{eq:mart_est_2} together with \eqref{eq:expect_est3} imply
	\begin{equation}\label{eq:L2uLamnmart}
		\begin{split}
			&\mathbb{E}_{x_0}[(N_\lambda^{n,x}(t\wedge \tau_m))^2] \\
			&\leq c(\alpha,n)\mathbb{E}_{x_0}\Bigg[\int_0^{t \wedge\tau_m }\int_\R (|r|(|X_s-x|^{\alpha-1}+1))^2\wedge (|r|^{\alpha-1}\\
			& \hspace{5cm}+|X_s-x|^{\alpha-1}+1)^2\tilde{\pi}(dr)ds\Bigg]<\infty.
		\end{split}
	\end{equation}
	Using Jensen's inequality, Fubini's theorem, and the substitution $r = (X_s-x-z)\tilde{r}$ we obtain
	\begin{equation}\label{eq:v_n_diff_est}
		\begin{split}
			&\int_\R \big(v_n(X_{s}-x+r)-v_n(X_{s}-x)\big)^2\tilde{\pi}(dr)\\
			&\leq \int_\R \varphi_n(z) \int_\R \big(v(X_{s}-x+r-z)-v(X_{s}-x-z)\big)^2\tilde{\pi}(dr)dz\\
			& = c(\alpha)\int_\R \varphi_n(z)|X_s-x-z|^{\alpha-2}dz \int_\R (|1+r|^{\alpha-1}-1)^2\tilde{\pi}(dr)\\
			& \leq c(\alpha,n) \big(|X_s-x|^{\alpha-1}+1\big).
		\end{split}
	\end{equation}
	Using \eqref{eq:expect_est3} again, this leads to
	\begin{equation}\label{eq:L2boundvnmart}
		\mathbb{E}_{x_0}\Big[\int_0^{t\wedge \tau_m}\int_\R \big(v_n(X_{s}-x+r)-v_n(X_{s}-x)\big)^2\tilde{\pi}(dr)ds\Big]   < \infty, \quad \forall t\geq 0.
	\end{equation}
	Applying Doob's martingale inequality yields
	\begin{equation}\label{eq:vnmartconv}
		\begin{split}
			&\mathbb{E}_{x_0}\Big[\sup_{0 \leq s \leq t\wedge \tau_m} \big( N_\lambda^{n,x}(s)+N^{n,x}_s\big)^2\Big] \\
			&\leq c
			\mathbb{E}_{x_0}\Big[\int_0^{t \wedge \tau_m} \int_\R \big(u_{\lambda,n}(X_{s}-x+r)-u_{\lambda,n}(X_{s}-x)\\
			&\hspace{3.5cm}+v_n(X_{s}-x+r)-v_n(X_{s}-x)\big)^2\tilde{\pi}(dr)ds\Big].
		\end{split}
	\end{equation}
	Recall that $u_\lambda(0)-u_{\lambda,n}(x)\rightarrow v_n(x)$ as $\lambda \rightarrow 0$ for any $x\in \R$ and therefore the right hand side of \eqref{eq:vnmartconv} converges to zero, as $\lambda \rightarrow 0$, by \eqref{eq:L2boundvnmart},\eqref{eq:L2uLamnmart} and dominated convergence.
	Furthermore, it holds that $\tau_m \rightarrow \infty$ a.s. because $\bar{A}_t$ is a.s. finite. We can thus deduce that $(N^{n,x}_{t })_{t \geq 0}$ is a local martingale .
	This finishes the proof. 
\end{proof}
\subsection{Convergence Results}
In this section, we are interested in passing to the limit of 
\begin{equation*}
	\int_0^t\phi_n(X_s-x)ds = v_n(X_t-x)- v_n(x_0-x) -N^{n,x}_t-\int_0^t (v_n)'\big(X_s-x\big)dA_s, \quad t \geq 0, \quad x\in \R
\end{equation*}
 from Lemma~\ref{lem:Ito_tan_func} as $n \rightarrow \infty$.
\begin{lemm}\label{lem:v_nconv}
	Let $X$ satisfy \eqref{eq:semimartingale} and let $\nu \in K^f_{\alpha-1}$. Then it holds for every $t \geq 0$ that
	\begin{equation*}
		\sup_{0 \leq s \leq t}\Big|\int_\R v_n(X_s-x) -v(X_s-x) \nu(dx)\Big| \rightarrow 0 \quad \bbP_{x_0}\text{-a.s.}
	\end{equation*}
	as $n \rightarrow \infty$.
\end{lemm}
\begin{proof}
	Note that
	\begin{align*}
		\int_\R v_n(X_s-x)\nu(dx) = \int_\R \varphi_n(z) \int_\R v(X_s-x-z) \nu(dx)dz, \quad \forall s\geq 0, x\in \R.
	\end{align*}
	If we look at the function $f(y) = \int_\R v(y-x)\nu(dx)$, it is clear, that we investigate the 
	convergence of the mollification of this function. Since $f$ is continuous this implies the
	uniform convergence on compacts. Since for $\bbP_{x_0}$-a.e. $\omega$, $(X_s)_{0 \leq s \leq t}$ is 
	contained in some compact subset $K_\omega$, the almost sure convergence follows. 
\end{proof}
\begin{lemm}\label{lem:drift}
	Let $X$ satisfy \eqref{eq:semimartingale} and $\nu \in K^f_{\alpha-1}$. Then it holds for every $t > 0$ that, as $n \rightarrow \infty$,
	\begin{equation*}
		\int_\R\int_{0}^{t}\big|(v_n)'(X_s-x)-v'(X_s-x)\big|d\bar{A}_s|\nu|(dx) \rightarrow 0, \quad \bbP_{x_0}\text{-a.s.}
	\end{equation*}
\end{lemm}
\begin{proof}
	For the sake of notational convenience, we set $w = v'$ and denote by $w_n$ its mollification.  
	Let $\epsilon > 0$ be arbitrary. Let $R >0$, and we will fix its value later. Let $n$ be large enough so that $2^{-n}< R$. Since $M_{\nu}(r) \rightarrow 0$ for $r \rightarrow 0$ we obtain
	\begin{align*}
		&\int_0^t \int_\R \big|w_n(X_s-x)-w(X_s-x)\big|d\bar{A}_s|\nu|(dx) \\
		&\leq \int_0^t \int_{B(X_s,R)}\big|w_n(X_s-x)-w(X_s-x)\big||\nu|(dx)d\bar{A}_s \\
		& \quad + \int_0^t \int_{B(X_s,R)^c}\big|w_n(X_s-x)-w(X_s-x)\big||\nu|(dx)d\bar{A}_s\\
		& \leq c(\alpha) \Big(\int_0^t \int_{B(X_s,R)}|X_s-x|^{\alpha-2}\nu(dx)d\bar{A}_s \\
		& \quad + \int_0^t \int_\R \varphi_n(z) \int_{B(X_s-z,R+2^{-n})}|X_s-x-z|^{\alpha-2}|\nu|(dx)dzd\bar{A}_s \\
		&  \quad + \int_0^t \int_{B(X_s,R)^c}|w_n(X_s-x)-w(X_s-x)||\nu|(dx)d\bar{A}_s\Big)\\
		& \leq c\big(M_\nu^{\alpha-1}(R) + M_\nu^{\alpha-1}(2R)\big) \\
		& \quad + \int_0^t \int_{B(X_s,R)^c}\big|w_n(X_s-x)-w(X_s-x)\big||\nu|(dx)d\bar{A}_s\\
		& \leq \epsilon + \int_0^t \int_{B(X_s,R)^c}\big|w_n(X_s-x)-w(X_s-x)\big||\nu|(dx)d\bar{A}_s, 
	\end{align*}
	$\mathbb{P}_{x_0}$-almost surely, where $R$ is so that $c\big(M_\nu^{\alpha-1}(R) + M_\nu^{\alpha-1}(2R)\big) <\epsilon$. Note that
	\begin{equation*}
		|w_n(x)| = \Big|\int_\R w(x-z) \varphi_n(z)dz \Big|\leq c(\alpha) |R-2^{-n}|^{\alpha-2} , \quad x \in B(0,R)^c.
	\end{equation*} 
	This means that 
	\begin{equation*}
		\lim_{n \rightarrow \infty}\int_0^t \int_{B(X_s,R)^c}\big|w_n(X_s-x)-w(X_s-x)\big||\nu|(dx)d\bar{A}_s =0, \quad \mathbb{P}_{x_0}\text{-a.s.}
	\end{equation*}
	by the bounded convergence theorem and convergence of $w_n(x)$ to $w(x)$ for every $x\neq 0$.
	This finishes the proof.
\end{proof}
\begin{lemm}\label{lem:l2bound}
	Let $X$ satisfy \eqref{eq:semimartingale} , $\nu \in K_{\alpha-1}^f$ and
	let $(N^x_t)_{t \geq 0}$ and $(N_t^{n,x})_{t\geq 0}$ for $x \in \R$ be defined as in \eqref{eq:martingale} and \eqref{eq:smooth_martingale}, respectively.
	Then it holds for every $t \geq 0$ that
	\begin{equation*}
		\bbE_{x_0}\Big[\sup_{0\leq s\leq t }\Big(\int_\R \big(N^{x}_s-N^{n,x}_s\big) \nu(dx)\Big)^2\Big] \rightarrow 0,
	\end{equation*}
	as $n \rightarrow \infty$.
\end{lemm}
\begin{proof}
	Note that by substituting $r = (X_{s-}-x)\tilde{r}$ in the same way as in \eqref{eq:v_n_diff_est}, we get
	\begin{align*}
		&\int_\R\bbE_{x_0}\Big[\sup_{0 \leq t \leq T} \int_0^t\int_\R \big(v(X_{s-}-x+r) -v(X_{s-}-x)\big)^2 \tilde{\pi}(dr)ds \Big]|\nu|(dx) \\
		& \leq c(\alpha)\int_0^T\bbE_{x_0}\Big[\int_\R  |X_{s-}-x|^{\alpha-2}|\nu|(dx) \Big ]ds < \infty,
	\end{align*}
	for all $T>0$. Thus, for $\nu$-almost every $x$ the process $(N^x_t)_{0 \leq t \leq T}$ is a martingale.
	In a similar manner, we can show that $(N^{n,x}_t)_{0 \leq t \leq T}$ is a martingale for $\nu$-almost every $x$.
	Thus, H\"older's inequality and  Doob's martingale inequality yield
	\begin{equation}\label{eq:mart_est}
		\begin{split}
			&\bbE_{x_0}\Big[\sup_{0\leq s\leq t }\Big(\int_\R N^x_s-N^{n,x}_s \nu(dx)\Big)^2\Big] \\
			& \leq c(\nu) \bbE_{x_0}\Big[\sup_{0\leq s\leq t }\int_\R \big(N^x_s-N^{n,x}_s\big)^2 |\nu|(dx)\Big] \\
			& \leq c(\nu) \int_\R \bbE_{x_0}\Big[\sup_{0 \leq s \leq t}\big(N^x_s-N^{n,x}_s\big)^2\Big] |\nu|(dx) \\
			& \leq c(\nu) \int_\R \bbE_{x_0}\Big[\big(N^x_t-N^{n,x}_t\big)^2\Big] |\nu|(dx).
		\end{split}
	\end{equation}
	Applying It\^o's isometry to the right hand side of \eqref{eq:mart_est} leads to 
	\begin{align*}
		\bbE_{x_0}\Big[\sup_{0\leq s\leq t }\Big(\int_\R N^x_s-N^{n,x}_s \nu(dx)\Big)^2\Big] 
		& \leq c(\nu,\alpha) \bbE_{x_0}\Big[\int_\R \int_0^t f_n(\omega,s,x) ds |\nu|(dx)\Big],
	\end{align*}
	where 
	\begin{align*}
		&f_n(\omega,s,x)
		= \int_\R
		\big(v(X_{s-}-x+r)-v(X_{s-}-x)\\
		& \hspace{5cm}-v_n(X_{s-}-x+r)+v_n(X_{s-}-x)\big)^2\frac{1}{|r|^{1+\alpha}}dr.
	\end{align*}
	Now consider the following decomposition 
	\begin{equation}\label{eq:decomp_f_n}
		\begin{split}
			\bbE_{x_0}\Big[\int_\R \int_0^t f_n(\omega,s,x) ds |\nu|(dx)\Big] 
			&\leq \bbE_{x_0}\Big[\int_0^t \int_{B(X_s,R)}f_n(\omega,s,x)|\nu|(dx)ds]\\
			& \hspace{0.4cm} +\bbE_{x_0}\Big[\int_0^t\int_{B(X_s,R)^c} f_n(\omega,s,x)  |\nu|(dx)ds \Big],
		\end{split}
	\end{equation}
	where $R >0$. We will fix the value of $R$ later.
	We first estimate the first term on the right-hand side of \eqref{eq:decomp_f_n}. To this end notice that for fixed $y\in \R$ the following holds
	\begin{align*}
		&\big(v_n(y+r)-v_n(y)\big)^2 = \Big(\int_R \varphi_n(z)\big(v(y+r-z)-v(y-z)\big)dz\Big)^2\\
		& = c(\alpha)\Big(\int_\R\varphi_n(z)\big(|y+r-z|^{\alpha-1}-|y-z|^{\alpha-1}\big)dz\Big)^2\\
		&\leq c(\alpha) \int_\R\varphi_n(z)\big(|y+r-z|^{\alpha-1}-|y-z|^{\alpha-1}\big)^2dz.
	\end{align*}
	Integrating with respect to $\frac{1}{|r|^{1+\alpha}}dr$ and substitution yields
	\begin{align*}
		&\int_\R \varphi_n(z) \int_\R\big(|y+r-z|^{\alpha-1}-|y-z|^{\alpha-1}\big)^2\frac{1}{|r|^{1+\alpha}}drdz \\
		& = \int_\R \varphi_n(z) \int_\R\big(|y-z|^{\alpha-1}|1+u|^{\alpha-1}-|y-z|^{\alpha-1}\big)^2\frac{1}{|(y-z)u|^{1+\alpha}}|y-z|dudz\\
		& = \int_\R \varphi_n(z) |y-z|^{\alpha-2} dz\int_\R \big(|1+u|^{\alpha-1}-1\big)^2\frac{1}{|u|^{1+\alpha}}du\\
		& = c(\alpha)\int_\R \varphi_n(z) |y-z|^{\alpha-2} dz.
	\end{align*}
	Letting $\epsilon > 0$, choosing $R >0$ small enough, so that $M_\nu(2R)< \epsilon$,  and $n$ big enough so that $2^{-n}< R/4$
	we use the above to get %be arbitrary  and plugging this in yields for $2^{-n} <R$ small enough
	\begin{align*}
		&\int_{B(X_s,R)}\int_\R \varphi_n(z) f_n(\omega,s,x) dz|\nu|(dx) \\
		& \leq c(\alpha)\Big(	\int_{B(X_s,R)}\int_\R \varphi_n(z) |X_s-z-x|^{\alpha-2} dz|\nu|(dx) +\int_{B(X_s,R)} |X_s-x|^{\alpha-2}| \nu|(dx)\Big) \\
		& \leq c(\alpha) \int_{B(X_s,R)}\int_\R \varphi_n(z) |X_s-z-x|^{\alpha-2} dz|\nu|(dx) + \epsilon \\
		& \leq c(\alpha) \int_\R \varphi_n(z)\int_{B(X_s-z,R+2^{-n})}|X_s-z-x|^{\alpha-2}| \nu|(dx) dz + \epsilon \leq 2 \epsilon,
	\end{align*}
	where we used $M_\nu(2R)< \epsilon$ in the second and fourth inequality.
	Now we turn to the second term on the right-hand side of \eqref{eq:decomp_f_n}. Note that 
	for $y \in \R, y \neq 0$, the following holds 
	\begin{align*}
		&\int_{\R} \big(v(y+r)-v(y)-v_n(y+r)+v_n(y)\big)^2 \frac{dr}{|r|^{1+\alpha}} \\
		& = c(\alpha)\int_\R \big(|y+r|^{\alpha-1}-|y|^{\alpha-1}-(\varphi_n \ast |\cdot|^{\alpha-1})(y+r)+(\varphi_n \ast |\cdot|^{\alpha-1})(y)\big)^2\frac{dr}{|r|^{1+\alpha}}\\
		& = c(\alpha)|y|^{\alpha-2} \int_{\R} \Big(|1+u|^{\alpha-1}-1-|y|^{1-\alpha}\Big(\varphi_n \ast \big(|\cdot(1+u)|^{\alpha-1}-|\cdot|^{\alpha-1}\big)\Big)(y)\Big)^2 \frac{du}{|u|^{1+\alpha}}.
	\end{align*}
	Furthermore, note that 
	\begin{align*}
		&(\varphi_n \ast |\cdot|^{\alpha-1})(y)= \int_\R \varphi_n(z) |y-z|^{\alpha-1}dz \\
		&=|y|^{\alpha-1}\int_\R |y|\varphi_n(yu) |1-u|^{\alpha-1}du \\
		&=|y|^{\alpha-1}\int_\R |y|\varphi_n(|y|u) v(1-u)du \\
		&= |y|^{\alpha-1} \int_\R \varphi^{(|y|^{-1}2^{-n})}(u)|1-u|^{\alpha-1}du
	\end{align*}
	and similarly
	\begin{align*}
		&(\varphi_n \ast |\cdot|^{\alpha-1})(y(1+u))= \int_\R \varphi_n(z) |y(1+u)-z|^{\alpha-1}dz \\
		&=|y|^{\alpha-1}\int_\R |y|\varphi_n(|y|u) v(1+u-r)dr \\
		&= |y|^{\alpha-1} \int_\R \varphi^{(|y|^{-1}2^{-n})}(u)v(1+u-r)dr.
	\end{align*}
	Thus, for all $x \in \R$ it follows that
	\begin{align*}
		&0 \leq f_n(\omega,s,x)\mathbbm{1}_{B(X_s,R)^c}(x)\\
		&= \int_{\R} \big(v(X_s-x+r)-v(y)-v_n(X_s-s+r)+v_n(y)\big)^2 \frac{dr}{|r|^{1+\alpha}} \mathbbm{1}_{B(X_s,R)^c}(x) \\
		&= \mathbbm{1}_{B(X_s,R)^c}|X_s-x|^{\alpha-2}\\
		& \qquad \times \int_\R \big(v(1+r)-v(1)-v^{(|X_s-x|^{-1}2^{-n})}(1+r)+v^{(|X_s-x|^{-1}2^{-n})}(1)\big)^2 \frac{dr}{|r|^{1+\alpha}}\\
		& \leq R^{\alpha-2}\int_\R \big(v(1+r)-v(1)-v^{(|X_s-x|^{-1}2^{-n})}(1+r)+v^{(|X_s-x|^{-1}2^{-n})}(1)\big)^2 \frac{dr}{|r|^{1+\alpha}}.
	\end{align*}
	Note that for $0<\delta < 1/4$ it follows
	\begin{equation}\label{eq: v_bound}
		\begin{split}
			&\big(v^{(\delta)}(1+r) -v^{(\delta)}(1) \big)^2\\
			&\leq \big(|r|\sup_{s\in (1,1+r)\cup (1+r,1)}|(v^{(\delta)})'(s)|\big)^2\mathbbm{1}_{|r|\leq 1/2} + \big(v^{(\delta)}(1+r) -v^{(\delta)}(1) \big)^2\mathbbm{1}_{|r|>1/2}\\
			&\leq c(\alpha)\Big(  |r|^24^{4-2\alpha}\mathbbm{1}_{|r|\leq 1/2} + (|1+r|^{2\alpha-2}+1)\mathbbm{1}_{|r|>1/2}\Big) =: \tilde{v}(r)
		\end{split}
	\end{equation}
	for all $r \in \R$. Hereby, we used the inequality $|a+b|^p \leq 2^{p-1}(|a|+|b|)$, for all $p \in (1,\infty)$ and all $a,b \in \R$ in the last step.
	Since $\tilde{v}$ is integrable with respect to $\frac{1}{|r|^{1+\alpha}}dr$
	it follows by dominated convergence that $f_n(\omega,s,x)\mathbbm{1}_{B(X_s,R)^c}$
	converges for almost all $(\omega,s,x)$ to zero. \\
	Finally it follows for $2^{-n}< R/4$ that
	\begin{equation*}
		|f_n(\omega,s,x)\mathbbm{1}_{B(X_s,R)^c}(x)|\leq c R^{\alpha-2} \int_\R \tilde{v}(x)\frac{1}{|x|^{1+\alpha}}dx =: \tilde{f}(\omega,s,x)
	\end{equation*}
	and 
	\begin{equation*}
		\bbE_{x_0}\Big[\int_0^t\int_\R \tilde{f}(\omega,s,x)|\nu|(dx)ds\Big] < \infty.
	\end{equation*}
	This demonstrates that the second term on the right-hand side of \eqref{eq:decomp_f_n} converges to zero as $n$ goes to infinity. Since $\epsilon > 0$ was arbitrary this proves the claim.
\end{proof}
Theorem ~\ref{lem: convergence} is a direct consequence of
Lemma\ref{lem:v_nconv}, Lemma~\ref{lem:drift}, and Lemma~\ref{lem:l2bound} together with Lemma~\ref{lem:Ito_tan_func}.

\section{Proof of Theorem~\ref{th:sdeilt}} \label{sec:sdeilt}
Before we prove Theorem~\ref{th:sdeilt}, we derive the following lemma concerning the sample path properties of the drift.
\begin{lemm}\label{lem:cont_finite_var}
	Let $X$ satisfy \eqref{eq:semimartingale} and
	let  $\gamma$ be its local time. Furthermore, let $\nu \in K^f_{\alpha-1}$ Then the  process
	\begin{equation}\label{eq:drift}
		V_t = \int_\R \gamma_t^x\nu(dx), \quad t \geq 0,
	\end{equation}
	is of finite variation and has a continuous version.
\end{lemm}
\begin{proof}
	Note,
	\begin{equation*}
		V_t = \int_\R\gamma_t^x \nu(dx)= \int_\R \gamma_t^x \nu^+(dx) - \int_\R \gamma_t^x \nu^-(dx)=: V_t^+-V_t^-, \quad t \geq 0.
	\end{equation*}
	Both $\mu^+$ and $\mu^-$ are nonnegative, finite measures in $K_{\alpha-1}$. Due to Proposition~\ref{lem: convergence} there exists a subsequence $(n_\ell)_{\ell \in \mathbb{N}}$ such that 
	\begin{equation*}
		\int_0^t \int_\R \varphi_{n_\ell}(X_s-x)\nu^+(dx)ds, \quad t \geq 0,
	\end{equation*}
	and 
	\begin{equation*}
		\int_0^t \int_\R \varphi_{n_\ell}(X_s-x)\nu^-(dx)ds, \quad t \geq 0,
	\end{equation*}
	converge uniformly on compacts $\bbP_{x_0}$-a.s. to $\tilde{A}_t^+$ and $\tilde{A}_t^-$, respectively.
	This implies that there 
	exist continuous versions of $V_t^+$ and $V_t^-$ since $C([0, \infty))$ endowed with the metric 
	\begin{equation*}
		d(f,g) = \sum_{n = 1}^\infty 2^{-n}1\wedge\|f-g\|_{[0,n], \infty}, \quad f,g \in C([0,\infty))
	\end{equation*}
	is complete. Furthermore note that for all $t\geq s \geq 0$ and all $\ell \in \mathbb{N}$
	\begin{equation*}
		\int_s^t\int_\R \varphi_{n_\ell}(X_\tau-x)\mu^+(dx)d\tau \geq 0.
	\end{equation*}
	It follows that $V^+$ is monotone increasing. Similarly, it can be shown that the same is true
	for $V^-$. The Jordan decomposition thus demonstrates that $V$ has finite variation. 
\end{proof}
Let $X$ be as in \eqref{eq:semimartingale}. Then it follows from Proposition~\ref{lem: convergence} that there exists a subsequence $n_k$ such that $\int_0^t \mu_{n_k}(X_s)ds$ converges uniformly on compacts $\bbP_{x_0}$-a.s.. Hence,  for $\bbP$-almost every $\omega\in \Omega$ we have
\begin{equation}\label{eq:subseq}
	\begin{split}
		&\sup_{k}\int_0^t |\mu_{n_k}(X_s) |ds \leq  \sup_{k}\int_0^t \mu_{n_k}^+(X_s)ds + \sup_{k}\int_0^t \mu_{n_k}^-(X_s)ds \\
		&\leq \sup_k \Big|\int_0^t \mu_{n_k}^+(X_s)ds - A_t^+\Big| +\sup_k \Big|\int_0^t \mu_{n_k}^-(X_s)ds - A_t^-\Big| + A_t^++A_t^- < \infty.
	\end{split}
\end{equation}
\begin{proof}[Proof of Theorem~\ref{th:sdeilt}]
	First, we prove that every weak solution to \eqref{eq:singularSDE} is a weak solution to \eqref{eq:SDE}.  Let $X = x_0+L+A$ be a weak solution of \eqref{eq:singularSDE}.  By Theorem~\ref{th:sde_exist}
	the process $(A_t)_{t \geq 0}$
	is a continuous process of finite variation.
	Then it follows by Proposition~\ref{lem: convergence} that $A_t = \int_\R \gamma_t^x\mu(dx)$ almost surely.
	Thus $X$ is a solution of \eqref{eq:SDE}.
	Now we prove the other direction. 
	Suppose that $X$ is a weak solution of \eqref{eq:SDE} . Then it follows by  Proposition~\ref{lem: convergence}, Lemma~\ref{lem:cont_finite_var} and \eqref{eq:subseq},
	that $X$ is a weak solution of \eqref{eq:singularSDE}.  This means, that every solution to \eqref{eq:SDE} is a solution 
	tu \eqref{eq:singularSDE} and vice versa and thus proves \ref{itm:equiv}. \\
	Now we consider the case where $\mu \in K_{\alpha-1-\epsilon}$ for some $\epsilon >0$. Let $X$ be a solution of 
	\eqref{eq:singularSDE} and thus also of \eqref{eq:SDE}.
	Due to Proposition~\ref{lem:density} there exists a jointly continuous occupation density $\ell_t^x$ of $X$.
	Thus, it follows that for any $t \geq 0$, as $n \rightarrow \infty$, 
	\begin{equation*}
		\int_\R \ell_t^x \mu_n(x)dx \rightarrow \int_\R \ell_t^x \mu(dx), \quad \mathbb{P}_{x_0}\text{- a.s.,}
	\end{equation*}
	by the weak convergence of $\mu_n$ to $\mu$. This proves that $A_t = \int_\R \ell_t^x \mu(dx)$ because $\ell_t^x$
	is an occupation density.
	Furthermore, it follows  for almost all $\omega$ and all $\upsilon \in S_{K_0}$ that
	\begin{equation*}
		\int_\R \int_0^t \varphi_n(X_s-x)ds\upsilon(dx)=	\int_0^t \upsilon_n(X_s)ds = \int_\R  \ell_t^x \upsilon_n(x)dx.
	\end{equation*}
	As $n \rightarrow \infty$, the right-hand side converges for all $\omega$ such that $\ell_t^x$ is jointly continuous. 
	By Proposition~\ref{lem: convergence} it follows
	that for all $t \geq 0$ we have
	\begin{equation*}
		\int_\R \gamma_t^x \upsilon(dx) = \int_\R \ell_t^x \upsilon(dx)  \quad\bbP_{x_0}{-a.s.} \quad  \forall \upsilon \in S_{K_0}.
	\end{equation*} 
	Fix arbitrary $t\geq 0$ and consider an arbitrary measure $\upsilon \in S_{K_0}$. Then it follows that 
	\begin{equation*}
		\int_\R \gamma_t^x f(x)\upsilon(dx) = \int_\R \ell_t^x f(x)\upsilon(dx)  \quad\bbP_{x_0}{-a.s.,} 
	\end{equation*}
	for all $f \in \mathcal{C}_c$ since $f(x)\upsilon(dx) \in S_{K_0}$. 
	Consider a countable dense subset 
	\begin{equation*}
		(f_m)_{m \in \mathbb{N}} \subset \mathcal{C}_c.
	\end{equation*}
	For each $m \in \mathbb{N}$ there exists a measurable set $N_m \subset \Omega$ with $\bbP(N_m) = 0$ such that for all $\omega \in N_m^c$
	\begin{equation*}
		\int_\R \gamma_t^x f_m(x)\upsilon(dx) = \int_\R \ell_t^x f_m(x)\upsilon(dx).
	\end{equation*}
	In particular 
	\begin{equation*}
		\int_\R \gamma_t^x f_m(x)\upsilon(dx) = \int_\R \ell_t^x f_m(x)\upsilon(dx), \quad \forall m \in \mathbb{N},  \quad \forall \omega \in N^c,
	\end{equation*}
	where $N = \cup_{m \in \mathbb{N}} N_m$ with $\bbP(N) =0$. This implies that $\gamma_x^t = \ell_x^t$, $\bbP_{x_0}$-a.s.
	and  for $\upsilon$-almost every $x$. Since $\upsilon \in S_{K_0}$ was arbitrary this yields 
	\begin{equation*}
		\gamma_x^t = \ell_x^t,  \quad \mathbb{P}_{x_0}\text{-a.e.}, \quad \upsilon\text{-a.s. } \forall \upsilon \in S_{K_0}.
	\end{equation*}
	Consider the set 
	\begin{equation*}
		\tilde{N}=\{x\in \R \colon \gamma_t^x \neq \ell_t^x \text{ with positive probability}\} 
	\end{equation*}
	and suppose 
	$\mathrm{Cap}_{(\alpha-1)/2}(\tilde{N})>0$. By Theorem~\ref{th:capacity_condition} this implies that there exists $\upsilon \in S_{K_0}$ such that $\upsilon(\tilde{N})>0$.
	This is a contradiction and thus
	it follows that $\gamma_t^x = \ell_t^x$, $\bbP$-a.s. for all $x\in\R\backslash \tilde{N}$, for some $\tilde{N}$ with $\mathrm{Cap}_ {(\alpha-1)/2}(\tilde{N}) =0$. This finishes the prove of \ref{itm:density}.
\end{proof}

\section{Sharpness of results}\label{sec:sharp}
In this section, we investigate whether Theorem~\ref{th:sde_exist}  is sharp.
Consider the (formal) stochastic equation
\begin{equation}\label{eq:SDE_non_exist}
	X_t =- \int_0^t C_\alpha\sgn(X_s)|X_s|^{1-\alpha}\mathbbm{1}_{X_s\neq 0}ds+ L_t = A_t+L_t, \quad t \geq 0,
\end{equation}
where $C_\alpha$ is a positive constant whose exact value will be fixed at a later point.
Assume that $X$ is a solution to \eqref{eq:SDE_non_exist} in the sense of Definition~\ref{def:sing_sol}, such that $X_t = A_t+L_t$, $t \geq 0$, where 
$A_t$ is a continuous process of finite variation. 
Then it follows by Corollary~\ref{cor:occ_dens_lev} that the local time $\gamma$ as defined in Definition~\ref{def:loc_time} is an occupation density.  
Now, since $X$ has an occupation density $\ell$, we can use the monotone convergence theorem to
rewrite the solution $X$ as 
\begin{equation*}
	X_t = - \int_\R C_\alpha \sgn(x) |x|^{1-\alpha}\mathbbm{1}_{x\neq 0} \ell_t^xdx + L_t = - \int_\R C_\alpha \sgn(x) |x|^{1-\alpha}\ell_t^xdx + L_t, \quad t \geq 0.
\end{equation*}
Now, we use a similar argument as in
\cite[Proposition 4.4]{Salminen07}.  That is, we use that the following identity holds for  all $x\in \R$ outside of set of Lebesgue measure zero
\begin{equation}\label{eq:Tan_counter}
	v(X_t-x)- v(x) = N_t^x + \int_0^s v'(X_s-x)dA_s + \ell_t^x, \quad t \geq 0
\end{equation}
and we integrate \eqref{eq:Tan_counter} with respect to $\frac{dz}{|z|^{\alpha-\zeta}}$ for 
some fixed $\alpha-1<\zeta < \alpha/2$. 
Using the occupation time formula  and that $v(x) = c_2(\alpha)|x|^{\alpha-1}$ for all $x \in \R$ yields
\begin{equation}\label{eq:integrated_loc_time}
	\begin{split}
		&\int_\R \big(|X_t-z|^{\alpha-1}-|z|^{\alpha-1}\big)\frac{dz}{|z|^{\alpha-\zeta}}\\
		&= c_2(\alpha)^{-1}\Big(\int_\R N_t^z \frac{dz}{|z|^{\alpha-\zeta}} +\int_0^t |X_s|^{\zeta-\alpha}ds +\int_\R \int_0^t v'(X_s-z)dA_s\frac{dz}{|z|^{\alpha-\zeta}}\Big),
	\end{split}
\end{equation}
for all $t \geq 0$.
Note that $v'(x) = \tilde{c}(\alpha)\sgn(x)|x|^{\alpha-2}$ for all $x\in \R$, where $\tilde{c}(\alpha)= (\alpha-1)c_2(\alpha)$ is a fixed constant. Using this,
the last term on the right-hand side of \eqref{eq:integrated_loc_time} equals to
\begin{equation}\label{eq:right-hand}
	\begin{split}
		&\int_\R \int_0^t v'(X_s-z)dA_s\frac{dz}{|z|^{\alpha-\zeta}}\\
		& =-\tilde{c}(\alpha)C_\alpha\int_\R \int_\R \sgn(x-z)|x-z|^{\alpha-2}\sgn(x)|x|^{1-\alpha}\ell_t^xdx \frac{dz}{|z|^{\alpha-\zeta}} \\
		& = -\tilde{c}(\alpha)C_\alpha\int_\R \sgn(u-1)\sgn(u)|u-1|^{\alpha-2}|u|^{1-\alpha} \int_\R |z|^{\zeta-\alpha} \ell_t^{zu}ds du\\
		& =-\tilde{c}(\alpha)C_\alpha \int_\R \sgn(u-1)\sgn(u)|u-1|^{\alpha-2}|u|^{-\zeta} \int_\R |w|^{\zeta-\alpha} \ell_t^wdwdu\\
		& =-\tilde{c}(\alpha)C_\alpha \int_\R \sgn(u-1)\sgn(u)|u-1|^{\alpha-2}|u|^{-\zeta}du \int_0^t |X_s|^{\zeta-\alpha}ds,
	\end{split}
\end{equation}
where we used the definition of $A$ in the first equality, substitution and Fubini's theorem in the second and third equality.
Now, the idea is to choose $C_\alpha$ in such a manner that the last two terms on the right-hand side of \eqref{eq:integrated_loc_time}
cancel out. To this end we need to prove
\begin{equation}\label{eq:constant_choose}
	0<\int_\R \sgn(u-1)\sgn(u)|u-1|^{\alpha-2}|u|^{-\zeta}du <\infty.
\end{equation}
The second inequality follows from the fact that $-1<\alpha-2$, $\zeta<1$
and $\alpha-2-\zeta< -1$. 
Now, let us derive the first inequality in \eqref{eq:constant_choose}:
\begin{align*}
	&\int_\R \sgn(u-1)\sgn(u)|u-1|^{\alpha-2}|u|^{-\zeta}du \\
	& =\int_1^\infty |u-1|^{\alpha-2}|u|^{-\zeta}du + \int_{-\infty}^0|u-1|^{\alpha-2}|u|^{-\zeta}du
	- \int_0^1|u-1|^{\alpha-2}|u|^{-\zeta}du \\
	& = \int_1^\infty |u-1|^{\alpha-2}|u|^{-\gamma}du + \int_1^\infty|u|^{\alpha-2}|u-1|^{-\zeta}du
	- \int_0^1|u-1|^{\alpha-2}|u|^{-\zeta}du \\
	& = \int_0^1|1-u|^{\alpha-2}|u|^{\zeta-\alpha}du +\int_0^1|u|^{\zeta-\alpha}|1-u|^{-\zeta}du -\int_0^1|u-1|^{\alpha-2}|u|^{-\zeta}du \\
	&= \int_{0}^1|1-u|^{\alpha-2}|u|^{\zeta-\alpha}(1-|u|^{\alpha-2\zeta})du + \int_0^1|u|^{\zeta-\alpha}|1-u|^{-\zeta}du >0,
\end{align*}
where we used the substitution $w= u^{-1}$ in the third equality and $\alpha-2\zeta>0$ in the last inequality.
Thus we can choose
\begin{equation}\label{eq:calpha_choice}
	C_\alpha = \tilde{c}(\alpha)^{-1}\Big(\int_\R \sgn(u-1)\sgn(u)|u-1|^{\alpha-2}|u|^{-\gamma}du\Big)^{-1}.
\end{equation}
By a scaling argument the left hand side of \eqref{eq:integrated_loc_time} is equal to 
\begin{equation}\label{eq:left-hand}
	\hat{c}(\alpha)|X_t|^\zeta
\end{equation}
for some constant $\hat{c}(\alpha)>0$.
Thus with $C_\alpha$ as in \eqref{eq:calpha_choice} we use \eqref{eq:left-hand}, \eqref{eq:right-hand} and \eqref{eq:integrated_loc_time} to get
\begin{equation}\label{eq:nonexist}
	|X_t|^\zeta =c(\alpha)M^\zeta_t,
	\quad t \geq 0,
\end{equation}
where $M^\zeta_t = \int_\R N_t^z \frac{dz}{|z|^{\alpha-\gamma}} $ is a local martingale and $c(\alpha)>0$ is a constant depending on $\alpha$.
It thus  follows that $(|X_t|^\gamma)_{t \geq 0}$ is a 
non-negative, local martingale with $\bbE[|X_0|^\gamma] = 0$. This implies that $X_t \equiv 0$
a.s. for all $t \geq 0$ and can not be a solution to  \eqref{eq:nonexist}.\\
The discussion demonstrates that \cite[Theorem 1.2]{Chen16} is sharp, since
\begin{equation*}
	\lim_{r \rightarrow 0}\sup_{x\in \R}\int_{B(x,r)} |x-y|^{\eta -1}|y|^{1-\alpha}dy = 0, \forall \eta >\alpha-1,
\end{equation*}
but 
\begin{equation*}
	\lim_{r \rightarrow 0}\sup_{x\in \R}\int_{B(x,r)} |x-y|^{\alpha-2}|y|^{1-\alpha}dy = \infty,
\end{equation*}
and \eqref{eq:nonexist} does not admit a solution. Since Theorem~\ref{th:sde_exist} is formulated in terms of finite signed measures, the example \eqref{eq:nonexist} does not immediately fit in its setting. This is due to the fact that 
\begin{equation*}
	\int_0^\infty|x|^{1-\alpha}dx = \infty, \quad 	\int_{-\infty}^0|x|^{1-\alpha}dx = -\infty,
\end{equation*}
which means, that $\sgn(x)|x|^{1-\alpha}\mathbbm{1}_{x \neq 0}dx$ is not a well defined signed measure. 
However, if we consider
\begin{equation*}
	\mu(dx) =-C_\alpha\mathbbm{1}_{x\in[-1,1]} \sgn(x)|x|^{1-\alpha}\mathbbm{1}_{x \neq 0}dx, 
\end{equation*} 
and 
\begin{equation}\label{eq:nonexists2}
	Y_t = \int_0^t \mu(Y_s)ds + L_t, \quad t \geq 0,
\end{equation}
where we again use the notion of solution according to Definition~\ref{def:sing_sol}, then $Y = \tilde{A} +L$ and $X =A+L$
coincide for all $t <T$, where $T = \inf\{t \geq 0 \colon |Y_t|\geq 1\}$.  Note that $\tilde{A}$ is almost surely continuous and thus $T>0$ almost surely.
Since there exists no solution to \eqref{eq:nonexist}, this implies that there also exists no solution to \eqref{eq:nonexists2} on $[0,T]$.
Thus, we can conclude, that Theorem 
~\ref{th:sde_exist} is sharp as well. 
\appendix
\addtocontents{toc}{\protect\setcounter{tocdepth}{0}}
\section{}
\subsection{Proof of Lemma~\ref{lem:KatoBesov}}\label{sec:appendic_bk}

\begin{proof}[Proof of Lemma~\ref{lem:KatoBesov}]
	The following characterization of Kato class will be useful throughout this proof. For a signed Radon measure $\nu$ on $\R$, define
	\begin{equation*}
		N_\nu^\eta(t) := \sup_{x \in \R}\int_0^t\int_R\tau^{\frac{\eta}{2}-1}h(\tau,x,y)|\nu|(dy)d\tau,
	\end{equation*}
	where $h(\tau,x,y) = \frac{1}{\sqrt{2 \pi \tau}}\ee^{-\frac{|x-y|^2}{2\tau}}$ for $\tau \geq 0 $ and $x,y \in \R$.
	Then $\nu \in K_{\eta}$, $\eta \in (0,1)$ if and only if 
	\begin{equation}\label{eq:equiv_Kato}
		\lim_{t \rightarrow 0}N_\nu^{\eta}(t) = 0,
	\end{equation}
	which follows by the proof of \cite[Proposition 2.3]{Kim06}.
	Thus in order to prove \ref{itm:fin_m_bes} it is enough to prove \eqref{eq:equiv_Kato}.  Recall the characterization \eqref{eq:besov_char}.
	We see that for a  nonnegative measure $\nu \in \mathcal{C}^{-s}$ with $ s \leq \eta-\epsilon$ for some $\epsilon >0$
	the following holds for $t \leq 1$ that
	\begin{equation}\label{eq:n_nu_est}
		\begin{split}
			N_\nu^{\eta} (t)&= \sup_{x \in \R}\int_0^t\int_R\tau^{\frac{\eta}{2}-1}h(\tau,x,y)|\nu|(dy)d\tau\\
			& \leq \int_0^t\tau^{\frac{\eta}{2}-1}\sup_{x \in \R}\int_Rh(\tau,x,y)\nu(dy)d\tau\\
			& \leq \int_0^t \tau^{-1+\epsilon/2}d\tau\sup_{\tau\in (0, 1] }\|\tau^{\frac{\eta-\epsilon}{2}}e^{\tau\Delta}\nu\|_{L^\infty}.
		\end{split}
	\end{equation}
	This demonstrates $ N^\eta_\nu(t) \rightarrow 0$ as $t \rightarrow 0$ and thus we proved \ref{itm:fin_m_bes}. \\
	Next, we prove inclusion \ref{itm:kato_bes_non}.
	We see by using \eqref{eq:besov_char} again, that
	for a  measure $\nu \in K_{\eta}^f$ it holds for some fixed $R >2$ that
	\begin{equation*}
		\begin{split}
			&\||\nu|\|_{\mathcal{C}^{-s}}\\
			&\leq C\sup_{t \in (0,1]}t^{\frac{s}{2}}\sup_{x \in \R}\int_\R h(t,x,y)|\nu|(dy)dt\\
			& \leq C(\nu, R) \Big(1+  \sup_{t \in (0,1])} t^{\frac{s}{2}}\sup_{x \in \R}\int_{|x-y| \in B(0,R)} t^{-1/2}\ee^{-\frac{|x-y|^2}{4t}} |\nu|(dy) \Big)\\
			& \leq C(\nu , R)\Big(1 + \sup_{t \in (0,1]} t^{\frac{s}{2}}\sup_{x \in \R}\int_{|x-y| \in B(0,R)} t^{-1/2}\ee^{-\frac{|x-y|^2}{4t}}|x-y|^{1-\eta}|x-y|^{-1+\eta} |\nu|(dy)\Big).
		\end{split}
	\end{equation*}
	We note that 
	\begin{equation*}
		\sup_{z \in B(0,R)}\ee^{-\frac{|z|^2}{4t}}|z|^{1-\eta} \leq C(\nu, R)t^{\frac{1-\eta}{2}},
	\end{equation*}
	since the map $z \mapsto \ee^{-\frac{|z|^2}{4t}}|z|^{1-\eta}$ with $t \in (0,1]$ attains its maximum in the interval $[0,R]$ at $z = \sqrt{2t(1-\eta)}$. Thus
	\begin{equation}\label{eq:Kato_totvar_in _Bes}
		\||\nu|\|_{\mathcal{C}^{-s}} \leq C(\nu, R)\Big( 1 + M_\nu^\eta(R) \Big).
	\end{equation}
	This means that 
	$|\nu| \in \mathcal{C}^{-s}$ and thus \ref{itm:kato_bes_non} holds. 
\end{proof}
\subsection{Proof of Lemma~\ref{lem:conv_A} and Lemma~\ref{lem:deriv_u_lam_conv_append}}\label{appendix}
\begin{proof}[Proof of Lemma~\ref{lem:conv_A}]
	The proof follows  the proof of \cite[Proposition 5.6]{Kim14} line by line, we present it here for the sake of completeness. 
	Denote 
	\begin{equation*}
		B_t^n = \int_0^t \ee^{-\lambda s}dA_s^\nu-\int_0^t \ee^{-\lambda s}\nu_n(X_s)ds, \quad t \geq 0.
	\end{equation*}
	By Lemma~\ref{lem:pot_conv} it follows that $\lim_{n \rightarrow\infty}\sup_{x\in \R }\big|\bbE_x [B^n_\infty]\big| = 0$ and
	thus it follows by \cite[Lemma 3.10]{Bass03} that $\lim_{n \rightarrow \infty}\sup_{x\in \R}\bbE_x\big[\sup_{t \geq 0}|B_t^n|^2\big] =0$.
	This yields
	\begin{align*}
		&\lim_{n \rightarrow \infty}\sup_{x\in \R}\bbE\Big[\sup_{0\leq s\leq t}\Big|A_s^\nu- \int_0^s\nu_n(X_\tau)d\tau\Big|^2\Big] \\
		&= \lim_{n \rightarrow \infty}\sup_{x\in \R}\bbE\Big[\sup_{0\leq s\leq t}\Big|\int_0^s\ee^{\lambda \tau}dB^n_\tau\Big|^2\Big] \\
		&= \lim_{n \rightarrow \infty}\sup_{x\in \R} \bbE_x\Big[\sup_{0 \leq s \leq t}\Big|\ee^{\lambda t}B^n_s -\int_0^s \lambda \ee^{\lambda \tau}B^n_\tau d\tau\Big|^2\Big] = 0,
	\end{align*}
	by integration by parts. This finishes the proof.
\end{proof}

\begin{proof}[Proof of Lemma~\ref{lem:deriv_u_lam_conv_append}]
	The proof uses similar ideas as the proof of \cite[Lemma 3.5]{Kim14}. The main differences are that we need to consider $\lambda >0$, because the zero potential does not exists in dimension $d = 1$, and the fact that we consider finite measures, whereas \cite[Lemma 3.5]{Kim14} is restricted to compactly supported measures. \\ 
	Let $\epsilon > 0$ be arbitrary and $m$ sufficiently large. Then  there exist $r>0$ sufficiently small and $R > 0$ sufficiently large 
	such that $R >r> 2^{-m}$ and
	\begin{align*}
		&\sup_{x\in \R} \Big|\frac{\partial}{\partial x}U_\lambda \nu_m(x)  - \frac{\partial}{\partial x}U_\lambda \nu(x) \Big| \\
		& \leq \sup_{x \in \R} \Big| \int_{B(x,r)} \frac{\partial}{\partial x}u_\lambda(x,y) \nu_m(dy)\Big| \\
		& \quad + \sup_{x \in \R} \Big| \int_{B(x,r)} \frac{\partial}{\partial x}u_\lambda(x,y) \nu(dy)\Big| \\
		& \quad + \sup_{x \in \R} \Big| \int_{B(x,R)^c} \frac{\partial}{\partial x}u_\lambda(x,y) \nu_m(dy)\Big| \\
		&\quad + \sup_{x \in \R} \Big| \int_{B(x,R)^c} \frac{\partial}{\partial x}u_\lambda(x,y) \nu(dy)\Big| \\
		& \quad +\sup_{x \in \R }\Big|\int_{B(x,R)\backslash B(x,r)}\frac{\partial}{\partial x}u_\lambda(x,y) (\nu_m-\nu)(dy)\Big| \\
		& =: I_1+ I_2+I_3+I_4 +I_5.
		%& \leq \frac{\epsilon}{4}+ \frac{\epsilon}{4}+\frac{\epsilon}{4}+\frac{\epsilon}{4} + \sup_{x \in \R }\Big|\int_{B(x,R)\backslash B(x,r)}\frac{\partial}{\partial x}u_\lambda(x,y) (\nu_m-\nu)(dy)\Big|.
	\end{align*}
	Hereby we used the Leibniz rule to exchange the derivative and the integral since $\frac{\partial}{\partial x}u_\lambda(x,y)$ is integrable with respect to $\nu$.
	Since $\nu$ is in $K_{\alpha-1}$ it follows by using \eqref{eq:lamPot_grad}  and \cite[Lemma 2.5]{Kim14} that 
	\begin{equation*}
		I_1 \leq c(\alpha)M^{\alpha-1}_\nu(r) \leq \epsilon/4
	\end{equation*}
	and hence $I_2 \leq \epsilon/4$. Again, using \eqref{eq:lamPot_grad} and the fact that $\nu$ is finite and the fact that $|\nu_n|(\R) \leq 2 |\nu|(\R)$ implies
	\begin{equation*}
		I_3 \leq c(\alpha,\mu) \lambda^{-2}R^{-2-\alpha} \leq \epsilon/4
	\end{equation*}
	and thus also $I_4 \leq \epsilon /4$.
	This means it remains to show the convergence to zero of $I_5$. 
	\begin{align*}
		&\sup_{x \in \R }\Big|\int_{B(x,R)\backslash B(x,r)}\frac{\partial}{\partial x}u_\lambda(x,y) (\nu_m-\nu)(dy)\Big|\\
		& =  \sup_{x \in \R }\Big|\int_{B(x,R)\backslash B(x,r)}\frac{\partial}{\partial x}u_\lambda(x,y)\int_\R \varphi_m(y-z) \nu(dz) dy\\
		& \qquad-  \int_{B(x,R)\backslash B(x,r)}\frac{\partial}{\partial x}u_\lambda(x,y) \nu(dy) \Big| \\
		& =\sup_{x \in \R }\Big| \ \int_{B(x,R)\backslash B(x,r)} \int_\R\frac{\partial}{\partial x}u_\lambda(x,y) \varphi_m(y-z)dy \nu(dz)- \int_{B(x,R)\backslash B(x,r)}\frac{\partial}{\partial x}u_\lambda(x,z)\nu(dz)\Big|.
	\end{align*}
	Denoting $\mathbbm{1}_{B(x,R)\backslash B(x,r)}(y)\frac{\partial}{\partial x}u_\lambda(x,y) = f_\lambda(x,y)$
	leads to
	\begin{align*}
		&\sup_{x \in \R }\Big|\int_{B(x,R)\backslash B(x,r)}\frac{\partial}{\partial x}u_\lambda(x,y) (\nu_m-\nu)(dy)\Big|\\
		& = \sup_{x\in \R }\Big|\int_{B(x,R)\backslash B(x,r)}\int_\R \frac{\partial}{\partial x}u_\lambda(x,y)\varphi_m(y-z)dy\nu(dz)\\
		&\hspace{1cm}- \int_{B(x,R)\backslash B(x,r)}\frac{\partial}{\partial x}u_\lambda(x,y) \nu(dy)\Big|\\
		& = \sup_{x \in \R }\Big|\int_\R \big(\varphi_m \ast f_\lambda(x, \cdot)\big)(y) - f_\lambda(x,y) \nu(dy)\Big|.
	\end{align*}
	In fact, $u_\lambda(x,y)$ only depends on $|x-y|$ and by a slight abuse of notation
	\begin{equation*}
		f_\lambda(x,y) = \sgn(x-y) \mathbbm{1}_{r \leq |x-y|<R}u_\lambda'(|x-y|)
	\end{equation*}
	and
	\begin{equation*}
		\big(\varphi_m \ast f_\lambda(x, \cdot)\big)(y) = \int_\R \varphi_m(y-z) \sgn(x-z) \mathbbm{1}_{r \leq |x-z|<R} u_\lambda'(|x-z|)dz.	
	\end{equation*}
	Using the above, we obtain
	\begin{align*}
		&\sup_{x \in \R }\Big|\int_\R \big(\varphi_m \ast f_\lambda(x,\cdot)\big)(y) - f_\lambda(x,y) \nu(dy)\Big| \\
		& = \sup_{x\in \R}\Big|\int_\R \int_\R \varphi_m(z)\big(\sgn(x-y+z)\mathbbm{1}_{r\leq |x-y+z| < R}u_\lambda'(|x-y+z|) \\
		& \hspace{6cm} -\sgn(x-y)\mathbbm{1}_{r \leq |x-y|<R}u_\lambda'(|x-y|)\big)dz\nu(dy)\Big| \\
		& \leq \sup_{x\in \R}\Big|\int_\R \int_\R \varphi_m(z)\big(\sgn(x-y+z)\mathbbm{1}_{r \leq |x-y| }\mathbbm{1}_{r\leq |x-y+z| < R}u_\lambda'(|x-y+z|) \\
		& \hspace{6cm} -\sgn(x-y)\mathbbm{1}_{r \leq |x-y|<R}u_\lambda'(|x-y|)\big)dz\nu(dy)\Big| \\
		& \quad + \sup_{x\in \R}\Big|\int_\R \int_\R \varphi_m(z)\big(\sgn(x-y+z)\mathbbm{1}_{ |x-y| < r}\mathbbm{1}_{r\leq |x-y+z| < R}u_\lambda'(|x-y+z|)dz \nu(dy)\Big|\\
		& = J_1 +J_2\\
	\end{align*}
	First we estimate the second term. Note that 
	\begin{equation*}
		D_{z,x} = \{ y \colon |x-y|< r \leq |x-y+z| \} \subseteq  B(x-r, 2|r|)\cup B(x+r,2|r|), \text{ for } |z|< 2^{-m}.
	\end{equation*}
	Thus, using \eqref{eq:lamPot_grad} , we arrive at
	\begin{align*}
		&J_2 \leq c\sup_{x\in \R }\int_\R\varphi_m(z) \int_\R \mathbbm{1}_{ |x-y| < r}\mathbbm{1}_{r\leq |x-y+z| < R}|x-z-y|^{\alpha-2}|\nu|(dy)dz\\
		& \leq c \sup_{x\in \R}\int_\R\varphi_m(z) \int_{D_{z,x}}|x-y+z|^{\alpha-2}|\nu|(dy)dz \leq c M_\nu(4r).
	\end{align*}
	Now we turn to the first term.  Note that 
	\begin{equation}\label{eq:signs}
		\sgn(x-y) = \sgn(x-y+z), \quad \text{ for }r \leq |x-y|, |x-y-z|, |z| \leq 2^{-m}.
	\end{equation}
	Furthermore, note that
	\begin{align}\label{eq:Dtilde}
		\{y \colon |x-y+z| < r \leq |x-y|\} \subset B(x-r,2|r|)\cup B(x+r,2|r|).
	\end{align}
	By using \eqref{eq:signs} and \eqref{eq:Dtilde} we arrive at 
	\begin{align*}
		&J_1 \leq \sup_{x\in \R}\Big|\int_\R \int_\R \varphi_m(z)\big(\mathbbm{1}_{r \leq |x-y| }\mathbbm{1}_{r\leq |x-y+z| < R}u_\lambda'(|x-y+z|) \\
		& \hspace{6cm} -\mathbbm{1}_{r \leq |x-y|<R}u_\lambda'(|x-y|)\big)dz\nu(dy)\Big| \\
		&\leq \sup_{x\in \R}\Big|\int_\R \int_\R \varphi_m(z)\big(\mathbbm{1}_{r \leq |x-y| <R}\mathbbm{1}_{r\leq |x-y+z| < R}u_\lambda'(|x-y+z|) \\
		& \hspace{6cm} -\mathbbm{1}_{r \leq |x-y|<R}\mathbbm{1}_{r\leq |x-y+z| < R}u_\lambda'(|x-y|)\big)dz\nu(dy)\Big|\\
		& + \sup_{x\in \R}\int_\R \int_\R \varphi_m(z)\mathbbm{1}_{|x-y|>R}|u'_\lambda(|x-y+z|)||\nu(dy)|dz\\
		&+ \sup_{x\in \R}\int_\R \int_\R \varphi_m(z)\mathbbm{1}_{|x-y+z|>R}|u'_\lambda(|x-y|)||\nu(dy)|dz\\
		&+\sup_{x\in \R}\int_\R \int_\R \varphi_m(z)\mathbbm{1}_{r\leq |x-y|<R}\mathbbm{1}_{|x-y+z|<r}|u'_\lambda(|x-y+z|)||\nu(dy)|dz\\
		& \leq c \Big( \sup_{w \in B(0,R)\backslash B(0,r)}\sup_{|z|\leq 2^{-m}}|u_\lambda'(w+z)-u_\lambda'(w)| +M_\nu(4r)+ (R-2^{-m})^{\alpha-2}  \Big).
	\end{align*}
	Due to the uniform continuity of $\frac{\partial}{\partial x}u_\lambda(x)$ on compacts the convergence of $J_1$ to zero follows and thus the convergence of $I_5$ to zero. This completes the proof.
\end{proof}

\subsection{Useful Estimates}\label{append3}
\begin{lemm}\label{lem:grad_est_pot_dens}
	Let $\mu \in K_{\alpha-1}$ be a finite measure and $\mu_m$, $m \in \mathbbm{N}$ its mollification. 
	It holds for all $x,y \in \R$ that 
	\begin{equation*}
		\Big| u_\lambda \ast \big(\frac{\partial}{\partial x}u_\lambda \mu\big)^{\ast k-1} \Big|(x,y) \leq C(\lambda, \mu)^{k-1}|u_\lambda(x,y)|,
	\end{equation*}
	and 
	\begin{equation*}
		\Big| u_\lambda \ast \big(\frac{\partial}{\partial x}u_\lambda \mu_m\big)^{\ast k-1} \Big|(x,y) \leq C(\lambda, \mu)^{k-1}|u_\lambda(x,y)|
	\end{equation*}
	where $C(\lambda,\mu) \rightarrow 0$ as $\lambda \rightarrow \infty$.
\end{lemm}
\begin{proof}
	Following the proof of \cite[Lemma 17]{Bogdan07} line by line  and using \cite[Lemma 2.5]{Kim14} yields
	\begin{equation*}
		\int_\R u_\lambda(x,z) \big|\frac{\partial}{\partial z} u_\lambda(z,y)\big||\mu_m|(dy)\leq c\int_\R u_\lambda(x,z) \big|\frac{\partial}{\partial z} u_\lambda(z,y)\big||\mu|(dy) \leq C(\lambda,\mu) u_\lambda(x,y),
	\end{equation*}
	for $x,y \in \R$, where $C(\lambda,\mu) \rightarrow 0 $ as $\lambda \rightarrow \infty$. Thus, induction yields the wanted result.
\end{proof}
\begin{lemm} \label{lem:pot_grad_upper_est}
	Let $\mu \in K_{\alpha-1}$ be a finite measure and $\lambda > 0$.
	Then there exists a constant $C(\lambda,\mu)>0$ such that 
	\begin{align}\label{eq:grad_pot_meas_int}
		\Big|\int_\R (\frac{\partial}{\partial x} u_\lambda\mu)^{\ast k}(dz,x)\Big| \leq C(\lambda,\mu)^k, \quad x\in \R, k \in \mathbbm{N}
	\end{align} 
	and $C(\lambda,\mu) \rightarrow 0$ for $\lambda \rightarrow \infty$. 
\end{lemm}
\begin{proof}
	The proof follows by induction. First note that
	\begin{equation*}
		\Big|\int_\R  \frac{\partial}{\partial x}u_\lambda(z,x)\mu(dz)\Big|\leq  \int_\R  \big| \frac{\partial}{\partial x}u_\lambda(z,x)| |\mu|(dz).
	\end{equation*}
	Using the fact that $\mu\in K_{\alpha-1}$ and \eqref{eq:lamPot_grad} yield that there exists a constant $C(\lambda,\mu)$ such that $C(\lambda, \mu) \rightarrow 0$ as $\lambda \rightarrow \infty$ and 
	\begin{equation*}
		\int_\R  \big| \frac{\partial}{\partial x}u_\lambda(z,x)| |\mu|(dz) < C(\lambda, \mu), \quad \forall x\in \R.
	\end{equation*}
	Let $k \geq 2$ and assume that \eqref{eq:grad_pot_meas_int} holds for $k-1$, then
	\begin{equation*}
		\begin{split}
			&\Big|\int_\R (\frac{\partial}{\partial x} u_\lambda\mu)^{\ast k}(dz,x)\Big| =\Big| \int_\R \int_\R \frac{\partial}{\partial x}u_\lambda(z,z_1)(\frac{\partial}{\partial x} u_\lambda\mu)^{\ast k-1}(dz_1,x)\mu(dz)\Big|\\
			&\leq C(\lambda, \mu)^{k-1} \int_\R  \big| \frac{\partial}{\partial x}u_\lambda(dz_1,x)| |\mu|(dz) \leq C(\lambda, \mu)^k.
		\end{split}
	\end{equation*}
	This finishes the proof.
\end{proof}
\begin{lemm}\label{lemm:help_bounded}
	For all $\lambda > 0$ the map
	\begin{equation*}
		x\mapsto \int_0^\infty \frac{\sin(\xi) \xi}{\lambda |x|^\alpha+\xi^\alpha}d\xi
	\end{equation*}
	is bounded by a constant $C$ that is independent of $\lambda$. 
\end{lemm}
\begin{proof}
	First, consider $x \neq 0$. Integration by parts leads to 
	\begin{align*}
		&\Big|\int_0^\infty \sin(\xi) \frac{\xi}{\lambda|x|^\alpha+ \xi^\alpha}d\xi\Big|\\
		& \leq \Big| \int_0^{\pi/2}\sin(\xi) \frac{\xi}{\lambda|x|^\alpha+ \xi^\alpha}d\xi\Big|+\Big|\int_{\pi/2}^\infty \sin(\xi) \frac{\xi}{\lambda|x|^\alpha+ \xi^\alpha}d\xi\Big| \\
		&\leq \int_0^{\pi/2} \xi^{1-\alpha}d\xi +\Big|\int_{ \pi/2}^\infty \cos(\xi) \frac{\partial}{\partial \xi}\frac{\xi}{\lambda |x|^ \alpha+ \xi^\alpha}d\xi\Big|\\
		&\leq \int_0^{\pi/2} \xi^{1-\alpha}d\xi + \int_{\pi/2}^\infty\Big| \frac{\lambda |x|^\alpha - (\alpha - 1) \xi^\alpha}{(\lambda|x|^\alpha + |\xi|^\alpha)^2}\Big| d\xi \\
		& \leq \int_0^{\pi/2} \xi^{1-\alpha}d\xi + \int_{\pi/2}^\infty \frac{\lambda |x|^\alpha}{\max(\lambda|x|^\alpha, |\xi|^\alpha)^2}d\xi + (\alpha-1)\int_{\pi/2}^\infty \xi^{-\alpha}d\xi.
	\end{align*}
	In the case where $(\pi/2)^\alpha<\lambda|x|^\alpha$ we get 
	\begin{align*}
		&\int_{\pi/2}^\infty \frac{\lambda |x|^\alpha}{\max(\lambda|x|^\alpha, |\xi|^\alpha)^2}d\xi \\
		& \leq \int_{\pi/2}^{\lambda^{1/\alpha}|x|}\lambda^{-1}|x|^{-\alpha} d\xi + \int_{\lambda^{1/\alpha}|x|}^\infty \xi^{-\alpha} d\xi \\
		& \leq 1+ \int_1^\infty \xi^{-\alpha}d\xi \leq K_1
	\end{align*}
	and in the case where $(\pi/2)^\alpha\geq \lambda|x|^\alpha$ it holds that 
	\begin{align*}
		&\int_{\pi/2}^\infty \frac{\lambda |x|^\alpha}{\max(\lambda|x|^\alpha, |\xi|^\alpha)^2}d\xi \\
		& \leq \int_{\pi/2}^\infty \xi^{-\alpha}d\xi \leq K_2,
	\end{align*} 
	where $K_1, K_2 >0$ are constants.
	This implies the existence of a constant $C>0$ such that
	\begin{align*}
		&\Big|\int_0^\infty \sin(\xi) \frac{\xi}{\lambda|x|^\alpha+ \xi^\alpha}d\xi\Big| \leq C
	\end{align*}
	for all $x \neq 0$.
	In the case $x=0$ the map evaluates to
	\begin{equation*}
		\int_0^\infty \frac{\sin(\xi)}{\xi^{\alpha-1}}d\xi < \infty,
	\end{equation*}
	and the bound on the integral follows from Dirichlet's test. 
	This proves the claim.
\end{proof}

\bibliographystyle{abbrv}
\bibliography{bibliography.bib}

\end{document}